\documentclass{ut-thesis}

\usepackage{graphicx,amsmath,amsthm,amscd,amssymb,pifont}
\degree{Doctor of Philosophy}
\department{Mathematics}
\gradyear{2012}
\author{Michael Bailey}
\title{On the local and global classification of generalized complex structures}

\begin{document}

%% ***   NOTE   ***
%% You should put all of your `\newcommand', `\newenvironment', and
%% `\newtheorem's (in other words, all the global definitions that
%% you will need throughout your thesis) in a separate file and use
%% "\input{filename}" to input it here.
%\usepackage{graphicx,amsmath,amsthm,amscd,amssymb,pifont}

% ----------------------------------------------------------------
\vfuzz2pt % Don't report over-full v-boxes if over-edge is small
\hfuzz2pt % Don't report over-full h-boxes if over-edge is small
% THEOREMS -------------------------------------------------------
\newtheorem{thm}{Theorem}[section]
\newtheorem{cor}[thm]{Corollary}
\newtheorem{lem}[thm]{Lemma}
\newtheorem{prop}[thm]{Proposition}
\newtheorem*{main thm}{Main Theorem}
\newtheorem*{main lem}{Main Lemma}
\theoremstyle{definition}
\newtheorem{defn}[thm]{Definition}
\newtheorem{notn}[thm]{Notation}
\newtheorem{example}[thm]{Example}
\newtheorem{exercise}[thm]{Exercise}
\newtheorem{question}[thm]{Question}
\newtheorem{answer}[thm]{Answer}
\theoremstyle{remark}
\newtheorem{rem}[thm]{Remark}
\numberwithin{equation}{section}
% MATH -----------------------------------------------------------
\newcommand{\norm}[1]{\Lieeft\Vert#1\right\Vert} \newcommand{\abs}[1]{\Lieeft\vert#1\right\vert}
\newcommand{\set}[1]{\Lieeft\{#1\right\}}
\newcommand{\Hom}{\textnormal{Hom}}
\newcommand{\Real}{\mathbb R}
\newcommand{\D}{\mathbb D}
\newcommand{\R}{\mathbb R}
\newcommand{\C}{\mathbb C} 
\newcommand{\Z}{\mathbb Z}
\newcommand{\N}{\mathbb N}
\newcommand{\Tc}{\mathbb T} 
\newcommand{\Lie}{\mathcal{L}} 
\newcommand{\J}{\mathcal{J}}
\newcommand{\T}{\textsf{T}} 
\newcommand{\I}{\mathcal{I}}
\newcommand{\tens}{\otimes} 
\newcommand{\dsum}{\oplus} 
\newcommand{\x}{\times}
\newcommand{\proj}{\mathbb P} 
\newcommand{\iso}{\simeq} 
\newcommand{\isoto}{\overset\sim\to}
\newcommand{\dt}{\frac{\partial}{\partial t}} 
\newcommand{\eps}{\varepsilon}
\renewcommand{\d}{\textrm{d}}
\renewcommand{\phi}{\varphi} 
\renewcommand{\to}{\longrightarrow}
\newcommand{\oto}[1]{\overset{#1}\to} 
\newcommand{\into}{\hookrightarrow}
\newcommand{\ointo}[1]{\overset{#1}\into} 
\renewcommand{\mapsto}{\longmapsto}
\newcommand{\actson}{\circlearrowright} 
\renewcommand{\i}{\iota} 
\newcommand{\comp}{\circ}
\newcommand{\sr}{\mathcal} 
\newcommand{\A}{\mathcal{A}} 
\renewcommand{\O}{\mathcal{O}}
\newcommand{\B}{\mathcal{B}} 
\newcommand{\vH}{\check{H}} 
\newcommand{\gcb}{\mathfrak{b}}
\newcommand{\cnj}{\overline} 
\newcommand{\gca}{\mathfrak{a}} 
\newcommand{\gcg}{\mathfrak{g}}
\newcommand{\gch}{\mathfrak{h}} 
\newcommand{\gcl}{\mathfrak{l}} 
\newcommand{\gci}{\mathfrak{i}}
\newcommand{\ad}{\textrm{ad}} 
\newcommand{\Ad}{\textrm{Ad}} 
\newcommand{\Aut}{\textrm{Aut}}
\newcommand{\Ann}{\textrm{Ann}} 
\renewcommand{\Re}{\textrm{Re}} 
\newcommand{\del}{\partial}
\newcommand{\Cour}[1]{[\![#1]\!]}
\newcommand{\pair}[1]{\left\langle #1 \right\rangle}
\newcommand{\fl}[1]{\left\lfloor #1 \right\rfloor}
\newcommand{\form}{\pair{\cdot,\cdot}} 
\newcommand{\cupp}{\smile} 
\newcommand{\capp}{\frown}
\newcommand{\moneq}{\overset{m}{\sim}} 
\renewcommand{\ddot}{\bullet} 
\newcommand{\Sym}{\textrm{Sym}}
\renewcommand{\^}{\wedge} 
\renewcommand{\=}{\bar} 
\renewcommand{\epsilon}{\varepsilon}
\newcommand{\hide}[1]{} 
\newcommand{\lra}{\longrightarrow}
\renewcommand{\H}{\sr H}
\newcommand{\incirc}[1]{\mbox{\textcircled{#1}}}
\newcommand{\Id}{\textrm{Id}}
\newcommand{\LP}{\mathfrak{L}}
\newcommand{\im}{\textrm{im}}

\newcommand{\article}[1] {}
\newcommand{\thesis}[1] {#1}

%% This sets the page style and numbering for preliminary sections.
\begin{preliminary}

%% This generates the title page from the information given above.
\maketitle

%% There should be NOTHING between the title page and abstract.

%% This generates the abstract page, with the line spacing adjusted
%% according to SGS guidelines.
\begin{abstract}
%% ***   Put your Abstract here.   ***
%% (At most 150 words for M.Sc. or 350 words for Ph.D.)
We study a number of local and global classification problems in \emph{generalized complex geometry}.  Generalized complex geometry is a relatively new type of geometry which has applications to string theory and mirror symmetry.  Symplectic and complex geometry are special cases.

In the first topic, we characterize the local structure of generalized complex manifolds by proving that a generalized complex structure near a complex point arises from a holomorphic Poisson structure.  In the proof we use a smoothed Newton's method along the lines of Nash, Moser and Conn.

In the second topic, we consider whether a given regular Poisson structure and transverse complex structure come from a generalized complex structure.  We give cohomological criteria, and we find some counterexamples and some unexpected examples, including a compact, regular generalized complex manifold for which nearby symplectic leaves are not symplectomorphic.

In the third topic, we consider generalized complex structures with nondegenerate type change; we describe a generalized Calabi-Yau structure induced on the type change locus, and prove a local normal form theorem near this locus.  Finally, in the fourth topic, we give a classification of generalized complex principal bundles satisfying a certain transversality condition; in this case, there is a \emph{generalized flat connection}, and the classification involves a monodromy map to the Courant automorphism group.
\end{abstract}

%% Anything placed between the abstract and table of contents will
%% appear on a separate page since the abstract ends with \newpage
%% and the table of contents starts with \clearpage.

%% This generates a "dedication" section, if needed.
%% (uncomment to have it appear in the document)
%\begin{dedication}
%% ***   Put your Dedication here.   ***
%\end{dedication}

%% The `dedication' and `acknowledgements' sections do not create new
%% pages so if you want the two sections to appear on separate pages,
%% you should put an explicit \newpage between them.

%% This generates an "acknowledgements" section, if needed.
%% (uncomment to have it appear in the document)
\begin{acknowledgements}
I would like to thank my supervisors, Marco Gualtieri and Yael Karshon, for their knowledge, guidance and patience---and especially their constructive criticism; Ida Bulat and the rest of the staff of the Department for helping me avoid various crises; and most of all, unnamed friends and loved ones for taking me seriously and helping make these last few years worthwhile.
\end{acknowledgements}

%% This generates the Table of Contents (on a separate page).
\tableofcontents

%% This generates the List of Tables (on a separate page), if needed.
%% (uncomment to have it appear in the document)
%\listoftables

%% This generates the List of Figures (on a separate page), if needed.
%% (uncomment to have it appear in the document)
%\listoffigures

%% End of the preliminary sections: reset page style and numbering.
\end{preliminary}

%%%%%%%%%%%%%%%%%%%%%%%%%%%%%%%%%%%%%%%%%%%%%%%%%%%%%%%%%%%%%%%%%%%%%%
%%  Put your Chapters here; the easiest way to do this is to keep   %%
%%  each chapter in a separate file and `\include' all the files    %%
%%  right here.  Note that each chapter file should start with the  %%
%%  line "\chapter{ChapterName}".  Note that using `\include'       %%
%%  instead of `\input' makes each chapter start on a new page.     %%
%%%%%%%%%%%%%%%%%%%%%%%%%%%%%%%%%%%%%%%%%%%%%%%%%%%%%%%%%%%%%%%%%%%%%%

%% ***   Include chapter files here.   ***

\chapter{Introduction}

In this thesis we study several problems in the local and global classification of generalized complex manifolds.  Generalized complex geometry is a generalization of both symplectic and complex geometry, introduced by Hitchin \cite{Hitchin2003}, and developed by Gualtieri (we refer to his recent publication \cite{Gualtieri2011} rather than his thesis), Cavalcanti \cite{Cavalcanti} and others.  Its applications include the study of 2-dimensional supersymmetric quantum field theories, which occur in topological string theory, as well as compactifications of string theory with fluxes \cite{Grana}, and the study of mirror symmetry \cite{Ben-Bassat} \cite{CavalcantiGualtieri2011}.

Whereas a complex manifold may be defined by an integrable complex structure on the tangent bundle, a generalized complex manifold is given by an integrable (in some sense) complex structure on an extension of the tangent bundle by the cotangent bundle, called a \emph{Courant algebroid}.  Chapter-by-chapter, we introduce a number of points of view on generalized complex structures as needed---as deformations of complex structures, as complex pure spinor line bundles, as complex structures on abstract Courant algebroids.

The chapters of this work are being prepared for publication separately, and this fact is evident in the organization of the work.  For example, some introductory material is reiterated.

\section{Summary of the chapters}

\subsection*{Chapter 2: Local holomorphicity of generalized complex structures}

A generalized complex structure determines a Poisson structure and, transverse to its symplectic leaves, a complex structure.  In fact, Gualtieri showed that near a regular point of a generalized complex manifold, there is a local normal form constructed as the product of a symplectic manifold with a complex manifold \cite{Gualtieri2011}.  However, near points where the Poisson rank changes, much less was known.  Abouzaid and Boyarchenko \cite{AbouzaidBoyarchenko} showed that near any point of a generalized complex manifold there is a local model constructed as the product of a symplectic manifold with a generalized complex manifold whose Poisson tensor vanishes at the point (similar to Weinstein's result on the local normal form of a Poisson structure).

The question that remains of the local structure, then, is: what do generalized complex structures look like near a point with vanishing Poisson tensor, that is, near a point of complex type?  Every known example arose from a deformation of a complex structure by a \emph{holomorphic} Poisson structure.  In Chapter 1, we prove that this is always the case.  We use a Nash-Moser type rapidly-converging algorithm on shrinking neighbourhoods, in the style of Conn \cite{Conn}, adapting a more modern formalism of ``scaled $C^\infty$ spaces'' from Miranda, Monnier and Zung \cite{MirandaMonnierZung}.

%Thus, the Poisson structure of a generalized complex structure is more restricted than general real smooth Poisson structures.  In particular, the type-change locus, where the rank of the Poisson tensor is non-generic, will locally admit the structure of an analytic subvariety. (We investigate this in more detail in Chapter 4.)  This is in marked contrast to the real smooth case, where the type-change locus could be very badly behaved.

\subsection*{Chapter 3: Generalized complex structures on symplectic foliations}

In this chapter, we give answers to the question ``when are a regular Poisson structure, along with a complex structure transverse to its leaves, induced by generalized complex structure?''  Working within the \emph{pure spinor} formalism of generalized complex geometry, we show that a necessary and sufficient condition is the existence of certain integrating forms satisfying a system of differential equations.  We give more concrete criteria in the case of smooth symplectic families over a complex manifold---these are fibre bundles with leafwise symplectic structure.  We find that, as a necessary condition, the relative cohomology of the symplectic form should be \emph{pluriharmonic} under the Gauss-Manin connection.  Thus, we find examples of smooth symplectic families which are not induced by generalized complex structures.

If a smooth symplectic family induced by a generalized complex structure is a surface bundle, or in higher dimensions if certain topological conditions are satisfied, we show that it is in fact a symplectic fibre bundle, that is, it has symplectic trivializations.  (We use this result in Chapter 4.)  However, these conditions are quite special, and we give an example in 6 dimensions of a smooth symplectic family which is generalized complex but which is not a symplectic fibre bundle, that is, its fibres are inequevalent as symplectic manifolds.

\subsection*{Chapter 4: Nondegenerate type-change loci}

\emph{Type change} in generalized complex geometry is the phenomenon that the number of symplectic and complex dimensions may vary from place to place on a generalized complex manifold.  In this chapter, a joint project with Marco Gualtieri, we study the phenomenon of type change in high dimensions.  (Some work on the 4-dimensional case is in \cite{CavalcantiGualtieri} and elsewhere.)  Subject to a nondegeneracy condition, we show that the type-change locus inherits the structure of a generalized Calabi-Yau manifold \cite{Hitchin2003}, that is, the \emph{canonical line bundle} of the reduced generalized complex structure on the locus has a closed section.  Furthermore, subject to compactness and connectedness conditions, we show that such a locus has the structure of a smooth symplectic family over an elliptic curve.

We give a characterization of neighbourhoods of nondegenerate type change loci, by showing that the generalized complex structure near such a locus induces a certain linear generalized complex structure on its normal bundle, and then showing that this structure is equivalent to the original generalized complex structure near the locus.

\subsection*{Chapter 5: Generalized complex flat principal bundles}

Even if two generalized complex structures share their induced Poisson and transverse complex structures, they may be inequivalent.  Thus, in this chapter, we turn from the existence question to a question of classification.  We consider a notion of group action on generalized complex manifolds.  When such an action is free and proper, and satisfies a complementarity condition between the orbits and the symplectic leaves, we will have a generalized complex principal bundle; furthermore, the symplectic leaves will induce a generealized complex flat connection.

Therefore, in order to study this case, we give a definition of generalized complex flat principal bundles, and prove that free and proper generalized complex group actions (satisfying the complementarity condition) are examples of such bundles.  We prove a classification result, analogous to the classification of (non-generalized) flat principal bundles by their monodromy.  Because the automorphism group of a Courant algebroid is larger than the diffeomorphism group of the underlying manifold, there are ``non-geometric'' degrees of freedom.  We study this situation in more detail when the group is a torus, in which case a very explicit description of the extra degrees of freedom is given.

\title{Local holomorphicity of generalized complex structures}
\author{Michael Bailey}
\date{December 2011}
\maketitle

\begin{abstract}
A generalized complex structure determines a Poisson structure and, transverse to its symplectic leaves, a complex structure.  Abouzaid and Boyarchenko (2004) showed that near any point of a generalized complex manifold there is a local model constructed as the product of a symplectic manifold with a generalized complex manifold whose Poisson tensor vanishes at the point (similar to Weinstein's result on the local normal form of a Poisson structure).

The question that remains of the local structure, then, is: what do generalized complex structures look like near a point with vanishing Poisson tensor, that is, near a point of complex type?  We prove that they arise as a deformation of a (non-unique) complex structure by a \emph{holomorphic} Poisson structure.  We use a Nash-Moser type rapidly-converging algorithm on shrinking neighbourhoods, in the style of Conn (1985), adapting a more modern formalism of ``scaled $C^\infty$ spaces'' from Miranda, Monnier and Zung (2011).
\end{abstract}
}

\thesis{
\chapter{Local holomorphicity of generalized complex structures}
}

\section{Local structure of generalized complex structures}

\begin{defn}
A generalized complex structure on a manifold $M$ is a complex structure, $J^2 = -1$, on the vector bundle $TM \dsum T^*M$, which is orthogonal with respect to the standard symmetric pairing, and whose $+i$-eigenbundle is involutive with respect to the Courant bracket, defined as follows: let $X,Y \in \Gamma(TM)$ and $\xi,\eta \in \Gamma(T^*M)$; then
\begin{equation}\label{bracket formula}
[X+\xi,Y+\eta] = [X,Y]_{\textnormal{Lie}} + \Lie_X \eta - \iota_Y d\xi.
\end{equation}
\end{defn}

The Courant bracket (actually, in this form, due to Dorfman \cite{Dorfman}) usually has an additional twisting term involving a closed 3-form.  However, every such bracket is in a certain sense locally equivalent to the untwisted bracket above, and since this \thesis{chapter}\article{paper} studies the local structure of generalized complex structures, we ignore the twisting for now.

\begin{example}\label{symplectic structure}
If $\omega:TM \to T^*M$ is a symplectic structure, then
$$J_\omega =
\left[\begin{array}{cc}
0 & -\omega^{-1} \\
\omega & 0
\end{array}\right]$$
is a generalized complex structure.
\end{example}

\begin{example}\label{complex structure}
If $I:TM \to TM$ is a complex structure, then
$$J_I = 
\left[\begin{array}{cc}
-I & 0 \\
0 & I^*
\end{array}\right]$$
is a generalized complex structure.
\end{example}

\begin{rem}
A generalized complex structure may be of complex type or symplectic type \emph{at a point $p$}, if it is of one of the above forms on $T_pM\dsum T^*_pM$, while having a different type elsewhere.
\end{rem}

\begin{example}
If $J_1$ is a generalized complex structure on $M_1$ and $J_2$ is a generalized complex structure on $M_2$, then $J_1 \times J_2$ is a generalized complex structure on $M_1 \times M_2$ in the obvious way.
\end{example}

\begin{defn}
A \emph{Courant isomorphism} $\Phi : TM\dsum T^*M \to TN\dsum T^*N$ is a vector bundle isomorphism of $TM\dsum T^*M$ to $TN \dsum T^*N$ which respects the Courant bracket, the symmetric pairing, and the projection to the tangent bundle.
\end{defn}

The first result on the local structure of generalized complex structures was due to Gualtieri \cite{Gualtieri2011}.  It was strengthened by Abouzaid and Boyarchenko \cite{AbouzaidBoyarchenko}, as follows:

\begin{thm}[Abouzaid, Boyarchenko]
If $M$ is a generalized complex manifold and $p\in M$, then there is a neighourhood of $p$ which is isomorphic, via Courant isomorphism, to a product of a generalized complex manifold which is symplectic everywhere and a generalized complex manifold which is of complex type at the image of $p$.
\end{thm}

This resembles Weinstein's local structure theorem for Poisson structures \cite{Weinstein}.  In fact, a generalized complex structure induces a Poisson structure, for which this result produces the Weinstein decomposition.

Thus, the remaining question in the local classification of generalized complex structures is: what kinds of generalized complex structures occur near a point of complex type?  There is a way in which any holomorphic Poisson structure (see Section \ref{holomorphic Poisson} below) induces a generalized complex structure (as described in Section \ref{deformation section}).  Our main result, then, is as follows:

\begin{main thm}
Let $J$ be a generalized complex structure on a manifold $M$ which is of complex type at point $p$.  Then, in a neighbourhood of $p$, $J$ is Courant-equivalent to a generalized complex structure induced by a holomorphic Poisson structure, for some complex structure near $p$.
\end{main thm}

This is finally proven in Section \ref{lemma implies theorem}.  Most of the work happens in earlier sections, in proving the following lemma:

\begin{main lem}
Let $J$ be a generalized complex structure on the closed unit ball $B_1$ about the origin in $\C^n$.  Suppose $J$ is a small enough deformation of the complex structure on $B_1$, and suppose that $J$ is of complex type at the origin.  Then, in a neighbourhood of the origin, $J$ is Courant-equivalent to a deformation of the complex structure by a holomorphic Poisson structure on $\C^n$.
\end{main lem}

In Section \ref{deformation section}, we explain how one generalized complex structure may be understood as a deformation of another.  By the smallness condition we mean that there is some $l\in\N$ such that if the deformation is small enough in its $C^l$-norm then the conclusion holds.  (See Section \ref{norm section} for details about the norms.)  The proof of the Main Lemma is in Section \ref{prove main lemma} (modulo technical results in Sections \ref{verifying SCI} and \ref{verifying hypotheses}).

In some sense, then, generalized complex structures are holomorphic Poisson structures twisted by (possibly) non-holomorphic Courant gluings.  A generalized complex manifold may not, in general, admit a global complex structure \cite{CavalcantiGualtieri2007} \cite{CavalcantiGualtieri}, emphasizing the local nature of our result.

\subsection{Holomorphic Poisson structures}\label{holomorphic Poisson}

A holomorphic Poisson structure on a complex manifold $M$ is given by a holomorphic bivector field $\beta \in \Gamma(\^ ^2 T_{1,0}M),\; \bar\del\beta=0,$ for which the Schouten bracket, $[\beta,\beta]$, vanishes.  $\beta$ determines a Poisson bracket on holomorphic functions, $\{f,g\} = \beta(df,dg)$.  The holomorphicity condition, $\bar\del\beta=0$, means that, if $\beta$ is written in local coordinates,
$$\beta = \sum_{i,j} \beta_{ij} \frac{d}{dz_i}\^\frac{d}{dz_j},$$
then the component functions, $\beta_{ij}$, are holomorphic.  For a review of holomorphic Poisson structures see \cite{LSX}.

The \emph{type-change locus} of a generalized complex structure induced by a holomorphic Poisson structure, that is, the locus where the Poisson rank changes, is determined by the vanishing of an algebraic function of the component functions above, thus,
\begin{cor}
The type-change locus of a generalized complex structure locally admits the structure of an analytic subvariety.
\end{cor}

After a review of the literature we believe the following to be an open question: is every holomorphic Poisson structure locally equivalent to one which is polynomial in some coordinates?  We are unaware of any counterexamples, and there are some partial results \cite{DufourWade} \cite{Lohrmann}.

\subsection{Outline of the proof of the Main Lemma}\label{outline of proof}

In Section \ref{deformation section} we describe the deformation complex for generalized complex structures, and how it interacts with \emph{Courant flows} coming from \emph{generalized vector fields}.  In Section \ref{infinitesimal case} we solve an infinitesimal version of the problem, by showing that, to first order, an infinitesimal generalized complex deformation of a holomorphic Poisson structure is equivalent to another holomorphic Poisson structure.  Then the full problem is solved by iterating an approximate version of the infinitesimal solution:

At each stage of the iteration, we have a generalized complex structure which is a deformation of a given complex structure.  We seek to cancel the part of this deformation which is \emph{not} a Poisson bivector.  We construct a generalized vector field whose Courant flow acting on the deformation should cancel this non-bivector part, to first order.  Then after each stage the unwanted part of the deformation should shrink quadratically.  We mention two issues with this algorithm:

Firstly, at each stage we ``lose derivatives,'' meaning that the $C^k$-convergence will depend on ever higher $C^{k+i}$-norms.  The solution is to apply Nash's smoothing operators at each stage to the generalized vector field, where the smoothing degree is carefully chosen to compensate for loss of derivatives while still achieving convergence.  A good general reference for this sort of technique (in the context of compact manifolds) is \cite{Hamilton}, and it is tempting to try to apply the Nash-Moser implicit function theorem directly. However, this is frustrated by the second issue:

Since we are working on a neighbourhood of a point $p$, the generalized vector field will not integrate to a Courant automorphism of the whole neighbourhood.  Thus, after each stage we may have to restrict our attention to a smaller neighbourhood of $p$.  If the radius restriction at each stage happens in a controlled way, then the limit will be defined on a ball of radius greater than 0.  The technique for proving Nash-Moser type convergence results on shrinking neighbourhoods comes from Conn \cite{Conn}.

We adopt a more recent formalization of this technique, by Miranda, Monnier and Zung \cite{MirandaMonnierZung} \cite{MonnierZung} (Section \ref{SCI section}).  In fact, for much of the work we use a general technical lemma of theirs (Theorem \ref{Miranda-Monnier-Zung}) with only a few changes.  Even so, we must prove estimates for the behaviour of Courant flows acting on deformations (see Sections \ref{verifying SCI} and \ref{verifying hypotheses}).

\section{The deformation complex}\label{deformation section}
In this section, if $V$ is a vector bundle, let $\Gamma(V)$ denote its smooth sections.  We will now describe the deformation complex for generalized complex structures.  Except where we remark otherwise, the results in this section (\ref{deformation section}) can be found in Section 5 of \cite{Gualtieri2011}.  We make use of the fact that a generalized complex structure is determined by its $+i$-eigenbundle.

Let $T$ be the tangent bundle of some manifold $M$.  Let $L \subset \C \tens (T \dsum T^*)$ be the $+i$-eigenbundle for an \emph{initial} generalized complex structure.  We will often take this initial structure to be the complex structure on $\C^n$, in which case
$$L = T_{0,1} \dsum T^*_{1,0}.$$
Another example arises from a holomorphic Poisson structure on $\C^n$.  If $\beta : T^*_{1,0} \to T_{1,0}$ is a holomorphic Poisson bivector, then we define the corresponding generalized complex structure, with $+i$-eigenbundle
\begin{equation}\label{L for holomorphic bivector}
L = T_{0,1} \dsum \textnormal{graph}(\beta).
\end{equation}

In any case, the $+i$-eigenbundle of a generalized complex structure is a maximal isotropic subbundle.  Using the pairing, we choose an embedding of $L^*$ in $T_\C \dsum T_\C^*$, which will be transverse to $L$ and isotropic with respect to the standard symmetric pairing.  One choice is $L^* \iso \bar{L}$, though we may take others.  Any maximal isotropic $L_\epsilon$ close to $L$ may thus be realized as
\begin{equation}\label{define deformation}
L_\epsilon = (1 + \epsilon)L,
\end{equation}
where $\epsilon : L \to L^* \subset T_\C \dsum T^*_\C$.  As a consequence of the maximal isotropic condition on $L_\epsilon$, $\epsilon$ will be antisymmetric, and we can say that $\epsilon \in \Gamma(\^ ^2 L^*)$.  In fact, for any $\epsilon \in \Gamma(\^ ^2 L^*)$, $L_\epsilon$ is the $+i$--eigenbundle of an almost generalized complex structure.  Of course, for $L_\epsilon$ to be integrable, $\epsilon$ must satisfy a differential condition---the Maurer-Cartan equation (see Section \ref{Maurer-Cartan section}).

\begin{rem}\label{L* convention}
For initial complex structures, we will use the convention $L^* \iso \bar{L}$, so that $L^* = T_{1,0} \dsum T^*_{0,1}$.  Since the only requirement on the embedding of $L^*$ is that it be transverse to $L$ and isotropic (and thus give a representation of $L^*$ by the pairing), we will take this same choice of $L^*$ whenever possible; that is, we henceforth fix the notation
\begin{equation}
L^* = T_{1,0} \dsum T^*_{0,1},
\end{equation}
regardless of which eigenbundle $L$ we are dealing with.
\end{rem}

\begin{rem}
If the initial structure is complex and $\epsilon = \beta \in \Gamma(\^ ^2 T_{1,0})$ is a holomorphic Poisson bivector, then the deformed eigenbundle $L_\epsilon$ agrees with \eqref{L for holomorphic bivector}.
\end{rem}

\subsection{Generalized Schouten bracket and Lie bialgebroid structure}
While $T\dsum T^*$ is not a Lie algebroid for the Courant bracket (which isn't antisymmetric), the restriction of the bracket to $L$ \emph{does} give a Lie algebroid structure.  From this, there is a naturally-defined differential
$$d_L : \Gamma(\^ ^k L^*) \to \Gamma(\^ ^{k+1} L^*)$$
as well as an extension of the bracket (in the manner of Schouten) to higher wedge powers of $L$.  But $L^*$ is also a Lie algebroid, and the same structures apply.  Together they form a Lie bialgebroid (actually, a differential Gerstenhaber algebra if we consider the wedge product), meaning that $d_L$ is a derivation for the bracket on $\^ ^\bullet L^*$:
\begin{equation}\label{bialgebroid}
d_L[\alpha,\beta] = [d_L\alpha,\beta] + (-1)^{|\alpha|-1}\, [\alpha,d_L\beta]
\end{equation}
For more details on Lie bialgebroids see \cite{LiuWeinsteinXu}, and for their relation to generalized complex structures see \cite{GrandiniPoonRolle} and \cite{Gualtieri2011}.

\begin{example}
If $L$ corresponds to a complex structure, then $d_L = \bar\del$.  We can find the differential for other generalized complex structures by using the following fact, quoted from \cite{GrandiniPoonRolle}:
\end{example}

\begin{prop}\label{deformation of d_L}
Let $L_\epsilon$ be an integrable deformation of a generalized complex structure $L$ by $\epsilon \in \^ ^2 L^*$.  As per Remark \ref{L* convention}, we identify $L_\epsilon^* = L^* = \bar{L}$, and thus identify their respective differential complexes as sets.  Then for $\sigma \in \Gamma(\^ ^k L^*)$,
$$d_{L_\epsilon} \sigma = d_L \sigma + [\epsilon, \sigma].$$
\end{prop}

\begin{example}\label{d_L for Poisson}
Thus, the differential on $\Gamma(\^ ^k L^*)$ coming from a holomorphic Poisson structure $\beta \in \Gamma(\^ ^2 T_{1,0})$ is just
$$d_{L_\beta} = \bar\del + d_\beta,$$
where $d_\beta$ is the usual Poisson differential $[\beta,\cdot]$.
\end{example}

\subsection{Integrability and the Maurer-Cartan equation}\label{Maurer-Cartan section}

For a deformed structure $L_\epsilon$ to be integrable, $\epsilon$ must satisfy the Maurer-Cartan equation,
\begin{equation}\label{MC equation}
d_L \epsilon + \frac{1}{2}[\epsilon,\epsilon] = 0
\end{equation}

\begin{notn}
Suppose $L$ is the $+i$-eigenbundle for the generalized complex structure on $\C^n$ coming from the complex structure, and suppose that $L^* = T_{1,0} \dsum T^*_{0,1}$ as in Remark \ref{L* convention}.  We may write
$$\^ ^2 L^* = (\^ ^2 T_{1,0}) \dsum (T_{1,0} \tens T^*_{0,1}) \dsum (\^ ^2 T^*_{0,1}).$$
If $\epsilon \in \Gamma(\^ ^2 L^*)$ is a deformation, we will write $\epsilon$ correspondingly as $\epsilon_1 + \epsilon_2 + \epsilon_3$, where $\epsilon_1$ is a bivector field, $\epsilon_2 \in \Gamma(T_{1,0} \tens T^*_{0,1})$, and $\epsilon_3$ is a 2-form.
\end{notn}

Then the Maurer-Cartan condition (\ref{MC equation}) on $\epsilon$ splits into four equations:
\begin{align}
\^ ^3 T_{1,0}  & \quad:\quad  [\epsilon_1,\epsilon_1] = 0  \label{MC1}\\
\^ ^2 T_{1,0} \tens T^*_{0,1}  & \quad:\quad  [\epsilon_1,\epsilon_2] + \bar\del\epsilon_1 = 0  \label{MC2}\\
T_{1,0} \tens \^ ^2 T^*_{0,1}  & \quad:\quad  \frac{1}{2} [\epsilon_2,\epsilon_2] + [\epsilon_1,\epsilon_3] + \bar\del\epsilon_2 = 0  \label{MC3}\\
\^ ^3 T^*_{01,}  & \quad:\quad  \bar\del\epsilon_3 = 0  \label{MC4} 
\end{align}

\begin{rem}\label{bivector implies holomorphic}
By \eqref{MC1}, $\epsilon_1$ always satisfies the Poisson condition.  If $\epsilon_2=0$ then, by \eqref{MC2}, $\epsilon_1$ is also holomorphic.  Therefore, to say that an integrable deformation $\epsilon$ is holomorphic Poisson is the same as to say that $\epsilon_2$ and $\epsilon_3$ vanish, that is, that $\epsilon$ is just a bivector.
\end{rem}

\subsection{Courant automorphisms}\label{Courant automorphisms}

\begin{defn}\label{Courant automorphism}
A \emph{Courant automorphism} $\Phi : T\dsum T^* \to T\dsum T^*$, also called a \emph{generalized diffeomorphism}, is an isomorphism of $T\dsum T^*$ (covering some diffeomorphism) which respects the Courant bracket, the symmetric pairing, and the projection to the tangent bundle.

A $B$-transform is a particular kind of Courant automorphism: if $B:T\to T^*$ is a closed 2-form and $X+\xi\in T\dsum T^*$, then
$e^B(X+\xi) = (1+B)(X+\xi) = X+\iota_X B+\xi.$

Another kind of Courant automorphism is a diffeomorphism acting by pushforward (which means inverse pullback on the $T^*$ component).  We will typically identify a Courant automorphism $\Phi$ with a pair $(B,\phi)$, where $B$ is a closed 2-form and $\phi$ is a diffeomorphism---then $\Phi$ acts first through a $B$-transform and then through pushforward by $\phi_*$.  Such pairs exhaust the Courant automorphisms \cite{Gualtieri2011}.
\end{defn}

\begin{rem}\label{Courant composition}
Let $\Phi=(B,\phi)$ and $\Psi=(B',\psi)$ be Courant automorphisms.  Then
$$\Phi \comp \Psi = (\psi^*(B) + B',\phi\comp\psi)  \quad\textnormal{and}\quad
\Phi^{-1} = (-\phi_*(B),\phi^{-1})$$
\end{rem}

\begin{defn}\label{Courant action on deformation}
If $L_\epsilon$ is a deformation of generalized complex structure $L$, and $\Phi$ is a Courant automorphism of sufficiently small 1-jet, then $\Phi(L_\epsilon)$ is itself a deformation of $L$.  Let $\Phi\cdot\epsilon \in \Gamma(\^ ^2 L^*)$ be such that $L_{\Phi\cdot\epsilon} = \Phi(L_\epsilon)$, that is,
\begin{equation}
\Phi\left((1+\epsilon)L\right) = (1+\Phi\cdot\epsilon)L.
\end{equation}
\end{defn}

\begin{rem}\label{not pushforward}
For a more concrete formula for $\Phi\cdot\epsilon$, see Proposition \ref{action formula}.  In general $\Phi\cdot\epsilon$ should not be understood as a pushforward of the tensor $\epsilon$.  (In fact, $\Phi\cdot0$ may be nonzero!)  However, if $\Phi(L)=L$ then indeed $\Phi\cdot\epsilon = \Phi_*(\epsilon)$ suitably interpreted.
\end{rem}

\begin{defn}\label{Courant flow}
A section $v \in \Gamma(T\dsum T^*)$ is called a \emph{generalized vector field}.  We say that $v$ generates the 1-parameter family $\Phi_{tv}$ of generalized diffeomorphisms if for any section $\sigma\in T\dsum T^*$,
\begin{equation}\label{tensor Courant derivative}
\left.\frac{d}{dt}\right|_{\tau=t}\left(\Phi_{\tau v}\right)_*\sigma = [v,\left(\Phi_{tv}\right)_*\sigma].
\end{equation}
\end{defn}

The generalized diffeomorphism thus defined is related to a classically generated diffeomorphism as follows:

Let $v = X + \xi$, where $X$ is a vector field and $\xi$ a 1-form.  If $X$ is small enough, or the manifold is compact, then it integrates to the diffeomorphism $\phi_X$ which is its time-$1$ flow.  Let
$$B_v = \int_0^1 \phi_{tX}^* (d\xi) dt.$$
Then $\Phi_v = (B_v, \phi_X)$ is the time-1 Courant flow of $v$.

\begin{rem}
If $X$ does not integrate up to time 1 from every point, then $\phi_X$, and thus $\Phi_v$, is instead defined on a subset of the manifold.  In this case, $\Phi_v$ is a \emph{local} Courant automorphism (or \emph{local} generalized diffeomorphism).
\end{rem}

\begin{rem}
While \eqref{tensor Courant derivative} gives the derivative of a Courant flow acting by pushforward on a tensor, it does \emph{not} hold for derivatives of Courant flows acting by the deformation action of Definition \ref{Courant action on deformation}, as we see from Remark \ref{not pushforward}.
\end{rem}

The following is a corollary to \cite[Prop. 5.4]{Gualtieri2011}: 
\begin{lem}
If $0 \in \Gamma(\^ ^2 L^*)$ is the trivial deformation of $L$ and $v \in \Gamma(T\dsum T^*)$, then
$$\left.\frac{d}{dt}\Phi_{tv}\cdot 0\right|_{t=0} = d_L v^{0,1},$$
where $v^{0,1}$ is the projection of $v$ to $L^*$.
\end{lem}

Then combining this fact with Proposition \ref{deformation of d_L} we see that
\begin{prop}\label{infinitesimal flow}
If $\epsilon \in \Gamma(\^ ^2 L^*)$ is an integrable deformation of $L$, and $v \in \Gamma(T\dsum T^*)$, then
$$\left.\frac{d}{dt}\Phi_{tv}\cdot \epsilon\right|_{t=0} = d_L v^{0,1} + [\epsilon,v],$$
where $v^{0,1}$ is the projection of $v$ to $L^*$.
\end{prop}

\begin{rem}
Definition \ref{Courant flow} makes sense if $v$ is a real section of $T\dsum T^*$.  On the other hand, if $v \in \Gamma(T_\C \dsum T^*_\C)$ is complex, we may interpret $\Phi_v$ in the presence of an underlying generalized complex structure as follows.  $v$ decomposes into $v^{1,0}\in L$ plus $v^{0,1}\in\bar{L}$.  We see in Proposition \ref{infinitesimal flow} that the component in $L$ has no effect on the flow of deformations, therefore we define
$$\Phi_v := \Phi_{v^{0,1} + \overline{v^{0,1}}},$$
where $v^{0,1} + \overline{v^{0,1}}$ is now real.  Proposition \ref{infinitesimal flow} still holds.
\end{rem}

\section{The infinitesimal case}\label{infinitesimal case}

We would like to make precise and then prove the following rough statement: if $\epsilon$ is an infinitesimal deformation of a holomorphic Poisson structure on the closed unit ball $B_1 \subset \C^n$, then we may construct an infinitesimal flow by a generalized vector field $V$ which ``corrects'' the deformation so that it remains within the class of holomorphic Poisson structures.  This is a cohomological claim about the complex $(\^ ^\bullet L^*, d_L)$.  When we consider the full problem of finite deformations, this will still be approximately true in some sense, which will help us prove the Main Lemma.

Suppose that $\epsilon_t$ is a one-parameter family of deformations of $L$.  Differentiating equation \eqref{MC equation} by $t$, we get that
$$d_L \dot\epsilon_t + [\epsilon_t,\dot\epsilon_t] = 0$$
If $\epsilon_0 = 0$, then we have the condition $d_L \dot\epsilon_0 = 0$.  That is, an infinitesimal deformation of $L$ must be $d_L$-closed.

Thus we make precise the statement in the opening paragraph of this section:
\begin{prop}\label{infinitesimal correction}
Suppose that $L$ is the $+i$-eigenbundle corresponding to a holomorphic Poisson structure $\beta$ on $B_1\subset\C^n$, and suppose that $\epsilon \in \Gamma(\^ ^2 L^*)$ satisfies $d_L \epsilon = 0$.  Then there exists $V(\beta,\epsilon) \in \Gamma(L^*)$ such that $\epsilon + d_L V(\beta,\epsilon)$ has only a bivector component.
\end{prop}

\begin{proof}
As in section \ref{Maurer-Cartan section}, we write $\epsilon = \epsilon_1 + \epsilon_2 + \epsilon_3$ where the terms are a bivector field, a mixed co- and contravariant term, and a 2-form respectively.  The closedness condition, $(\bar\del + d_\beta)\epsilon = 0$ (as per Example \ref{d_L for Poisson}), may be decomposed according to the co- and contravariant degree.

For example, we have $\bar\del\epsilon_3 = 0$.  Since $\bar\del$-cohomology is trivial on the ball $B_1$, there exists a $(0,1)$-form $P\epsilon_3$ such that $\bar\del P\epsilon_3 = \epsilon_3$.  $-P\epsilon_3$ will be one piece of $V(\beta,\epsilon)$.

Another component of the closedness condition is $\bar\del\epsilon_2 + d_\beta\epsilon_2 = 0$.  Then
\begin{eqnarray*}
\bar\del(\d_\beta P\epsilon_3 - \epsilon_2) &=& \bar\del\d_\beta P\epsilon_3 + \d_\beta\epsilon_3 \\
&=& \bar\del\d_\beta P\epsilon_3 + \d_\beta\bar\del P\epsilon_3
\end{eqnarray*}
But $\bar\del$ and $\d_\beta$ anticommute, so this is $0$, i.e., $\d_\beta P\epsilon_3 - \epsilon_2$ is $\bar\del$-closed.  Therefore it is $\bar\del$-exact, and there exists some $(1,0)$-vector field $P(\d_\beta P\epsilon_3 - \epsilon_2)$ such that $\bar\del P(\d_\beta P\epsilon_3 - \epsilon_2) = \d_\beta P\epsilon_3 - \epsilon_2$.  Let
\begin{equation}\label{construct V}
V(\beta,\epsilon) = P(\d_\beta P\epsilon_3 - \epsilon_2) - P\epsilon_3
\end{equation}
Then
$$(\bar\del + \d_\beta)V(\beta,\epsilon) = \d_\beta P(\d_\beta P\epsilon_3 - \epsilon_2) - \epsilon_2 - \epsilon_3,$$
where $\d_\beta P(\d_\beta P\epsilon_3 - \epsilon_2)$ is a section of $\^ ^2 T_{1,0}$.  Therefore
$$\epsilon + d_L V(\beta,\epsilon) \;\in\; \Gamma(\^ ^2 T_{1,0})$$
\end{proof}

\subsection{The $\bar\del$ chain homotopy operator}

The non-constructive step in the proof of Proposition \ref{infinitesimal correction} is the operation $P$ which gives $\bar\del$-primitives for sections of $(T_{1,0} \tens T^*_{0,1}) \dsum \^ ^2 T^*_{0,1}$.  Fortunately, in \cite{NijenhuisWoolf} Nijenhuis and Woolf give a construction of such an operator and provide norm estimates for it.

\begin{prop}\label{existence of P}
For a closed ball $B_r \subset \C^n$, there exists a linear operator $P$ such that for all $i,j\geq0$,
$$P : \Gamma\left(\left(\^ ^i T_{1,0}\right) \tens \left(\^ ^{j+1} T^*_{0,1}\right)\right)
\to \Gamma\left(\left(\^ ^i T_{1,0}\right) \tens \left(\^ ^j T^*_{0,1}\right)\right)$$
such that
\begin{equation}\label{P error}
\bar\del P + P \bar\del = \Id.
\end{equation}
and such that the $C^k$-norms of $P$ satisfies the estimate, for all integers $k\geq0$,
$$\|P\epsilon\|_k \,\leq\, C\,\|\epsilon\|_k.$$
(See Section \ref{norm section} for details on $C^k$ norms.)
\end{prop}

We note that $P$ is defined on \emph{all} smooth sections, not just $\bar\del$-closed sections.  But if $\bar\del\epsilon=0$, $\bar\del P\epsilon = \epsilon$ as desired.

\begin{proof}
For a $(0,j)$ form, $P$ is just the operator $T$ defined in \cite{NijenhuisWoolf}.  We don't give the full construction here (or the proofs of its properties), but we remark that it is built inductively from the case of a 1-form $f\,d\bar{z}$ on $\C$, for which
$$(T\, f\,d\bar{z}) (x) = \frac{-1}{2\pi i} \int_{B_r} \frac{f(\zeta)}{\zeta-x}\, d\zeta\^d\bar\zeta.$$

On the other hand, if $\epsilon \in \Gamma\left(\left(\^ ^i T_{1,0}\right) \tens \left(\^ ^{j+1} T^*_{0,1}\right)\right)$ for $i>0$, we may write
$$\epsilon = \sum_I \frac{d}{dz_I} \tens \epsilon_I,$$
where $I$ ranges over multi-indices, $\frac{d}{dz_I}$ is the corresponding basis multivector, and $\epsilon_I \in \Gamma\left(\^ ^{j+1} T^*_{0,1}\right)$.  Then $T$ is applied to each of the $\epsilon_I$ individually.

The estimate is also from \cite{NijenhuisWoolf}, and by construction of $P$ clearly also applies to mixed co- and contravariant tensors.
\end{proof}

$P$ as defined depends continuously on the radius, $r$, of the polydisc---that is, it doesn't commute with restriction to a smaller radius.  We say no more about this quirk except to note that it is compatible with Theorem \ref{Miranda-Monnier-Zung}.

\subsection{Approximating the finite case with the infinitesimal solution}

We sketch how Proposition \ref{infinitesimal correction} roughly translates to the finite case (for details, see Lemma \ref{almost infinitesimal}):

We will be considering deformations $\epsilon = \epsilon_1 + \epsilon_2 + \epsilon_3$ of the complex structure on $B_r \subset \C^n$, which are close to being holomorphic Poisson; thus, $\epsilon_2$ and $\epsilon_3$ will be small and $\epsilon_1$ will almost be a holomorphic Poisson bivector.  We then pretend that $\epsilon_2 + \epsilon_3$ is a small deformation of the almost holomorphic Poisson structure $\beta = \epsilon_1$, and the argument for Proposition \ref{infinitesimal correction} goes through approximately.  Thus,

\begin{defn}\label{define V}
If $\epsilon \in \Gamma(\^ ^2 L^*)$, with the decomposition $\epsilon = \epsilon_1 + \epsilon_2 + \epsilon_3$ as in Section \ref{Maurer-Cartan section}.  Then let
$$V(\epsilon) = V(\epsilon_1,\epsilon_2 + \epsilon_3) = P([\epsilon_1, P\epsilon_3] - \epsilon_2 - \epsilon_3).$$
\end{defn}

In the above construction, we apply $P$ to sections which are not quite $\bar\del$-closed, so it will not quite yield $\bar\del$-primitives; this error is controlled by equation \eqref{P error}.  Furthermore, we can no longer say that $[\epsilon_1,\cdot]$ and $\bar\del$ anticommute; this error will be controlled by the bialgebroid property \eqref{bialgebroid}, with $d_L = \bar\del$, so that if $\theta \in \Gamma(\^ ^\bullet L^*)$ then
\begin{equation}\label{anticommute error}
\bar\del[\epsilon_1, \theta] = -[\epsilon_1, \bar\del \theta] + [\bar\del\epsilon_1, \theta].
\end{equation}

\section{SCI-spaces and the abstract normal form theorem}\label{SCI section}

As discussed in Section \ref{outline of proof}, the local equivalence of a generalized complex structure near a point $p$ to a holomorphic Poisson structure will be achieved by iteratively applying a particular sequence of \emph{local} generalized diffeomorphisms to the initial structure, and then arguing that in the limit this sequence takes the initial structure to a holomorphic Poisson structure.  A difficulty with this approach is that at each stage we may have to restrict to a smaller neighbourhood of $p$.  Thus the iteration is not in a fixed space of deformations, but rather through a collection of spaces, one for each neighbourhood of $p$.

The technique for handling this difficulty comes from Conn \cite{Conn}, though we have adopted some of the formalism of Miranda, Monnier and Zung \cite{MonnierZung} \cite{MirandaMonnierZung}, with \cite[Appendix 1]{MirandaMonnierZung} our main reference.  We adapt the definition of SCI-spaces---or ``scaled $C^\infty$'' spaces---SCI-groups and SCI-actions, with some changes which we discuss.  In particular, for simplicity we consider only the ``$C^\infty$'' part of the space (whereas in \cite{MirandaMonnierZung} $C^k$ sections are considered).  Hence, an SCI-space is a radius-parametrized collection of tame Frechet spaces.  To be precise:

\begin{defn}
An \emph{SCI-space} $\sr H$ consists of a collection of vector spaces $\sr{H}_r$ with norms $\|\cdot\|_{k,r}$---where $k\geq0$ (the \emph{smoothness}) is in $\Z$ and $0<r\leq1$ (the \emph{radius}) is in $\R$---and for every $0<r'<r\leq1$ a linear \emph{restriction map}, $\pi_{r,r'} : \sr H_r \to \sr H_{r'}$.  Furthermore, the following properties should hold:
\begin{itemize}
\item If $r>r'>r''$ then $\pi_{r,r''}=\pi_{r,r'}\comp\pi_{r',r''}$.
\end{itemize}
If $f \in \sr{H}_r$ then, to abuse notation, we denote $\pi_{r,r'}(f) \in \sr H_{r'}$ also by $f$.
\begin{itemize}
\item If $f$ in $\sr{H}$, $r'\leq r$ and $k'\leq k$, then
$$\|f\|_{k',r'} \leq \|f\|_{k,r},$$
\end{itemize}
where if neither $f$ nor a restriction of $f$ is in $\sr H_r$ then we interpret $\|f\|_r = \infty$.  We take as the topology for each $\sr{H}_r$ the one generated by open sets in every norm.  We require that
\begin{itemize}
\item If a sequence in $\sr{H}_r$ is Cauchy for each norm $\|\cdot\|_k$ then it converges in $\sr{H}_r$.
\item At each radius $r$ there are \emph{smoothing operators}, that is, for each real $t>1$ there is a linear map
$$S_r(t) : \sr H_r \to \sr H_r$$
such that for any $p>q$ in $\Z^+$ and any $f$ in $\sr H_r$,
\begin{eqnarray}
\|S_r(t)f\|_{p,r} &\leq& C_{r,p,q} t^{p-q}\|f\|_{q,r} \quad\textnormal{and} \\
\|f-S_r(t)f\|_{q,r} &\leq& C_{r,p,q} t^{q-p}\|f\|_{p,r},
\end{eqnarray}
where $C_{r,p,q}$ is a positive constant depending continuously on $r$.
\end{itemize}

An \emph{SCI-subspace} $\sr S \subset \sr H$ consists of a collection of subspaces $\sr S_r \subset \sr H_r$ which themselves form an SCI-space under the induced norms, restriction maps and smoothing operators.  An \emph{SCI-subset} of $\sr H$ consists of a collection of subsets of the $\sr H_r$ which is invariant under the restriction maps. 
\end{defn}

\begin{example}\label{sections are SCI}
Let $V$ be a finite-dimensional normed vector space.  For each $0<r\leq1$, let $B_r \subset \R^n$ or $\C^n$ be the closed unit ball of radius $r$ centred at the origin (under the $\sup$-norm, this is actually a rectangle or polydisc), and let $\sr H_r$ be the $C^\infty$-sections of the trivial bundle $B_r \times V$, with $\|\cdot\|_{k,r}$ the $C^k$-$\sup$ norm.  Then the $\sr H_r$ and $\|\cdot\|_{k,r}$ form an SCI-space.
\end{example}

\begin{rem}
At a fixed radius $r$, $\sr H_r$ is a tame Frechet space.  There are constructions of smoothing operators in many particular instances (see, eg., \cite{Hamilton}).  The essential point is that $S_r(t)f$ is a smoothing of $f$, in the sense that its higher norms are controlled by lower norms of $f$, and as $t$ gets larger, $S_r(t)f$ is a better approximation to $f$, but is less smooth.  As a consequence of the existence of smoothing operators, we have the \emph{interpolation inequality} (also see \cite{Hamilton}):
\end{rem}
\begin{prop}\label{interpolation inequality}
Let $\sr{H}$ be an SCI-space, let $0\leq l \leq m \leq n$ be integers, and let $r>0$.  Then there is a constant $C_{l,m,n,r}>0$ such that for any $f \in \sr{H}_r$,
$$\|f\|_m^{n-l} \leq C_{l,m,n,r}\, \|f\|_n^{m-l}\, \|f\|_l^{n-m}.$$
\end{prop}

\subsection{Notational conventions}

We will need to express norm estimates for members of SCI-spaces, that is, we will write SCI-norms into inequalities.  We develop some shorthand for this, which is similar to (but extends) the notation in \cite{MirandaMonnierZung}.

\subsubsection*{Spaces of sections}
If $E = B_1 \times V$ is a vector bundle over $B_1 \subset \R^n$ or $\C^n$, then by $\Gamma(E)$ we will always mean the SCI-space of local sections of $E$ near $0 \in \C^n$, as in Example \ref{sections are SCI}.
\subsubsection*{Radius parameters}
We will often omit the radius parameter when writing SCI-norms (but we will always include the degree).  The right way to interpret such notation is as follows: when the norms appear in an equation, the claim is that this equation holds for any common choice of radius where all terms are well-defined.  When the norms appear in an inequality, the claim is that the inequality holds for any common choice of radius $r$ for the lesser side of the inequality, with any common choice of radius $r' \geq r$ for the greater side of the inequality (for which all terms are well-defined).

For example, for $f\in\sr{H}$ and $g\in\sr{K}$,
$$\|f\|_k \leq \|g\|_{k+1}$$
means
$$\forall\; 0<r\leq r'\leq1,\;\textnormal{if}\,f\in\sr{H}_r\,\textnormal{and}\, g\in\sr{H}_{r'}\,\textnormal{then}\, \|f\|_{k,r}\leq\|g\|_{k+1,r'}$$
\begin{rem}
Since the norms are nondecreasing in radius, this convention is plausible.
\end{rem}
\subsubsection*{Constants}
Whenever it appears in an inequality, $C$ (or $C'$) will stand for a positive real constant, which may be different in each usage, and which may depend on the degree, $k$, of the terms, and continuously on the radius.
\subsubsection*{Polynomials}
Whenever the notation $$Poly(\|f_1\|_{k_1},\|f_2\|_{k_2},\ldots)$$ occurs, it denotes some polynomial in $\|f_1\|_{k_1}$, $\|f_2\|_{k_2}$, etc., with positive coefficients, which may depend on the degrees $k_i$ and continuously on the radius, and which may be different in each usage.  These polynomials will always occur as bounds on the greater side of an inequality, and it will not be important to know their exact form.
\subsubsection*{Leibniz polynomials}
Because they occur so often, we give special notation for a certain type of polynomial.  Whenever the notation
$$\LP(\|f_1\|_{k_1},\, \ldots,\, \|f_d\|_{k_d})$$
occurs it denotes a polynomial (with positive coefficients, which depends on the $k_i$ and continuously on the radius, and which may be different in each usage) such that each monomial term is as follows:

Each $\|f_i\|_\bullet$ occurs with degree at least $1$ (in some norm degree), and at most one of the $\|f_i\|$ has ``large'' norm degree $k_i$, while the other factors in the monomial have ``small'' norm degree $\fl{k_j/2}+1$, where $\fl{\,\cdot\,}$ denotes the integer part.

Equivalently, using the monotonicity in $k$ of $\|\cdot\|_k$, we can define $\LP$ in $Poly$ notation, as follows:
\begin{eqnarray}
& & A \leq \LP(\|f_1\|_{k_1},\, \ldots,\, \|f_d\|_{k_d}) \qquad\textnormal{if and only if}\notag \\
& & \qquad A \leq \sum_{i=1}^d \|f_i\|_{k_i}\,\times\, \|f_1\|_{\fl{k_1/2}+1}\, \ldots\, \widehat{\|f_i\|}_{\fl{k_i/2}+1}\, \ldots\cdot\, \|f_d\|_{\fl{k_d/1}+1}\, \notag \\
& & \qquad\qquad \times\, Poly(\|f_1\|_{\fl{k_1/2}+1},\, \ldots,\, \|f_d\|_{\fl{k_d/2}+1}),
\end{eqnarray}
where $\widehat{\|f_i\|}$ indicates this term is omitted from the product.  For example, we might say
$$\|f\|_k\, \|g\|_{\fl{k/2}+2}\, +\, \|f\|_{\fl{k/2}+1}\, \|g\|_{k+1}\, \|g\|_{\fl{k/2}+2} \;\leq\; \LP(\|f\|_k,\|g\|_{k+1}).$$

\begin{rem}
A typical example of how such terms arise is: to find the $C^k$-norm of a product of fields, we must differentiate $k$ times, applying the Leibniz rule iteratively. We get a polynomial in derivatives of the fields, and each monomial has at most one factor with more than $\fl{k/2}+1$ derivatives.  See Lemma \ref{bilinear estimate} for example, or \cite[II.2.2.3]{Hamilton} for a sharper estimate.
\end{rem}

We extend the definition to allow the entries in a Leibniz polyonmial to be polynomials themselves, eg.,
$$\|f\|_k\,\|h\|_{\fl{k/2}+1} \,+\, \|f\|_k\, \|g\|_{\fl{k/2}+1} + \|f\|_{\fl{k/2}+1}\, \|g\|_k \;\leq\; \LP\left(\|f\|_k,\, \|h\|_k+\|g\|_k\right).$$
In this case, we have used $\|h\|_k+\|g\|_k$ to indicate that not every monomial need have a factor of both $\|g\|$ and $\|h\|$.
\begin{lem}
Leibniz polynomials are closed under composition, e.g.,
$$\LP(\LP(\|f\|_a,\, \|g\|_b),\, \|h\|_c) \,\leq\, \LP(\|f\|_a,\, \|g\|_b,\, \|h\|_c)$$
\end{lem}

\begin{rem}
The approach in \cite{Hamilton} is to study \emph{tame} maps between tame Frechet spaces.  To say that a map is bounded by a Leibniz polynomial in its arguments is similar to the tameness condition.  However, rather than adapt this framework to SCI-spaces, we do as in \cite{MonnierZung} and \cite{MirandaMonnierZung}, working directly with bounding polynomials.
\end{rem}

\begin{rem}\label{radius dependence}
As noted in \cite{MonnierZung} and elsewhere, whether the coefficients of the polynomials vary continuously with the radius, or do not, makes no difference to the algorithm of Theorem \ref{Miranda-Monnier-Zung}, which ensures that all radii are between $R/2$ and $R$, over which we can find a radius-independent bound.
\end{rem}

\subsection*{SCI-groups}
We will give a definition of a group-like structure modelled on SCI-spaces, which is used in \cite{MirandaMonnierZung} to model local diffeomorphisms about a fixed point (and in our case to model local generalized diffeomorphisms); but first we feel we should give a conceptual picture to make the definition clearer:

Elements of an SCI-group will be identified with elements of an SCI-space, and we use the norm structure of the latter to express continuity properties of the former.  However, we do not assign any special meaning to the linear structure of the SCI-space---in particular, the SCI model-space for an SCI-group should not be viewed as its Lie algebra in any sense.  Furthermore, group elements will be defined at given radii, and their composition may be defined at yet a smaller radius---the amount by which the radius shrinks should be controlled by $\|\cdot\|_1$ of the elements (usually interpreted as a bound on their first derivative) and a fixed parameter for the group.

\begin{defn}\label{SCI group}
An \emph{SCI-group $\sr{G}$ modelled on an SCI-space $\sr{W}$} consists of elements which are formal sums
$$\phi = \Id + \chi,$$
where $\chi \in \sr{W}$, together with a \emph{scaled product} defined for some pairs in $\sr{G}$, i.e.:

There is a constant $c>1$ such that if $\phi$ and $\psi$ are in $G_r$ for some $r$ and
$$\|\phi-\Id\|_{1,r} \leq 1/c,$$
then,

(a) the \emph{product} $\psi \cdot \phi \in \sr G_{r'}$ is defined, where $r'=r(1-c\|\phi-\Id\|_{1,r})$; furthermore, the product operation commutes with restriction, and is associative modulo necessary restrictions, and

(b) there exists a \emph{scaled inverse} $\phi^{-1} \in \sr{G}_{r'}$ such that $\phi\cdot\phi^{-1} = \phi^{-1}\cdot\phi = \Id$ at radius $r''=r'(1-c\|\phi-\Id\|_{1,r})$.

Furthermore, for $k\geq1$ the following continuity conditions should hold:
\begin{eqnarray}
\|\psi^{-1}-\phi^{-1}\|_k &\leq& \LP(\|\psi-\phi\|_k,\,1+\|\phi-\Id\|_k) \label{inverse estimate}, \\[3pt]
\|\phi\cdot\psi - \phi\|_k &\leq& \LP(\|\psi-\Id\|_k,\, 1+\|\phi-\Id\|_{k+1}) \label{composition estimate 1} \\[3pt]
\textnormal{and} \qquad
\|\phi\cdot\psi - \Id\|_k &\leq& \LP(\|\psi-\Id\|_k + \|\phi-\Id\|_k). \label{composition estimate 2}
\end{eqnarray}
(As per the notational convention, these inequalities are taken at precisely those radii for which they make sense.)
\end{defn}

\begin{example}\label{local diffeo are SCI}
As in Example \ref{sections are SCI}, for each $0<r\leq1$ let $B_r \subset \R^n$ be the closed unit ball of radius $r$ centred at the origin, and let $\sr{W}_r$ be the space of $C^\infty$-maps from $B_r$ into $\R^n$ fixing the origin.  If $\chi$ is such a map, then by $\phi = \Id + \chi$ we mean the sum of $\chi$ with the identity map; then $\Id + \sr{W}_r$ forms an SCI-group under composition for some constant $c>1$.  These are the local diffeomorphisms. (See Lemma \ref{functions are SCI} and \cite{Conn} for details.)
\end{example}

\begin{rem}
Our definition of SCI-group is a bit different than that appearing in \cite{MirandaMonnierZung}.  Our continuity conditions look different, though, ignoring terms of norm degree $\fl{k/2}+1$, our conditions imply those in \cite{MirandaMonnierZung}.  (See Remark \ref{justify changes}.)
\end{rem}

\begin{defn}\label{SCI action}
A \emph{left (resp.\ right) SCI-action} of an SCI-group $\sr G$ on an SCI-space $\sr H$ consists of on operation $\phi\cdot f \in \sr H_{r'}$ for $\phi \in \sr G_r$ and $f \in \sr H_r$, which is defined whenever $r' \leq (1-c\|\phi-\Id\|_{1,r})r$ for some constant $c>1$, such that the following hold: the operation should commute with radius restriction, it should satisfy the usual left (resp. right) action law modulo radius restriction, and there should be some $s$ (called the \emph{derivative loss}) such that, for large enough $k$, for $\phi,\psi \in \sr{G}_r$ and $f,g \in \sr{H}_r$, the following continuity conditions hold:
\begin{eqnarray}
\|\phi\cdot f - \phi\cdot g\|_k &\leq& \LP(\|f-g\|_k,\, 1+\|\phi-\Id\|_{k+s}) \quad\textnormal{and} \label{action estimate 1} \\[3pt]
\|\psi\cdot f - \phi\cdot f\|_k &\leq& \LP\left(1+\|f\|_{k+s},\, \|\psi-\phi\|_{k+s},\, 1+\|\phi-\Id\|_{k+s}\right), \label{action estimate 2} 
\end{eqnarray}
(if these terms are well-defined).
\end{defn}

\begin{rem}
\eqref{action estimate 2} will ensure that if a sequence $\phi_1,\phi_2,\ldots$ converges, then so does $\phi_1\cdot f, \phi_2\cdot f,\ldots$.  Combining \eqref{action estimate 1} with \eqref{action estimate 2} for $f=0$ and $\psi=\Id$, we get 
\begin{equation}\label{nonlinear action estimate}
\|\phi\cdot g\|_k \;\leq\; \LP(\|g\|_k \,+\, \|\phi-\Id\|_{k+s})
\end{equation}
\end{rem}

\begin{rem}\label{action nonlinearity}
If the action is linear, we may equivalently simplify the hypotheses: we may discard $g$ entirely in \eqref{action estimate 1}, and, since each term will be first order in norms of $f$, we may replace $1+\|f\|_k$ with $\|f\|_k$ in \eqref{action estimate 2}; furthermore, in both estimates the polynomials will not have higher powers of $\|f\|$.  In \cite{MirandaMonnierZung}, our source for Theorem \ref{Miranda-Monnier-Zung}, only linear (and, in some sense, \emph{affine}) SCI-actions are considered.

Even considering this difference, our definition is a bit stronger than in \cite{MirandaMonnierZung}---as per our definition of Leibniz polynomials, $\LP$, we do not permit more than one factor of high norm degree in each monomial.
\end{rem}

\begin{example}\label{pullbacks are SCI}
The principal example of SCI-actions are local diffeomorphisms (Example \ref{local diffeo are SCI}) acting by pushforward or pullback on tensors, with derivative loss $s=1$.  See Section \ref{pushforwards and pullbacks} for details.
\end{example}

\subsection{Abstract normal form theorem}

The following theorem is adapted from \cite[Thm. 7.7]{MirandaMonnierZung}, with some changes.  The idea of applying the Nash-Moser fast convergence method to shrinking neighbourhoods of a point goes back to Conn's work in \cite{Conn}.  After the statement of the theorem, we give the interpretation of each SCI-space and map named in the theorem, as it applies to our situation, and show how the theorem may be used to prove our Main Lemma.  Then we address the differences between the theorem as we have presented it and as it appears in \cite{MirandaMonnierZung}.  A similar theorem is in \cite{MonnierZung}.

\begin{thm}\label{Miranda-Monnier-Zung}[MMZ]
Let $\sr{T}$ be an SCI-space, $\sr{F}$ an SCI-subspace of $\sr{T}$, and $\sr{I}$ a subset of $\sr{T}$ containing $0$.  Denote $\sr{N} = \sr{F} \cap \sr{I}$.  Let $\pi : \sr{T} \to \sr{F}$ be a projection commuting with restriction and inclusion, and let $\zeta = Id - \pi$.  Suppose that, for all $\epsilon \in \sr{T}$, and all $k \in \N$ sufficiently large,
\begin{equation}\label{MMZ1}
\|\zeta(\epsilon)\|_k \leq \LP(\|\epsilon\|_k).
\end{equation}

Let $\sr{G}$ be an SCI-group acting on $\sr{T}$, and let $\sr{G}^0 \subset \sr{G}$ be a closed subset of $\sr{G}$ preserving $\sr{I}$.

Let $\sr{V}$ be an SCI-space.  Suppose there exist maps
$$\sr{I} \oto{V} \sr{V} \oto{\Phi} \sr{G}^0$$
(with $\Phi(v)$ denoted $\Phi_v$) and $s \in \N$ such that, for every $\epsilon \in \sr{I}$, every $v, w \in \sr{V}$, and for large enough $k$,
\begin{eqnarray}
\|V(\epsilon)\|_k &\leq& \LP(\|\zeta(\epsilon)\|_{k+s},\, 1+\|\epsilon\|_{k+s}) \label{MMZ2} \\[4pt]
%%%
\|\Phi_v - Id\|_k &\leq& \LP(\|v\|_{k+s}), \label{MMZ3} \quad\mathnormal{and} \\[4pt]
%%%
\|\Phi_v\cdot \epsilon - \Phi_w\cdot \epsilon\|_k &\leq& \LP(\|v - w\|_{k+s},\, 1 + \|v\|_{k+s} + \|w\|_{k+s}+\|\epsilon\|_{k+s}) \notag \\
& & \;+\; \LP\left((\|v\|_{k+s} + \|w\|_{k+s})^2,\, 1+\|\epsilon\|_{k+s}\right) \label{MMZ4}
\end{eqnarray}

Finally, suppose there is a real positive $\delta$ such that for any $\epsilon\in\sr{I}$,
\begin{equation}
\|\zeta(\Phi_{V(\epsilon)}\, \cdot\, \epsilon)\|_k \leq \|\zeta(\epsilon)\|_{k+s}^{1+\delta}\, Poly\left(\|\epsilon\|_{k+s}, \|\Phi_{V(\epsilon)}-\Id\|_{k+s}, \|\zeta(\epsilon)\|_{k+s}, \|\epsilon\|_k\right) \label{MMZ5}
\end{equation}
where in this case the degree of the polynomial in $\|\epsilon\|_{k+s}$ does not depend on $k$.

Then there exist $l \in \N$ and two constants $\alpha>0$ and $\beta>0$ with the following property: if $\epsilon \in \sr{I}_R$ such that $\|\epsilon\|_{2l-1,R} < \alpha$ and $\|\zeta(\epsilon)\|_{l,R} < \beta$, there exists $\psi \in \sr{G}^0_{R/2}$ such that $\psi \cdot \epsilon \in \sr{N}_{R/2}$.
\end{thm}

\begin{rem}\label{interpretation of theorem}
In our case, the interpretation of the terms in this theorem will be as follows:
\begin{itemize}
\item $\sr{T}$ will be the space of deformations, $\Gamma(\^ ^2 L^*)$, of the standard generalized complex structure on $\C^n$.
\item $\sr{F} \subset \sr{T}$ will be the space of $(2,0)$-bivectors, the ``normal forms'' without the integrability condition---thus $\zeta(\epsilon) = \epsilon_2 + \epsilon_3$ is the non-bivector part of $\epsilon$, which we seek to eliminate.
\item $\sr{I}$ will be the integrable deformations, and thus $\sr{N} = \sr{F} \cap \sr{I}$ will be the holomorphic Poisson bivectors, i.e., the ``normal forms.''
\item $V$ produces a generalized vector field from a deformation.  As per Definition \ref{define V}, we will take $V(\epsilon) = P([\epsilon_1, P\epsilon_3] - \epsilon_2 - \epsilon_3)$.
\item $\sr{G}=\sr{G}^0$ will be the local generalized diffeomorphisms fixing the origin, acting on deformations as in Definition \ref{Courant action on deformation}, and $\Phi_v \in \sr{G}$ will be the time-1 flow of the generalized vector field $v$ as in Definition \ref{Courant flow}.
\end{itemize}

While estimates \eqref{MMZ2} through \eqref{MMZ4} in the hypotheses of the theorem may be understood as continuity conditions of some sort, estimate \eqref{MMZ5} expresses the fact that we have the ``correct'' algorithm, that is, each iteration will have a quadratically small error.
\end{rem}

\subsection{Proving the Main Lemma}\label{prove main lemma}
In Section \ref{verifying SCI} we verify that local generalized diffeomorphisms form a closed SCI-group, and that they act by SCI-action on the deformations.  In Section \ref{verifying hypotheses} we show that the other hypotheses of Theorem \ref{Miranda-Monnier-Zung}, estimates \eqref{MMZ1} through \eqref{MMZ5}, hold true for the interpretation above.  Thus, the theorem applies, and we conclude the following: if $\epsilon$ is a smooth, integrable deformation of the standard generalized complex structure in a neighbourhood of the origin in $\C^n$, and if $\|\epsilon\|_k$ is small enough (for some $k$ given by the theorem), then there is a local generalized diffeomorphism $\Psi$ fixing the origin such that $\zeta(\Psi\cdot\epsilon)=0$.  Then the Maurer-Cartan equations \eqref{MC1} and \eqref{MC2} tell us that $\Psi\cdot\epsilon$ is a holomorphic Poisson bivector, and thus the Main Lemma is proved.

\subsection{Sketch of the proof of Theorem \ref{Miranda-Monnier-Zung}}\label{sketch of MMZ proof}
The proof of Theorem \ref{Miranda-Monnier-Zung} is essentially in \cite[Appendix 1]{MirandaMonnierZung}, with the idea of the argument coming from \cite{Conn}.  We give a rough sketch of the argument as it appears in \cite{MirandaMonnierZung}.

We are given $\epsilon=\epsilon^0 \in \sr{I}_R$ and will construct a sequence $\epsilon^1,\epsilon^2,\ldots$.  We choose a sequence of smoothing parameters $t_d$, with $t_0>1$ (determined by the requirements of the proof) and $t_{d+1}=t_d^{3/2}$.  Then for $d>0$ let $v_d = S_{t_d} V(\epsilon^d)$, where $S_{t_d}$ is the smoothing operator, let $\Phi_{d+1} = \Phi_{v_d}$, and let $\epsilon^{d+1}=\Phi_{d+1}\cdot \epsilon^d$.  The $V(\epsilon^d)$ is smoothed before taking $\Phi_{d+1}$ so that we have some control over the loss of derivatives at each stage.

If $\|\epsilon\|_{2l-1}$ and $\|\zeta(\epsilon)\|_l$ are small enough, for certain $l$, and if $t_0$ is chosen carefully, then it will follow that the $\|\Phi_d\|_k$ approach zero quickly and the corresponding radii have lower bound $R/2$; by continuity properties of SCI-groups and -actions, the compositions $\Psi_{d+1} = \Phi_{d+1}\cdot\Psi_d$ will have a limit, $\Psi_\infty$, and the $\epsilon^d$ will have a limit, $\epsilon^\infty = \Psi_\infty\cdot \epsilon$.  Furthermore, it will follow that $\zeta(\epsilon^\infty) = \lim\zeta(\epsilon^d) = 0$, so $\epsilon^\infty\in\sr{N}$.

The above facts follow from two inductive lemmas.  The first fixes a norm degree, $l$, and an exponent, $A>1$, (determined by requirements of the proof) and proves inductively that for all $d\geq0$,
$$\begin{array}{llcl}
(1_d) & \|\Phi^{d+1}-\Id\|_{l+s} &<& t_d^{-1/2} \\
(2_d) & \|\epsilon^d\|_l &<& C\frac{d+1}{d+2} \\
(3_d) & \|\epsilon^d\|_{2l-1} &<& t_d^A \\
(4_d) & \|\zeta(\epsilon^d)\|_{2l-1} &<& t_d^A \\
(5_d) & \|\zeta(\epsilon^d)\|_l &<& t_d^{-1}
\end{array}$$
The second lemma uses the first to prove by induction on $k$ that, for all $k\geq l$, there is $d_k$ large enough such that for all $d\geq d_k$,
 $$\begin{array}{llcl}
(i) & \|\Phi^{d+1}-\Id\|_{k+s+1} &<& C_k t_d^{-1/2} \\
(ii) & \|\epsilon^d\|_{k+1} &<& C_k\frac{d+1}{d+2} \\
(iii) & \|\epsilon^d\|_{2k-1} &<& C_k t_d^A \\
(iv) & \|\zeta(\epsilon^d)\|_{2k-1} &<& C_k t_d^A \\
(v) & \|\zeta(\epsilon^d)\|_k &<& C_k t_d^{-1}
\end{array}$$

Given this setup, the proofs simply proceed in order through $1_d,\ldots,5_d,i,\ldots,iv$ by application of the hypotheses of Theorem \ref{Miranda-Monnier-Zung}, the continuity conditions for SCI-groups and SCI-actions, and the property of the smoothing operators.

\begin{rem}\label{justify changes}
The differences between the theorem as we have presented it and as it appears in \cite{MirandaMonnierZung} include notational and other minor changes, which we do not remark upon, and changes to the estimates coming from the nonlinearity of our action.  Our estimates imply theirs if any instance of $\|\epsilon\|_p$, $\|\zeta(\epsilon)\|_p$ or $\|\Phi-\Id\|_p$ is replaced with the nonlinear $\LP(1+\|\epsilon\|_p)$, $\LP(\|\zeta(\epsilon)\|_p)$ or $\LP(\|\Phi-\Id\|_p)$ respectively.  But this is not a problem---the estimates are \emph{locally equivalent} (we will be precise), and thus are valid over the sequence defined above.

To see why, we note that in the lemmas, $\|\epsilon^d\|_p$ only appears with $p\leq2l-1$ in the first case or $p\leq2k-1$ in the second.  But then
$$\LP(1+\|\epsilon^d\|_p) \,=\, (1+\|\epsilon\|_p)\,Poly(\|\epsilon^d\|_{\fl{p/2}+1}) \,\leq\, (1+\|f\|_p)\,Poly(\|\epsilon^d\|_l)$$
(and respectively for $k$.)  But the inductive hypothesis has that $\|\epsilon^d\|_l$ (resp. $\|\epsilon^d\|_k$) is bounded by a constant, so this extra factor does no harm.  Similarly, the higher-order terms on, eg., $\|\zeta(\epsilon^d)\|_{2l-1}$ are vanishingly small by the inductive hypothesis.

The remaining concern, then, is for $1+\|\epsilon\|_p$ in place of $\|\epsilon\|_p$.  This is already dealt with implicitly in the \emph{affine} version of the theorem in \cite{MirandaMonnierZung}: the space $\sr{T}$ may be embedded affinely in $\C \dsum \sr{T}$, by $\epsilon \mapsto (1,\epsilon)$, with the norm $\|(1,\epsilon)\|_p = 1 + \|\epsilon\|_p$.  The constraint $\alpha$ in the hypothesis, $\|\epsilon\|_{2l-1}\leq\alpha$, in the original theorem can always be chosen greater than $1$, so in the affine context we simply require that $\|\epsilon\|_{2l-1} \,\leq\, \alpha'=\alpha-1$.
\end{rem}

\section{Verifying the SCI estimates}\label{verifying SCI}

In this section we explain how the particular objects named in Remark \ref{interpretation of theorem} satisfy the SCI definitions.

\subsection{Norms}\label{norm section}

Throughout this \thesis{chapter}\article{paper} we will use various kinds of $C^k$ $\sup$-norms.

\begin{defn}\label{norms 1}
Let $X \in \R^q$ or $\C^q$.  $X_i$ is the $i$-th component.  Then let
$$\|X\| = \sup_i |X_i|.$$
Similarly, if $A=[a_{ij}]$ is an $n\times n$ matrix, let $\|A\| = \sup_{i,j} |a_{ij}|$.
\end{defn}
\begin{rem}\label{matrix inverse estimate}
Comparing our matrix norm to the operator norm $\|\cdot\|_{op}$, we have
$$\|A\| \,\leq\, \|A\|_{op} \,\leq\, n\,\|A\|.$$
Then if $\|A-\Id\| \,\leq\, \frac{1}{2n}$, $A$ is invertible and
$$\|A^{-1}\| \,\leq\, 2.$$
\end{rem}

\begin{defn}\label{norms}
Suppose now that $f$ is a vector-valued function, $f : U \to V$, where $U \subset \R^n$ or $\C^n$ and $V$ is a normed finite-dimensional vector space.  Then let
$$\|f\|_0 = \sup_{x\in U} \|X(x)\|.$$

Suppose furthermore that $f$ is smooth.  If $\alpha$ is a multi-index, then $f^{(\alpha)}$ is the corresponding higher-order partial derivative.  If $k$ is a non-negative integer, then $f^{(k)}$ is an array containing the terms $f^{(\alpha)}$ for $|\alpha|=k$.  Let
\begin{equation*}
\|f\|_k = \sup_{|\alpha| \leq k} \|f^{(\alpha)}\|_0 = \sup_{j \leq k} \|f^{(k)}\|_0.
\end{equation*}
\end{defn}

\begin{rem}
Ultimately, we will always be working on the manifold $\C^n$ or a subset thereof.  Using the standard trivialization of the tangent and cotangent bundles, Definitions \ref{norms 1} and \ref{norms} give us a nondecreasing family of norms, $\|\cdot\|_k$, on smooth tensor fields on subsets of $\C^n$.  This applies to generalized vector fields, $B$-fields, and higher rank tensors (including deformations in $\^ ^2 L^*$).  However, for technical reasons, we must use a slightly unusual norm for Courant automorphisms:
\end{rem}

\begin{defn}
If $\Phi = (B,\phi)$ is a Courant automorphism, then we usually only take norms of $\Phi - \Id = (B,\phi-\Id)$.  Considering $\phi-\Id$ as just a function from a subset of $\C^n$ to $\C^n$, let $\|\Phi-\Id\|_0 = \sup(\|B\|_0,\,\|\phi-\Id\|_0)$; but if $k\geq1$, then let
\begin{equation}
\|\Phi-\Id\|_k = \sup(\|B\|_{k-1},\,\|\phi-\Id\|_k).
\end{equation}
The difference in degree between $B$ and $\phi$ reflects the fact that $B$ acts on derivatives while $\phi$ acts on the underlying points of the manifold.
\end{defn}

%\begin{rem}
%We will use frequently without remark the fact that our norms increase in the degree, $k$.  Also note that each component ($B$-field and diffeomorphism) of a generalized diffeomorphism is smaller than the generalized diffeomorphism itself, and that each component (vector and 1-form) of a generalized vector field is smaller than the generalized vector field itself.
%\end{rem}

\subsection{Pushforwards and pullbacks}\label{pushforwards and pullbacks}

As mentioned in Examples \ref{local diffeo are SCI} and \ref{pullbacks are SCI},
\begin{lem}\label{functions are SCI}
Local diffeomorphisms from the closed balls $B_r \subset \R^n$ to $\R^n$ fixing the origin form an SCI-group under composition (see Example \ref{local diffeo are SCI}) with constant $c=2n$.  Furthermore, the pullback action, $\phi^*f = f\comp\phi$, of a local diffeomorphism $\phi$ on a function $f:B_r\to\C^p$, is a right SCI-action, with derivative loss $s=1$.
\end{lem}
As we said earlier, our definitions of SCI-group and SCI-action are slightly different from \cite{MirandaMonnierZung}: for SCI-groups, \eqref{inverse estimate} is stronger, and \eqref{composition estimate 1} and \eqref{composition estimate 2} have the same first-order behaviour in each group element, while having possibly higher-order terms (but only in $\fl{k/2}+1$ norms).  For SCI-actions, \eqref{action estimate 1} and \eqref{action estimate 2} are nonlinear counterparts to the conditions in \cite{MirandaMonnierZung}.
\begin{proof}
The proof of Lemma \ref{functions are SCI}, including the existence of compositions and inverses at the correct radii and the various continuity estimates, is essentially in \cite{Conn} (and, eg., \cite{MonnierZung}), with minor differences as noted.  We show only the proof of \eqref{inverse estimate}---the SCI-group continuity estimate for inverses---since it gives the flavour of the proofs of the other estimates.

Let $\phi$ and $\psi$ be local diffeomorphisms.  We proceed by induction on the degree of the norm.  If $\alpha=(\alpha_1,\ldots,\alpha_n)$ is a multi-index, we denote the $\alpha$-order partial derivative $D_\alpha$.  Suppose that \eqref{inverse estimate} holds for degree less than $k$, that is, whenever $|\alpha| < k$,
$$\|D_\alpha(\phi^{-1}-\psi^{-1})\|_0 \,\leq\, \LP(\|\phi-\psi\|_{|\alpha|},1+\|\psi-\Id\|_{|\alpha|}).$$
(This certainly holds for $|\alpha|=0$, given the hypothesis that $\|\phi-\Id\|_1 \leq 1/2n$ and likewise for $\psi$.)

Now suppose that $|\alpha| = k$.  We have the trivial identity
\begin{equation}\label{e80}
0 \;=\; D_\alpha\left((\phi-\psi)\comp\phi^{-1}\right) + D_\alpha\left(\psi\comp\phi^{-1} - \psi\comp\psi^{-1}\right)
\end{equation}

To compute the derivative $D_\alpha(\psi\comp\phi^{-1} - \psi\comp\psi^{-1})$ at $x\in B_r$, we make repeated applications of the chain rule and Leibniz rule, so that we have a sum of terms each of which has the form, for some $|\beta|\leq|\alpha|$,
\begin{equation}\label{e81}
D_\beta\psi|_{\phi^{-1}(x)}\cdot Q_\beta\left(\phi^{-1}\right)|_x - D_\beta\psi|_{\psi^{-1}(x)}\cdot Q_\beta\left(\psi^{-1}\right)|_x,
\end{equation}
where $Q_\beta(\phi^{-1})$ is a polynomial expression in derivatives of $\phi^{-1}$ up to order $\left.|\alpha|+1-|\beta|\right.$ (and likewise for $Q_\beta(\psi^{-1})$).  We remark that each term like \eqref{e81} will have at most one factor with higher derivatives than $\fl{k/2}+1$. Equivalently, \eqref{e81} is
\begin{eqnarray*}
D_\beta\psi|_{\psi^{-1}(x)}\cdot \left.\left(Q_\beta(\phi^{-1}) - Q_\beta(\psi^{-1})\right)\right|_x
\;+\; D_\beta(\psi|_{\phi^{-1}(x)} - \psi_|{\psi^{-1}(x)})\cdot Q_\beta\left(\phi^{-1}\right)|_x
\end{eqnarray*}

When $|\beta|=1$, this is, for some $i$,
\begin{equation}\label{e002}
\frac{\del}{\del x_i} \psi|_{\psi^{-1}(x)} \cdot D_\alpha(\phi^{-1}-\psi^{-1})|_x \;+\; D_\beta(\psi|_{\phi^{-1}(x)} - \psi|_{\psi^{-1}(x)})\cdot D_\alpha\phi^{-1}|_x.
\end{equation}
Thus, by inverting the matrix $a_{ij}=\frac{\del}{\del x_i} \psi_j|_{\psi^{-1}(x)}$ and applying it to the first term of expressions of the form \eqref{e002}, we may solve \eqref{e80} for $D_\alpha(\phi^{-1}-\psi^{-1})|_x$.  The solution will be a polynomial in:
\begin{itemize}
\item a matrix inverse whose norm is bounded by $2$
$$(\textnormal{since}\; \left\|\frac{\del}{\del x_i} \psi_j|_{\psi^{-1}(x)}-\Id\right\| \leq \frac{1}{2n}).$$
\item derivatives of $\psi$ up to order $k$,
\item derivatives of $\phi^{-1}-\psi^{-1}$ up to order $k-1$,
\item derivatives of $\phi-\psi$ up to order $k$,
\item derivatives of $\phi^{-1}$ up to order $k$, and
\end{itemize}

What if we compute the norm of this solution?  In the special case where $\phi=\Id$, by applying the induction hypothesis and combining the Leibniz polynomials we obtain that $\|\psi^{-1}-\Id\|_k \,\leq\, \LP(\|\psi-\Id\|_k)$.   Returning to the general case, this gives us the bound on $\|\phi^{-1}-\Id\|_k$ that we need to complete the proof.
\end{proof}

\begin{lem}\label{vector bundle is SCI}
Let $E=B_r \times V$ be a trivial rank-$n$ vector bundle over the closed ball $B_r \subset \C^n$, for each $0<r\leq1$.  Then the the vector bundle automorphisms covering the identity, $\Aut(E)$, form an SCI-group with constant $c=2n$, and act by SCI-action on the sections, $\Gamma(E)$, with derivative loss $s=0$.
\end{lem}
The necessary estimates follow in a straightforward way from the estimates for functions (Lemma \ref{functions are SCI}) applied to matrix calculations, using Remark \ref{matrix inverse estimate} for the inverse estimate \eqref{inverse estimate}.

The following lemma tells us that an action will be SCI if it is composed of SCI-actions in a certain sense.  In fact, we don't use any of the algebraic structure of actions.
\begin{lem}\label{composite action}
Let $\sr{A}$, $\sr{B}$, $\sr{G}$ and $\sr{H}$ be SCI-spaces, where $\sr{A}$, $\sr{B}$ and $\sr{G}$ each have a distinguished element $\Id$, let
$$\cdot:\sr{A}\times\sr{H} \to \sr{H} \quad\textnormal{and}\quad \cdot:\sr{B}\times\sr{H} \to \sr{H}$$
be operations satisfying estimates \eqref{action estimate 1} and \eqref{action estimate 2} with derivative loss $s_1$ and $s_2$ respectively (no other SCI-action structure is assumed), and let
$$\cdot:\sr{G}\times\sr{H} \to \sr{H}$$
be an operation such that, for each $\phi \in \sr{G}$ there are $\phi_A \in \sr{A}$ and $\phi_B \in \sr{B}$ (with $\Id_A=\Id$ and $\Id_B=\Id$) such that for each $h \in H$,
$$\phi\cdot h = \phi_A\cdot(\phi_B\cdot h).$$
Finally, suppose there is an $s_3$ such that for any $\phi,\psi \in \sr{G}$ and large enough $k$,
\begin{equation}\label{action comparison}
\|\phi_A-\psi_A\|_k \,\leq\, \LP(\|\phi-\psi\|_{k+s_3}) \quad\textnormal{and}\quad
\|\phi_B-\psi_B\|_k \,\leq\, \LP(\|\phi-\psi\|_{k+s_3}).
\end{equation}
Then the operation of $\sr{G}$ on $\sr{H}$ also satisfies estimates \eqref{action estimate 1} and \eqref{action estimate 2} with derivative loss $s_1+s_2+s_3$.
\end{lem}
\begin{proof}
If $\phi\in\sr{G}$ and $f,g \in \sr{H}$, we apply estimate \eqref{action estimate 1} for the actions of $\sr{A}$ and $\sr{B}$:
\begin{eqnarray*}
\|\phi\cdot f - \phi\cdot g\|_k &=& \|\phi_A\cdot(\phi_B\cdot f) - \phi_A\cdot(\phi_B\cdot g)\|_k \\
&\leq& \LP(\|\phi_B\cdot f - \phi_B\cdot g)\|_k,\, 1+\|\phi_A-\Id\|_{k+s}) \\
&\leq& \LP(\LP(\|f-g\|_k,\, 1+\|\phi_B-\Id\|_{k+s}),\, 1+\|\phi_A-\Id\|_{k+s})
\end{eqnarray*}
Composing the Leibniz polynomials and using \eqref{action comparison} for $\|\phi_A-\Id\|$ and $\|\phi_B-\Id\|$, we see that estimate \eqref{action estimate 1} holds for the action of $\sr{G}$.

If $\phi,\psi \in \sr{G}$ and $f \in \sr{H}$, then
\begin{eqnarray}
\|\psi\cdot f - \phi\cdot f\|_k &=& \|\psi_A\cdot(\psi_B\cdot f) \,-\, \phi_A\cdot(\phi_B\cdot f)\|_k \notag \\
&\leq& \|\psi_A\cdot(\psi_B\cdot f) - \phi_A\cdot(\psi_B\cdot f)\|_k \label{e222} \\
& & \qquad\qquad +\; \|\phi_A\cdot(\psi_B\cdot f) - \phi_A\cdot(\phi_B\cdot f)\|_k. \label{e223}
\end{eqnarray}
Similarly to above, we apply estimate \eqref{action estimate 2} to line \eqref{e222} and estimate \eqref{action estimate 1} to line \eqref{e223}, and then vice versa, followed by the estimates \eqref{action comparison}, and we see that \eqref{action estimate 2} holds for the action of $\sr{G}$, with total derivative loss $s_1 + s_2 + s_3$.
\end{proof}

\begin{lem}\label{pushforward and pullback are SCI}
The action of local diffeomorphisms by pushforward or by pullback on tensors constitutes an SCI-action with derivative loss $s=1$.
\end{lem}
\begin{proof}
If $\phi:B_r\to\R^n$ is a local diffeomorphism with $\|\phi-\Id\|_1\leq1/2n$ and $v:B_r \to TB_r \iso B_r\times\R^n$ is a vector field, then the pushforward of $v$ by $\phi$ may be decomposed as
$$\phi_*v = (D\phi\cdot v)\comp\phi^{-1},$$
where the derivative $D\phi$ is treated as a matrix-valued function, acting on $v$ by multiplication.  Similarly, if $\theta:B_r \to T^*B_r \iso B_r\times\R^n$ is a 1-form, then the pushforward of $\theta$ may be written
$$\phi_*\theta = ((D\phi^T)^{-1}\cdot\theta)\comp\phi^{-1},$$
where the $(D\phi^T)^{-1}$ is the matrix transpose and inverse at each point.  We regard $D\phi\cdot v$ and $(D\phi^T)^{-1}\cdot\theta$ as functions from $B_r$ to $\R^n$, in which case precomposotion by $\phi^{-1}$ acts by SCI-action with derivative loss $s=1$; and $D\phi$ and $D\phi^{-1}$ are automorphisms of the vector bundle $B_r \times \R^n$, and thus act by SCI-action with derivative loss $s=0$.  If $\psi$ is another local diffeomorphism then
$$\|D\psi-D\phi\|_{k-1} \,\leq\, \|\psi-\phi\|_k,$$
and
$$\|(D\psi^T)^{-1}-(D\phi^T)^{-1}\|_{k-1}  \,\leq\, \LP(\|D\psi-D\phi\|_{k-1}),$$
so by taking a degree-shifted norm on the $D\psi-D\phi$, we are in the case of Lemma \ref{composite action}.

A similar argument works for pullbacks, and for higher-rank tensors.
\end{proof}

\subsection{Estimates of Courant actions}

\begin{lem}
Local generalized diffeomorphisms on the balls $B_r\subset\C^n$ form an SCI-group.
\end{lem}
\begin{proof}
Recall (Definition \ref{Courant automorphism}) that a local generalized diffeomorphism $\Phi$ may be represented $(B,\phi)$, where $B$ is a closed 2-form and $\phi$ is a local diffeomorphism.  If $\Psi=(B',\psi)$ is another local generalized diffeomorphism, then
$$\Phi\comp\Psi = (\psi^*B+B',\phi\comp\psi) \qquad\textnormal{and}\qquad \Phi^{-1} = (-(\phi^{-1})^*B,\phi^{-1}).$$
We already know that local diffeomorphisms form an SCI-group, and
$$r(1-c\|\Phi-\Id\|_{1,r}) \leq r(1-c\|\phi-\Id\|_{1,r}),$$
thus products and inverses exist at precisely the radii required in the definition.  Furthermore, estimates \eqref{inverse estimate}, \eqref{composition estimate 1} and \eqref{composition estimate 2} will be satisfied for the diffeomorphism term, thus we only need to check them for the $B$-field term.

We first bound $\Phi^{-1}-\Psi^{-1}$ (estimate \eqref{inverse estimate}). Recall that the norm degree is shifted for the $B$-field term.  Since pushforward is an SCI-action with derivative loss 1, we may use the continuity estimates for SCI-actions:
\begin{eqnarray*}
& & \|(\phi^{-1})^*B - (\psi^{-1})^*B'\|_{k-1} \\
&=& \|\phi_*B - \psi_*B'\|_{k-1} \\
&\leq& \|\phi_*B-\phi_*B')\|_{k-1} + \|\phi_*B' - \psi_*B'\|_{k-1} \\
&\leq& \LP(\|B-B'\|_{k-1},\, 1+\|\phi-\Id\|_k) \;+\; \LP(\|B'\|_k,\, \|\phi-\psi\|_k,\, \|\psi-\Id\|_k) \\
&\leq& \LP(\|\Phi-\Psi\|_k,\, 1+\|\Phi-\Id\|_k) \;+\; \LP(\|\Psi-\Id\|_{k+1},\, \|\Phi-\Psi\|_k,\, \|\Psi-\Id\|_k)
\end{eqnarray*}
We use $1+\|\Phi-\Id\|_k \,\leq\, 1+\|\Psi-\Id\|_k+\|\Phi-\Psi\|_k$ and combine the Leibniz polynomials to get estimate \eqref{inverse estimate}.

Now we bound $\Phi\comp\Psi-\Phi$ (estimate \eqref{composition estimate 1}).  We use estimate \eqref{action estimate 2} for pullbacks on the second line:
\begin{eqnarray*}
\|\psi^*B+B'-B\|_{k-1} &\leq& \|\psi^*B-B\|_{k-1} \,+\, \|B'\|_{k-1} \\
&\leq& \LP(1+\|B\|_k,\, \|\psi-\Id\|_k) \,+\, \|B'\|_{k-1} \\
&\leq& \LP(1+\|\Phi-\Id\|_{k+1},\, \|\Psi-\Id\|_k) \,+\, \|\Psi-\Id\|_k \\
&\leq& \LP(1+\|\Phi-\Id\|_{k+1},\, \|\Psi-\Id\|_k)
\end{eqnarray*}

Finally, we bound $\Phi\comp\Psi-\Id$ (estimate \eqref{composition estimate 2}).  We use estimate \eqref{action estimate 1} on the second line:
\begin{eqnarray*}
\|\psi^*B + B'-0\|_{k-1} &\leq& \|\psi^*B\|_{k-1} \,+\, \|B'\|_{k-1} \\
&\leq& \LP(1+\|\psi-\Id\|_k,\, \|B\|_{k-1}) \,+\, \|B'\|_{k-1} \\
&\leq& \LP(1+\|\Psi-\Id\|_k,\, \|\Phi-\Id\|_k) \,+\, \|\Psi-\Id\|_k
\end{eqnarray*}
and the result follows.
\end{proof}

\begin{rem}
Regarding closedness, we consider a $C^\infty$-convergent sequence of local generalized diffeomorphisms,
$$\lim_{n\to\infty} (B_n,\phi_n) = (\lim_{n\to\infty} B_n, \lim_{n\to\infty} \phi_n) = (B,\phi).$$
Since local diffeomorphisms are closed, $\phi$ is a local diffeomorphism; if each $dB_n=0$ then, since the convergence is $C^\infty$, $dB=0$; thus $(B,\phi)$ is a local generalized diffeomorphism.  So the local generalized diffeomorphisms are closed.
\end{rem}

\begin{lem}
The action of local generalized diffeomorphisms on the deformations, $\Gamma(\^ ^2 L^*)$, of the standard generalized complex structure on $B_r \subset \C^n$, as in Definition \ref{Courant action on deformation}, is a left SCI-action.
\end{lem}
\begin{proof}
Since this action is defined over precisely the same $B_r$ as pushforward by local diffeomorphisms, we need only check the estimates \eqref{action estimate 1} and \eqref{action estimate 2}.  We consider $\epsilon \in \Gamma(L)$ as a map from $L \to \bar{L}$.  A section of $L_{\Phi\cdot\epsilon}$ is uniquely represented as $u + (\Phi\cdot\epsilon)(u)$, for some $u \in \Gamma(L)$.  By definition, this is also the image of $v + \epsilon(v)$ under $\Phi_*$, for some $v \in \Gamma(L)$.  Then
\begin{eqnarray*}
u + (\Phi\cdot\epsilon)(u) &=& \Phi_*(v+\epsilon(v)) \\
&=& \left(\Phi_*\comp(\Id+\epsilon)\right)^L_L(v) + \left(\Phi_*\comp(\Id+\epsilon)\right)^{\bar{L}}_L(v)
\end{eqnarray*}
The $L$ and $\bar{L}$ components agree, so $v=\left(\left(\Phi_*\comp(\Id+\epsilon)\right)^L_L\right)^{-1}(u)$ and
\begin{equation}\label{action formula}
(\Phi\cdot\epsilon)(u) \,=\, \left(\Phi_*\comp(\Id+\epsilon)\right)^{\bar{L}}_L \comp \left(\left(\Phi_*\comp(\Id+\epsilon)\right)^L_L\right)^{-1}(u)
\end{equation}

Similarly to the proof of Lemma \ref{pushforward and pullback are SCI}, we interpret $u$ as a map from $B_r$ to $L\subset B_r \times \C^{2n}$, and we may write \eqref{action formula} in terms of fibrewise linear maps and postcomposition by diffeomorphisms.

Let
$$D\Phi : B_r \to \Aut(\C \tens \T_0B_r \dsum T^*_0B_r) \iso \Aut(\C^n\dsum\C^n)$$
be the trivialization of fibrewise action of the Courant automorphism $\Phi$.  If $\Phi=(B,\phi)$, then $D\Phi$ acts on each fibre, $\C\tens T_x\dsum T^*_x \iso \C^n\dsum \C^n$, by $(D\phi(x) \dsum (D\phi(x)^T)^{-1}) \cdot (1+B|_x)$; then, as a simple example, as in Lemma \ref{pushforward and pullback are SCI} we see that
$$\Phi_*u = (D\Phi)\cdot(u\comp\phi^{-1}).$$
Now for the case of \eqref{action formula} we may write
\begin{eqnarray*}
(\Phi\cdot\epsilon)(u) &=& \left(D\Phi\cdot(\Id+\epsilon)\right)^{\bar{L}}_L \cdot \left(\left(\left(D\Phi\cdot(\Id+\epsilon)\right)^L_L\right)^{-1}\comp\phi^{-1}\right)\cdot u
\end{eqnarray*}

We see that $\Phi\cdot\epsilon$ is constructed from $\epsilon$, $\Phi$, $D\Phi$ and $\Id$ through the operations of sum, matrix multiplication and matrix inverse, pushforward by functions, and restriction and projection (to $L$ and $\bar{L}$).  Each of these operations is an SCI-action in the weak sense of Lemma \ref{composite action}, so the result follows.
\end{proof}

\section{Checking the hypotheses of the abstract normal form theorem}\label{verifying hypotheses}

\subsection{Preliminary estimates}

\begin{lem}\label{bilinear estimate}
If $\Theta : V_1 \times V_2 \to W$ is a bilinear function between normed finite-dimensional vector spaces, and $f : U \to V_1$ and $g : U \to V_2$ are smooth on a compact domain $U$, then, applying $\Theta$ to $f$ and $g$ pointwise,
$$\|\Theta(f,g)\|_k \leq C (\|f\|_k \|g\|_0 + \|f\|_0 \|g\|_k) \leq C'\, \|f\|_k \|g\|_k,$$
and of course, $\|\Theta(f,g)\|_k \leq \LP(\|f\|_k,\|g\|_k)$.
\end{lem}
\begin{proof}
As remarked in Proposition \ref{interpolation inequality}, as a consequence of the existence of smoothing operators on spaces of smooth functions, the \emph{interpolation inequality} holds---for nonnegative integers $p \geq q \geq r$ and any function $f$ as above,
$$\|f\|_q^{p-r} \leq C \|f\|_r^{p-q} \|f\|_p^{q-r}.$$
From this inequality, the result follows by a standard argument (see \cite[Cor. II.2.2.3]{Hamilton}). 
\end{proof}

\begin{lem}\label{bracket estimate}
If $\alpha\in \Gamma(\^ ^i L^*)$ and $\beta\in \Gamma(\^ ^j L^*)$, then for $k\geq0$,
$$\|[\alpha,\beta]\|_k \,\leq\, C \left(\|\alpha\|_{k+1} \|\beta\|_1 + \|\alpha\|_1 \|\beta\|_{k+1}\right) \,\leq\, C' \|\alpha\|_{k+1} \|\beta\|_{k+1},$$
and of course, $\|[\alpha,\beta]\|_k \,\leq\, \LP(\|\alpha\|_{k+1},\, \|\beta\|_{k+1})$.
\end{lem}
\begin{proof}
If $\alpha$ and $\beta$ are generalized vector fields, there are pointwise-bilinear functions $\Theta$ and $\Lambda$ which express the Courant bracket formula \eqref{bracket formula} as
$$[\alpha,\beta] = \Theta(\alpha,\beta^{(1)}) - \Lambda(\beta,\alpha^{(1)}).$$
Then by Lemma \ref{bilinear estimate},
\begin{eqnarray*}
\|[\alpha,\beta]\|_k &\leq& C' (\|\alpha\|_k \|\beta^{(1)}\|_0 + \|\alpha\|_0 \|\beta^{(1)}\|_k \\
& & +\; \|\alpha^{(1)}\|_0 \|\beta\|_k + \|\alpha^{(1)}\|_k \|\beta\|_0 ) \\
&\leq& C \left(\|\alpha\|_{k+1} \|\beta\|_1 + \|\alpha\|_1 \|\beta\|_{k+1}\right)
\end{eqnarray*}
If $\alpha$ and $\beta$ are higher-rank tensors and the bracket is the generalized Schouten bracket, a suitable choice of $\Theta'$ and $\Lambda'$ will give the same result.
\end{proof}

%For a multilinear product, we have a similar result to Lemma \ref{bilinear estimate}:
%\begin{lem}\label{multilinear estimate}
%If $\Theta : V_1 \times \ldots \times V_n \to W$ is a multilinear function between normed, finite-dimensional vector spaces, and if, for each $1 \leq i \leq n$, $f_i : U \to V_i$ is smooth on a compact domain $U$, then, applying $\Theta$ pointwise,
%\begin{eqnarray*}
%\|\Theta(f_1,\ldots,f_n)\|_k &\leq& C\, \sum_{i=1}^n \|f_i\|_k\, \cdot\, \|f_1\|_0\, \ldots\, \widehat{\|f_i\|_0}\, \ldots\, \|f_n\|_0 \\
%&\leq& \LP(\|f_1\|_k,\ldots,\|f_n\|_k)
%\end{eqnarray*}
%\end{lem}
%\begin{proof}
%This is just an iteration of Lemma \ref{bilinear estimate}; eg., a trilinear function $\Theta(f,g,h)$ may be %replaced by a bilinear function $\Theta'(f, g \tens h)$, and then $g \tens h$ is itself bilinear.
%\end{proof}

\subsection{Verifying estimates \eqref{MMZ1}, \eqref{MMZ2}, \eqref{MMZ3} and \eqref{MMZ4}}

Recall that if $\epsilon = \epsilon_1 + \epsilon_2 + \epsilon_3 \in \Gamma(\^ ^2 L^*)$, with the terms being a bivector, a mixed term and 2-form respectively, then $\zeta(\epsilon) = \epsilon_2 + \epsilon_3$.  Then the following is an obvious consequence of our choice of norms.

\begin{lem}[Estimate \ref{MMZ1}]\label{projection estimate}
For all $\epsilon \in \Gamma(\^ ^2 L^*)$ and any $k$, $\|\zeta(\epsilon)\|_k \,\leq\, \|\epsilon\|_k$.
\end{lem}

We recall following estimate, taken from \cite{NijenhuisWoolf}, which was mentioned in Lemma \ref{existence of P}:
\begin{lem}\label{P estimate}
For all $\epsilon \in \Gamma(\^ ^2 L^*)$ and any $k$, $\|P\epsilon\|_k \,\leq\, C\,\|\epsilon\|_k$.
\end{lem}

\begin{lem}[Estimate \ref{MMZ2}]\label{V estimate}
For any $\epsilon \in \Gamma(\^ ^2 L^*)$ and large enough $k$,
$$\|V(\epsilon)\|_k \leq C\, \|\zeta(\epsilon)\|_{k+1}\, (1+\|\epsilon\|_{k+1}).$$
\end{lem}
\begin{proof}
$V(\epsilon) = P[\epsilon_1,P\epsilon_3] - P\zeta(\epsilon)$.  But $\|\epsilon_1\|_k \leq \|\epsilon\|_k$ and $\|\epsilon_3\|_k \leq \|\zeta(\epsilon)\|_k$, so by applying the triangle inequality and then Lemmas \ref{P estimate} and \ref{bracket estimate} the result follows.
\end{proof}

\begin{lem}[Estimate \ref{MMZ3}]\label{flow estimate}
For any $v \in \Gamma(L^*)$, any $0\leq t\leq1$ and large enough $k$,
$$\|\Phi_{tv} - Id\|_k \leq \LP(\|v\|_k)$$
\end{lem}
\begin{proof}
Let $v = X+\xi$, where $X$ is a vector field and $\xi$ is a 1-form, and let $\Phi_{tv} = (B_{tv},\phi_{tX})$.  From \cite{MirandaMonnierZung} we know that a counterpart of this lemma holds for the local diffeomorphism $\phi_{tX}$, therefore we are only concerned with $B_{tv}$. By the SCI-action estimate \eqref{action estimate 1} for pullbacks of differential forms,
\begin{equation}\label{e300}
\|\phi_{tX}^* d\xi\|_{k-1} \,\leq\, \LP(\|\xi\|_k,\, 1+\|\phi_{tX}-\Id\|_k)
\end{equation}
The counterpart of this Lemma in \cite{MirandaMonnierZung} tells us that
$$\|\phi_{tX}-\Id\|_k \leq \LP(\|X\|_k).$$
We plug this into \eqref{e300} and recall that $\|v\|_k = \sup(\|X\|_k,\,\|\xi\|_k)$; then,
$$\|\phi^*_{tX} d\xi\|_{k-1} \,\leq\, \LP(\|v\|_k).$$
But
\begin{eqnarray*}
\|B_{tv}\|_{k-1} &=& \left\| \int_0^t (\phi^*_{\tau X} d\xi)\, d\tau\right\|_{k-1} \\
&\leq& \int_0^t \left\|\phi^*_{\tau X} d\xi\right\|_{k-1}\, d\tau \\
&\leq& \int_0^t \LP(\|v\|_k)\, d\tau
\end{eqnarray*}
and the result follows.
\end{proof}

\begin{lem}[Estimate \ref{MMZ4}]
There is some $s$ such that, for any $v,w \in \Gamma(L^*)$, any integrable deformation $\epsilon \in \Gamma(\^ ^2 L^*)$, and large enough $k$,
\begin{eqnarray}
\|\Phi_v\cdot \epsilon - \Phi_w\cdot\epsilon\|_k &\leq& \LP(\|v - w\|_{k+1},\, 1 + \|v\|_{k+2} + \|w\|_{k+2}+\|\epsilon\|_{k+1}) \notag \\
& & \;+\; \LP\left((\|v\|_{k+3} + \|w\|_{k+3})^2,\, 1+\|\epsilon\|_{k+2}\right) \label{e302}
\end{eqnarray}
\end{lem}
\begin{proof}
The integral form of Proposition \ref{infinitesimal flow} tells us that
\begin{eqnarray*}
& &\Phi_v \cdot \epsilon - \Phi_w \cdot \epsilon \\
&=& \int_0^1 \left(\bar\del v + [v,\phi_{tv}\cdot\epsilon]\right)dt \,-\, \int_0^1 \left(\bar\del w + [w,\phi_{tw}\cdot\epsilon]\right)dt \notag \\
&=& \bar\del(v-w) \,+\, \int_0^1 [v-w\,,\;\phi_{tv}\cdot\epsilon]\,dt
\,+\, \int_0^1 [w\,,\;\phi_{tv}\cdot\epsilon-\phi_{tw}\cdot\epsilon]\,dt \\
\end{eqnarray*}
Integrating again, this time within the second Courant bracket, we get
\begin{eqnarray}
& & \bar\del(v-w) \;+\; \int_0^1 [v-w\,,\;\phi_{tv}\cdot\epsilon]\,dt \label{e301} \\
& & \quad +\; \int_0^1 \int_0^t \left[w\,,\; \bar\del(v-w)\, + [v,\phi_{\tau v}\cdot\epsilon] \,-\, [w,\phi_{\tau w}\cdot\epsilon]\right]\,d\tau\,dt \notag
\end{eqnarray}

To estimate $\|\Phi_v\cdot \epsilon - \Phi_w\cdot \epsilon\|_k$, we apply the triangle inequality to \eqref{e301}, and consider the three terms in turn.  Clearly, $\|\bar\del(v-w)\|_k$ is bounded by the first term in \eqref{e302}.

An aside: using the action estimate \eqref{nonlinear action estimate} and then Lemma \ref{flow estimate}, we see that
$$\|\phi_{tv}\cdot\epsilon\|_k \,\leq\, \LP(\|\epsilon\|_k + \|v\|_{k+1}).$$

Turning now to the second term of \eqref{e301}, we carry the norm inside the integral then, using the bracket estimate (Lemma \ref{bracket estimate}) and the above remark, we see that this term is bounded by the first term in \eqref{e302}.  Similarly, the third term in \eqref{e301} will be bounded by terms which have a factor of $\|v-w\|$, $\|w\|\cdot\|v\|$ or $\|w\|^2$.  Counting the total number of derivatives lost on each factor, the result follows.
\end{proof}

\subsection{Lemmas for estimate \eqref{MMZ5}}

The following lemma says that in our case the operator $\bar\del + [\epsilon_1,\cdot]$ is a good approximation of the deformed Lie algebroid differential $\bar\del + [\epsilon,\cdot]$.
\begin{lem}\label{estimate d_L}
For any $\epsilon \in \Gamma(\^ ^2 L^*)$ and large enough $k$,
$$\left\|\left(\bar\del V(\epsilon) + [\epsilon,V(\epsilon)]\right) - \left(\bar\del V(\epsilon) + [\epsilon_1,V(\epsilon)]\right)\right\|_k \,\leq\,
C \|\zeta(\epsilon)\|_{k+2}^2 \; (1 + \|\epsilon\|_{k+2})$$
\end{lem}
\begin{proof}
\begin{eqnarray*}
\left\|\left(\bar\del V(\epsilon) + [\epsilon,V(\epsilon)]\right) - \left(\bar\del V(\epsilon) + [\epsilon_1,V(\epsilon)]\right)\right\|_k &=& \|[\zeta(\epsilon),V(\epsilon)]\|_k \\
&\leq& C \|\zeta(\epsilon)\|_{k+1}\, \|V(\epsilon)]\|_{k+1} \\
&\leq& C \|\zeta(\epsilon)\|_{k+2}^2\, (1+\|\epsilon\|_{k+1}),
\end{eqnarray*}
(using Lemma \ref{V estimate} for the last step).
\end{proof}

The following lemma should be viewed as an approximate version of Proposition \ref{infinitesimal correction}, telling us that the infinitesimal action of $V(\epsilon)$ on $\epsilon$ almost eliminates the non-bivector component.
\begin{lem}\label{almost infinitesimal}
For an \emph{integrable} deformation $\epsilon \in \Gamma(\^ ^2 L^*)$ and large enough $k$,
$$\|\zeta\left(\bar\del V(\epsilon) + [\epsilon_1,V(\epsilon)] + \epsilon\right)\|_k \;\leq\; C \|\zeta(\epsilon)\|_{k+2}^2 \, (1 + \|\epsilon\|_{k+2})$$
\end{lem}
\begin{proof}
\begin{eqnarray*}
\bar\del V(\epsilon) + [\epsilon_1,V(\epsilon)] &=&
\bar\del P[\epsilon_1,P\epsilon_3] - \bar\del P\epsilon_2 - \bar\del P\epsilon_3 \\
& & +\; [\epsilon_1,P[\epsilon_1,P\epsilon_3]] - [\epsilon_1,P\epsilon_2] - [\epsilon_1,P\epsilon_3]
\end{eqnarray*}
The terms $[\epsilon_1,P\epsilon_2]$ and $[\epsilon_1,P[\epsilon_1,P\epsilon_3]]$ lie in $\^ ^2 T_{1,0}$, so
$$\zeta\left(\bar\del V(\epsilon) + [\epsilon_1,V(\epsilon)]\right) \;=\;
\bar\del P[\epsilon_1,P\epsilon_3] - \bar\del P\epsilon_2 - \bar\del P\epsilon_3
- [\epsilon_1,P\epsilon_3].$$
We apply the identity $\bar\del P = 1 - P \bar\del$ \eqref{P error} to the first three terms on the right hand side, giving us
\begin{eqnarray}
& & [\epsilon_1,P\epsilon_3] - P\bar\del[\epsilon_1,P\epsilon_3] - \epsilon_2 + P\bar\del\epsilon_2 - \epsilon_3 + P\bar\del\epsilon_3 - [\epsilon_1,P\epsilon_3] \notag \\
&=& -P\bar\del[\epsilon_1,P\epsilon_3] + P\bar\del\epsilon_2 + P\bar\del\epsilon_3 - \zeta(\epsilon) \label{error 1}
\end{eqnarray}
We now use the fact that $\epsilon$ satisfies the Maurer-Cartan equations, \eqref{MC1} through \eqref{MC4}.  By \eqref{MC4}, $P\bar\del\epsilon_3$ vanishes.  By \eqref{MC3},
\begin{equation}\label{term 1}
P\bar\del\epsilon_2 = -P\left( \frac{1}{2} [\epsilon_2,\epsilon_2] + [\epsilon_1,\epsilon_3] \right).
\end{equation}
By equation \ref{anticommute error},
\begin{eqnarray}
-P\bar\del[\epsilon_1,P\epsilon_3] &=& P[\epsilon_1,\bar\del P\epsilon_3]
- P[\bar\del\epsilon_1, P\epsilon_3] \notag \\
&=& P[\epsilon_1,\epsilon_3] - P[\bar\del\epsilon_1, P\epsilon_3] \label{term 2}
\end{eqnarray}
$P[\epsilon_1,\epsilon_3]$ cancels between \eqref{term 1} and \eqref{term 2}.  Thus \eqref{error 1} becomes
$$-\frac{1}{2} P[\epsilon_2,\epsilon_2] - P[\bar\del\epsilon_1, P\epsilon_3] - \zeta(\epsilon)$$
Applying \eqref{MC2} to $\bar\del\epsilon_1$, this is
$$-\frac{1}{2} P[\epsilon_2,\epsilon_2] + P[[\epsilon_1,\epsilon_2], P\epsilon_3] - \zeta(\epsilon)$$
Through applications of Lemmas \ref{P estimate} and \ref{bracket estimate}, we find that the first two terms have $k$-norm bounded by
$$C \left( \|\epsilon_2\|_{k+1}^2 \,+\, \|\epsilon_1\|_{k+2} \, \|\epsilon_2\|_{k+2} \, \|\epsilon_3\|_{k+1} \right)
\;\leq\; C \|\zeta(\epsilon)\|_{k+2}^2 \, (1 + \|\epsilon\|_{k+2})$$
The result follows.
\end{proof}

The following lemma is a version of Taylor's theorem.
\begin{lem}\label{Taylor estimate}
There is some $s$ such that for any integrable deformation $\epsilon \in \Gamma(\^ ^2 L^*)$, any $v \in \Gamma(L^*)$, and large enough $k$, 
$$\|(\Phi_v\cdot\epsilon - \epsilon) - (\bar\del v + [\epsilon,v])\|_k \,\leq\, \LP(1+\|\epsilon\|_{k+s},\, \|v\|_{k+s}^2).$$
\end{lem}
\begin{proof}
Applying the integral form of Proposition \ref{infinitesimal flow}, we see that
\begin{eqnarray*}
\left\|(\Phi_v\cdot\epsilon - \epsilon) - (\bar\del v + [\epsilon,v])\right\|_k
&=& \left\|\int_0^1 (\bar\del v + [\Phi_{tv}\cdot\epsilon,v])\, dt - (\bar\del v + [\epsilon,v])\right\|_k \\
&=& \left\|\int_0^1 [\Phi_{tv}\cdot\epsilon - \epsilon,\,v]\, dt\right\|_k \\
&\leq& \int_0^1 \LP(\|v\|_{k+1},\, \|\Phi_{tv}\cdot\epsilon - \epsilon\|_{k+1})\, dt
\end{eqnarray*}
Where in the last line we have carried the norm inside the integral and applied Lemma \eqref{bracket estimate}.  Applying the second axiom of SCI-actions \eqref{action estimate 2} and then Lemma \ref{flow estimate}, for some $s$ and $s'$,
\begin{eqnarray*}
\|\Phi_{tv}\cdot\epsilon - \epsilon\|_{k+1} &\leq& \LP(1+\|\epsilon\|_{k+s'},\, \|\Phi_{tv}-\Id\|_{k+s'}) \\
&\leq& \LP(1+\|\epsilon\|_{k+s},\|v\|_{k+s})
\end{eqnarray*}
Integrating the above estimate, the result follows.
\end{proof}

\begin{lem}[Estimate \ref{MMZ5}]
There is some $s$ such that, for any integrable deformation $\epsilon \in \Gamma(\^ ^2 L^*)$ and large enough $k$,
$$\|\zeta(\Phi_{V(\epsilon)} \cdot \epsilon)\|_k \leq \|\zeta(\epsilon)\|_{k+s}^{1+\delta} Poly(\|\epsilon\|_{k+s}, \|\Phi_{V(\epsilon)}-\Id\|_{k+s}, \|\zeta(\epsilon)\|_{k+s}, \|\epsilon\|_k),$$
where in this case the polynomial degree in $\|\epsilon\|_{k+s}$ does not depend on $k$.
\end{lem}
\begin{proof}
This is just an application of the triangle inequality using the estimates in this section.  We will show that, in the following series of approximations, terms on either side of a $\sim$ are close in the sense required:
$$\zeta(\Phi_{V(\epsilon)}\cdot\epsilon - \epsilon)
\;\sim\; \zeta(\bar\del V(\epsilon) + [\epsilon,V(\epsilon)])
\;\sim\; \zeta(\bar\del V(\epsilon) + [\epsilon_1,V(\epsilon)])
\;\sim\; -\zeta(\epsilon)$$
If so, then $\zeta(\Phi_{V(\epsilon)}\cdot\epsilon) \;\sim\; 0$ as required.

Applying the estimate for $V(\epsilon)$ (Lemma \ref{V estimate}) to Lemma \ref{Taylor estimate}, we see that
\begin{eqnarray*}
& & \|(\Phi_{V(\epsilon)}\cdot\epsilon - \epsilon) - (\bar\del V(\epsilon) + [\epsilon,V(\epsilon)])\|_k \\
& & \qquad\qquad   \leq\; \LP\left(1+\|\epsilon\|_{k+s},\, \|\zeta(\epsilon)\|_{k+s'+1}^2\,(1+\|\epsilon\|_{k+s'+1})\right).%\|\zeta(\epsilon)\|_{k+s}^2\, (1+\|\epsilon\|_{k+s})^3\, Poly(\|\Phi_{V(\epsilon)}-\Id\|_{k+s}).
\end{eqnarray*}
Applying $\zeta$ to the left hand side, this is the first approximation above.  (We remark that for large $k$, $\fl{(k+s)/2}+1 \leq k$, so we have a strictly limited degree in $\|\epsilon\|_l,\, l>k$.)  The remaining approximations are Lemma \ref{estimate d_L} (after applying $\zeta$ to its left hand side) and Lemma \ref{almost infinitesimal} respectively. 
\end{proof}

As remarked in Section \ref{prove main lemma}, we should now consider the Main Lemma proved.

\section{Main Lemma implies Main Theorem}\label{lemma implies theorem}

It is certainly the case that, near a complex point, a generalized complex structure is a deformation of a complex structure.  However, this deformation may not be small in the sense we need.  Therefore we use two means to control its size.

In this section, by $\delta_t : \C^n \to \C^n$ we will mean the dilation, $x \mapsto tx$.  If $\epsilon$ is a tensor on $\C^n$, then by $\delta_t \epsilon$ we mean the \emph{pushforward} of $\epsilon$ under the dilation map $x \mapsto tx$.  The complex structure on $\C^n$ is invariant under $\delta_t$; therefore if $\epsilon \in \Gamma(\^ ^2 L^*)$ is a deformation of the complex structure, then $\delta_t \epsilon = (0,\delta_t)\cdot\epsilon$ as in Definition \ref{Courant action on deformation}.

Suppose that $\epsilon \in \Gamma(\^ ^2 L^*)$, where $L^* = T_{1,0} \dsum T^*_{0,1}$, and that $\epsilon$ is decomposed into $\epsilon_1 + \epsilon_2 + \epsilon_3$, where $\epsilon_1$ is a bivector, $\epsilon_3$ is a 2-form and $\epsilon_2$ is of mixed type, as in Section \ref{Maurer-Cartan section}.  We wish to see how $\delta_t$ acts on these terms.

\begin{prop}\label{delta action}
Suppose that $t>0$.  For any $x \in \C^n$ and any $k$ we have the following pointwise norm comparisons for derivatives of $\epsilon$, before and after the dilation.  Let $k\geq0$.  Then
\begin{eqnarray*}
\|(\delta_t\epsilon_1)^{(k)}(tx)\|_0 &\leq& t^{2-k}\;\|\epsilon_1^{(k)}(x)\|_0 \\[3pt]
\quad \|(\delta_t\epsilon_2)^{(k)}(tx)\|_0 &\leq& \;t^{-k}\;\|\epsilon_2^{(k)}(x)\|_0 \\[3pt]
\textnormal{and}\quad \|(\delta_t\epsilon_3)^{(k)}(tx)\,\|_0 &\leq& t^{-2-k}\|\epsilon_3^{(k)}(x)\|_0.
\end{eqnarray*}
\end{prop}

\begin{proof}
Under a dilation, vectors scale with $t$ and covectors scale inversely with $t$.  Then
\begin{equation}\label{scaling 1}
(\delta_t\epsilon_1)(tx) = t^2\epsilon_1(x), \quad (\delta_t\epsilon_2)(tx) = \epsilon_2(x)
\quad\textnormal{and}\quad (\delta_t\epsilon_3)(tx) = t^{-2}\epsilon_3(x).
\end{equation}
If $x_i$ is a coordinate and $f$ a tensor, then
$$\frac{\del}{\del x_i}(\delta_t f)(tx) = \delta_t\left(\frac{\del}{\del tx_i} f\right)(tx)
= t^{-1}\delta_t\left(\frac{\del}{\del x_i} f\right)(tx)$$
This tells us that
$$\left\|(\delta_t f)^{(k+1)}(tx)\right\|_0
\leq t^{-1}\left\|\delta_t\left(f^{(k)}\right)(tx) \right\|_0$$
By induction on this inequality and then applying the formulas in \eqref{scaling 1}, the result follows.
\end{proof}

We now define the $\lambda$-transform, which is not a Courant isomorphism, but which does take generalized complex structures to generalized complex structures.

\begin{defn}
If $t>0$, let $\lambda_t : T \dsum T^* \to T \dsum T^*$ so that $\lambda_t(X,\xi) = (tX,\xi)$.  Then $\lambda_t$ also acts on generalized complex structures by mapping their eigenbundles (or by conjugating $J$).
\end{defn}

$\lambda_t$ commutes with diffeomorphisms, but it does not quite commute with Courant isomorphisms.
\begin{notn}
If $\Phi = (B,\phi)$ is a Courant isomorphism, then let $\lambda_t\cdot\Phi = (t^{-1}B,\phi)$.
\end{notn}

\begin{prop}
If $\Phi$ is a Courant isomorphism then
$$\Phi \comp \lambda_t = \lambda_t \comp (\lambda_t\cdot\Phi).$$
\end{prop}

Again we consider a deformation $\epsilon = \epsilon_1 + \epsilon_2 + \epsilon_3$ of the complex structure on $\C^n$.  $\lambda_t(L_\epsilon)$ will be another generalized complex structure.

\begin{prop}\label{lambda action}
$\lambda_t(L_\epsilon) = L_{\lambda_t\epsilon}$, where
$$\lambda_t\epsilon = t \epsilon_1 + \epsilon_2 + t^{-1} \epsilon_3.$$
\end{prop}

\begin{rem}
We can check that this transformation respects the Maurer-Cartan equations, \eqref{MC1} through \eqref{MC4}, which tells us that if $L_\epsilon$ was generalized complex then so is $\lambda_t(L_\epsilon)$.
\end{rem}

We can now prove that the Main Theorem follows from the Main Lemma.  Recall: 
\begin{main lem}
Let $J$ be a generalized complex structure on the closed unit ball $B_1$ about the origin in $\C^n$.  Suppose that $J$ is a small enough deformation of the complex structure on $B_1$, and suppose that $J$ is of complex type at the origin.  Then, in a neighbourhood of the origin, $J$ is equivalent to a deformation of the complex structure by a holomorphic Poisson structure on $\C^n$.
\end{main lem}

\begin{main thm}
Let $J$ be a generalized complex structure on a manifold $M$ which is of complex type at point $p$.  Then, in a neighbourhood of $p$, $J$ is equivalent to a generalized complex structure induced by a holomorphic Poisson structure, for some complex structure near $p$.
\end{main thm}

\begin{proof}[Proof of Main Theorem from Main Lemma]
Suppose that $J$ is a generalized complex structure on $M$, with $p$ a point of complex type.  We may assume without loss of generality that $p=0$ in the closed unit ball $B_1 \subset \C^n$, where the complex structure on $T_0 \C^n$ induced by $J$ agrees with the standard one.  By application of an appropriate $B$-transform, we may assume that, at $0$, $J$ agrees with the standard generalized complex structure, $J_{\C^n}$, for $\C^n$.  Then, near $0$, $J$ is a deformation of $J_{\C^n}$ by $\epsilon \in \Gamma(\^ ^2 L^*)$, and $\epsilon$ vanishes at $0$.

For $t>0$, let
$$R_t \epsilon = \delta_{t^{-1}}\, \lambda_{t^2}\, \epsilon = \lambda_{t^2}\, \delta_{t^{-1}}\, \epsilon.$$
Since $\epsilon$ (and hence $R_t \epsilon$) vanishes at $0$ to at least first order, there is some $C>0$ such that, for all $0 < t \leq 1$,
\begin{equation}\label{e701}
\|(R_t \epsilon) (x)\|_0 \leq C\,t\, \|(R_t \epsilon)(t^{-1}x)\|_0.
\end{equation}
For derivatives $k>0$, we apply Proposition \ref{delta action} to the components $\epsilon_1$, $\epsilon_2$ and $\epsilon_3$, and
\begin{eqnarray*}
\|(R_t \epsilon_1)^{(k)}(t^{-1}x)\|_0 &\leq& t^{-2+k}\, \|(\lambda_{t^2}\epsilon_1)^{(k)}(x)\|_k, \\
\|(R_t \epsilon_2)^{(k)}(t^{-1}x)\|_0 &\leq& \; t^k\quad \|(\lambda_{t^2}\epsilon_2)^{(k)}(x)\|_k \quad\textnormal{and} \\
\|(R_t \epsilon_3)^{(k)}(t^{-1}x)\|_0 &\leq& t^{2+k}\; \|(\lambda_{t^2}\epsilon_3)^{(k)}(x)\|_k.
\end{eqnarray*}
But according to Proposition \ref{lambda action}, the effect of $\lambda_{t^2}$ is
\begin{eqnarray*}
\|(R_t \epsilon_1)^{(k)}(t^{-1}x)\|_0 &\leq& t^k\, \|\epsilon_1^{(k)}(x)\|_k, \\
\|(R_t \epsilon_2)^{(k)}(t^{-1}x)\|_0 &\leq& \; t^k\, \|\epsilon_2^{(k)}(x)\|_k \quad\textnormal{and} \\
\|(R_t \epsilon_3)^{(k)}(t^{-1}x)\|_0 &\leq& t^k\, \|\epsilon_3^{(k)}(x)\|_k.
\end{eqnarray*}
Thus, if $k>0$ or (because of \eqref{e701}, if $k=0$ also), we have
$$\|(R_t \epsilon)^{(k)} (x)\|_0 \leq C\,t\, \|\epsilon^{(k)}(x)\|_0.$$
So,
$$\|(R_t \epsilon) (x)\|_k \leq C\,t\, \|\epsilon(x)\|_k.$$

Taking the $\sup$-norm always over the fixed set $B_1$, we have that $\|R_t \epsilon\|_k \leq C\,t\, \|R_t \epsilon\|_k$.
Thus $\|R_t \epsilon\|_k$ is as small as we like for some $t$, and satisfies the hypotheses of the Main Lemma; so there exists a local Courant isomorphism $\Phi$ such that $\Phi_t \cdot R_t \epsilon = \beta$, where $\beta$ is a holomorphic Poisson bivector.  Then
\begin{eqnarray*}
\beta &=& \Phi_t \cdot \left(\lambda_{t^2}\, \delta_{t^{-1}}\, \epsilon\right) \\
&=& \lambda_{t^2} \left((\lambda_{t^2}\cdot\Phi_t) \cdot \delta_{t^{-1}}\, \epsilon\right)
\end{eqnarray*}
But the action of $\lambda_{t^2}$ on a bivector is just scaling by $t^2$, so
$$(\lambda_{t^2}\cdot\Phi_t) \cdot \delta_{t^{-1}}\, \epsilon = t^{-2}\beta.$$
Thus, starting from a suitably small neighbourhood of $0$, by applying first the dilation $\delta_{t^{-1}}$ and then the Courant isomorphism $\lambda_{t^2}\cdot\Phi_t$, we see that $\epsilon$ is locally equivalent to the holomorphic Poisson structure $t^{-2}\beta$.
\end{proof}

\article{
\input{biblio-hol}

%\documentclass{amsart}
%\usepackage{graphicx}
%\usepackage{amsmath,amscd,amssymb}%,accents}

%\input{declarations}

% Chapter mode (no abstract, acknowledgments, etc.)
% \newcommand{chaptertext}[1]{#1}
% \newcommand{papertext}[1]{}

% Paper mode (include abstract, acknowledgments, etc.)
%\newcommand{chaptertext}[1]{}
%\newcommand{papertext}[1]{#1}

%\begin{document}

%\begin{abstract}
%In this paper, we give answers to the question ``when are a regular Poisson structure along with a complex structure transverse to its leaves induced by generalized complex structure?''  We show that a necessary and sufficient condition is the existence of certain integrating forms satisfying a system of differential equations.  Specializing to the case of smooth symplectic families over a complex manifold, we find that as a necessary condition the relative cohomology of the symplectic form should be pluriharmonic under the Gauss-Manin connection.  We find examples of smooth symplectic families which are not induced by generalized complex structures.  Finally, we give an example of a smooth symplectic family which is generalized complex but which is not a symplectic fibre bundle.
%\end{abstract}

%\title{Generalized complex structures and symplectic foliations}
%\author{Michael Bailey}
%\maketitle

\chapter{Generalized complex structures on symplectic foliations}

A generalized complex structure induces a Poisson structure and, transverse to its symplectic foliation, a complex structure.  For a \emph{regular} generalized complex structure, there is no more local information than this (up to isomorphism).

In considering the relation between, on the one hand, pairs $(P,I)$ of regular Poisson structures $P$ and transverse complex structures $I$ and, on the other hand, regular generalized complex structures, two questions come naturally to mind.  First of all, given such a $(P,I)$, is it induced by a generalized complex structure?  Locally, the answer is always yes, so any obstruction must be global.  Second, if $(P,I)$ is induced by a generalized complex structure, how may we classify (up to $B$-transform) those generalized complex structures which induce $(P,I)$?  It is the first of these questions we address in this chapter.

In Section \ref{intro section}, we review the definitions and basic facts of generalized complex structures from the pure spinor viewpoint.  In Section \ref{problem statement}, we state our problem precisely, give the basic construction we will continue to use throughout the chapter, and give some simple sufficient conditions for an affirmitive answer.  In Section \ref{integrability conditions}, we study our construction in more detail, and give necessary and sufficient conditions for an affirmitive answer (see Theorem \ref{H iff alpha and beta}).

In Section \ref{smooth symplectic families}, we study these conditions in the case where the symplectic foliation of $P$ corresponds to a fibre bundle over a complex base.  We find that, as a necessary condition, the relative cohomology of the symplectic form should be \emph{pluriharmonic} under the Gauss-Manin connection (Theorem \ref{[omega] pluriharmonic}).  If a smooth symplectic family induced by a generalized complex structure is a surface bundle, or in higher dimensions if certain topological conditions are satisfied, we show that it is in fact a symplectic fibre bundle, that is, it has symplectic trivializations (Theorem \ref{symplectic}).

We find some counterexamples, for which a regular Poisson structure and transverse complex structure do not come from a generalized complex structure, as well as some unexpected examples (in particular, see Example \ref{counterexample}).

\section{Pure spinors and generalized complex structures}\label{intro section}

We briefly review the pure spinor formalism of generalized complex structures.  This is not the usual way these structures are introduced---a generalized complex structure $J$ on a manifold $M$ is usually defined as a complex structure on a \emph{Courant algebroid} over $M$ (see, eg., the introduction to Chapter 2, or Chapter 5, Definitions \ref{Courant algebroid} and \ref{abstract generalized complex} for more details)---but the data in either formalism determine each other, and in this chapter we stick to only one for the sake of brevity.  For details, and proofs of claims in this section, see \cite{Gualtieri2011}.

\begin{rem}
We may consider the complexification of every quantity in the following definitions, eg., complex pure spinors.  As a point of notation, we indicate the complexification of a real vector bundle by a subscript $\C$.  For example, if $V$ is a vector bundle over $M$, then $V_\C = \C \tens V$. In this chapter we only consider smooth sections of bundles.
\end{rem}

\subsection{Algebraic definitions}

\begin{defn}\label{define spinor}
By a \emph{spinor} on a manifold $M$ we will mean a (complex) mixed-degree differential form $\rho \in \Gamma(\^ ^\bullet T^*_\C M)$.
\end{defn}

\begin{defn}
Sections of $TM \dsum T^*M$ act on the spinors via the \emph{Clifford action}, by contraction and wedging: if $(X,\xi) \in \Gamma(TM \dsum T^*M)$ and $\rho$ is a spinor, then
$$(X,\xi)\cdot\rho = \iota_X \rho + \xi\^\rho.$$
Therefore, every spinor $\rho$ on $M$ has a \emph{null subbundle} $L_\rho \subset TM \dsum T^*M$ which is just its annihilator under the Clifford action.
\end{defn}

\begin{defn}
A spinor $\rho$ is \emph{pure} if $L_\rho$ is a \emph{maximal} isotropic subbundle with respect to the standard symmetric pairing on $TM \dsum T^*M$. 
\end{defn}
Such a maximal isotropic subbundle will have half the rank of $TM \dsum T^*M$; that is, its rank will be the dimension of $M$.

\begin{defn}
A (complex) maximal isotropic $L \subset T_\C M \dsum T^*_\C M$ has \emph{real rank zero} if its intersection with its complex conjugate is trivial, that is, $L \cap \bar{L} = 0$.  We will also say that a (complex) pure spinor $\rho$ has real rank zero if $L_\rho$ does.
\end{defn}

An almost complex structure may be given by its canonical line bundle, the top wedge power of the $(1,0)$-forms.  Analogously,
\begin{defn}\label{almost gc}
An \emph{almost generalized complex structure} $J$ on $M$ is given by a pure spinor line bundle $\kappa_J \subset \^ ^\bullet T^*_\C M$ of real rank zero, called the \emph{canonical line bundle} of $J$.

The \emph{type} of $J$ at a point $x$ is the lowest nontrivial degree of its canonical line bundle at $x$.  $J$ is \emph{regular} at $x$ if its type is constant near $x$.
\end{defn}
We understand the type of $J$ as the number of complex dimensions, as will be made clear (see Proposition \ref{pointwise form of J}).

\begin{rem}\label{formalism correspondence}
The relation between the pure spinor formalism and the definition of generalized complex structures as anti-involutions, $J:T_\C M \dsum T_\C M \to T_\C M \dsum T_\C M$, on the standard Courant algebroid, is just that the $+i$-eigenbundle of $J$ is identified with the null subbundle, $L_\rho$, of the pure spinor bundle.  The conditions on integrability (Definition \ref{define integrability}) will be equivalent \cite{Gualtieri2011}.
\end{rem}

\begin{defn}\label{define transverse complex}
Let $S \subset TM$ be a distribution on $M$.  Then an \emph{almost complex structure transverse to $S$} is an almost complex structure on $NS = TM/S$. 

A transverse almost complex structure $I$ gives a decomposition of $N^*_\C S$ into $+i$ and $-i$--eigenbundles, $N^*_{1,0} S$ and $N^*_{0,1} S$ respectively.  The \emph{canonical line bundle} of $I$ is
$$\kappa_I = \^ ^k N^*_{1,0} S,$$
where $k = \dim_\C(N_{1,0} S)$.

A transverse almost complex structure is \emph{integrable} if $S \dsum N_{1,0}$ is Lie-involutive.
\end{defn}

\begin{defn}\label{define exponential}
We define the spinor exponential with the usual Taylor series: let $B$ be a form of even degree.  Then
$$e^B = 1 + B + \frac{1}{2} B \^ B + \ldots,$$
where in this case the wedge products eventually vanish and the series is finite.
\end{defn}
Note that, since wedge product is symmetric for forms of even degree, this exponential is actually a homomorphism from addition to wedge product.

\begin{prop}\label{pointwise form of J}
If $J$ is an almost generalized complex structure, then \emph{at a point} its canonical line bundle is of the form
\begin{equation}
\kappa_J = e^{B+i\omega} \^ \kappa_I
\end{equation}
for some almost complex structure $I$ transverse to a (possibly singular) distribution $S$, and real 2-forms $B$ and $\omega$, where the pullback of $\omega$ to $S$ is nondegenerate.  If $J$ is \emph{regular}, then such a representation exists in a neighbourhood of any point.
\end{prop}

\begin{rem}\label{S and I well-defined}
$\kappa_I$ is the lowest-degree component of $\kappa_J$; hence, the distribution $S = \Ann(\kappa_I)$ and the transverse almost complex structure $I$ are uniquely determined by $Ja$.  However, $B + i\omega$ is not.  Rather, $B+i\omega$ is well-defined up to $N^*_{1,0} \^ T^*_\C M$; that is, $B+i\omega$ is a well-defined section of $\^ ^2 \left(S_\C \dsum N_{0,1}\right)^*$.
\end{rem}

In generalized geometry, the symmetry group of a manifold is taken to be larger than just the diffeomorphisms.  In addition, it includes the following:
\begin{defn}
If $B$ is a real 2-form and $J$ is an almost generalized complex structure with canonical line bundle $\kappa_J$, then the \emph{$B$-field transform} (or just \emph{$B$-transform}) of $J$ is written $B\cdot J$, and may be defined in terms of its action on $\kappa_J$:
$$\kappa_{B\cdot J} = e^B \^ \kappa_J.$$
\end{defn}
We distinguish between closed $B$-transforms and non-closed $B$-transforms, since when the 2-form $B$ is non-closed, the integrability condition changes (see Proposition \ref{H+dB}).

\subsection{Integrability of generalized complex structures}

\begin{defn}\label{define integrability}
If $H$ is a closed real 3-form and $\rho$ is a spinor, then
$$d_H \rho := d\rho + H\^\rho.$$
We say that a pure spinor $\rho$ is \emph{$H$-integrable} if
$$d_H \rho = (X,\xi)\cdot\rho$$
for some $(X,\xi) \in T_\C M \dsum T^*_\C M$.

We say that an almost generalized complex structure $J$ is $H$-integrable, or, alternatively, that $J$ is a generalized complex structure with curvature $H$, if near any point the canonical line bundle $\kappa_J$ of $J$ has an $H$-integrable local generating section.
\end{defn}
\begin{rem}
If such a cover of $H$-integrable sections of $\kappa_J$ exists, then in fact \emph{all} local sections of $\kappa_J$ will be $H$-integrable.

Furthermore, if $J$ is $H$-integrable and is regular at $x$, then in fact there is a $d_H$-closed local generating section of $\kappa_J$ near $x$.
\end{rem}

\begin{prop}\label{H+dB}
If $J$ is an $H$-integrable generalized complex structure for some closed 3-form $H$, and $B$ is a 2-form, then $B\cdot J$ is an $(H + dB)$-integrable generalized complex structure. 
\end{prop}

\begin{prop}\label{J induces P and I}
Let $J$ be a generalized complex structure integrable with respect to some closed 3-form.  As per Remark \ref{S and I well-defined}, let $S$ be the (possibly singular) distribution determined by $J$, let $I$ be the transverse almost complex structure, and let $B + i\omega$ be the 2-form on $S \dsum N_{0,1}S$.

Then $S$ integrates to a foliation, $\omega$ pulls back to this foliation to give the symplectic leaves of a Poisson structure, and $I$ is integrable.
\end{prop}

Thus, a generalized complex structure determines a Poisson structure and a transverse complex structure.  The following local normal form theorem says that this is the only local information. 

\begin{thm}[Gualtieri, \cite{Gualtieri2011}]\label{regular local normal form}
If $J$ is an (integrable) generalized complex structure regular at $x$, then there is a neighbourhood of $x$ which is isomorphic---via diffeomorphism and $B$-transform---to a neighbourhood in the following generalized complex manifold:

Let $\kappa_I = \^ ^k T_{1,0} \C^k$ be the canonical line bundle of $\C ^k$ for some $k$, and let $\omega$ be the standard symplectic form on $\R^{2m}$ for some $m$.  Then the line bundle,
$$\kappa = e^{i\omega} \^ \kappa_I,$$
is the canonical bundle for a generalized complex structure with curvature $H=0$ on $\C^k \times \R^{2m}$.
\end{thm}
In other words, near a regular point, any generalized complex structure is equivalent to the product of a complex structure with a symplectic structure.

\section{Problem statement and non-integrable solution}\label{problem statement}

In what follows, suppose that $P$ is a \emph{regular} Poisson structure on $M$ with symplectic foliation $F$, and that $I$ is a transverse complex structure to $F$.  Our question is:
\begin{itemize}
\item When does the pair $(P,I)$ come from a generalized complex structure on $M$, as in Proposition \ref{J induces P and I}?
\end{itemize}
This is always the case locally, as a corollary to Theorem \ref{regular local normal form}, so any obstruction must be global.  The global answer is ``not always'' (even though $P$ and $I$ are integrable).  We will provide counterexamples.

However, we can always find \emph{almost} generalized complex structures inducing $(P,I)$.  We give a construction, which will then be the basis of our solution of the original question.

\begin{notn}
We may understand the Poisson structure as a map $P:T^*M \to TM$.  Then, in what follows, let $S = \im(P) = TF$, where $F$ is the symplectic foliation.  $P$ determines a leafwise symplectic form $\omega \in \Gamma(\^ ^2 S^*)$.
\end{notn}

$\omega$ is leafwise-closed, and leafwise nondegenerate.  The complexification of the normal bundle, $N = TM/S$, splits as $N_\C = N_{1,0} \dsum N_{0,1}$ according to the transverse complex structure $I$.  The integrability condition on $I$ is just that the bundle $N_{0,1} \dsum S_\C$ is Lie-involutive in $T_\C M$.

\begin{defn}\label{definition from N}
If $N \subset TM$ is a smooth distribution complementary to $S$ (using the same notation as for the normal bundle), then we may extend $\omega$ to $M$ by specifying $\ker(\omega) = N$.

Then the almost generalized complex structure, $J_N$, induced by $(P,I)$ and the choice of $N$, is defined by the canonical line bundle
\begin{equation}\label{defining spinor}
\kappa = e^{i\omega} \^ (\^ ^k N_{1,0}^*) = e^{i\omega} \^ \kappa_I
\end{equation}
\end{defn}

\begin{prop}\label{construction is general}
Let $J$ be an almost generalized complex structure inducing $(P,I)$.  Then for any choice of complementary distribution $N \subset TM$, $J$ is equivalent, up to a $B$-transform, to the structure $J_N$.
\end{prop}
\begin{proof}
As per Proposition \ref{pointwise form of J}, the canonical line bundle of $J$ is
$$\kappa_J = e^{B+i\omega} \^ \kappa_I,$$
where $B+i\omega$ is a section of $\^ ^2 \left(S_\C \dsum N_{0,1}\right)^*$.  By choosing $N \subset TM$ and specifying $B+i\omega \in \Ann(N_{1,0})$, we extend $B+i\omega$ to a section of $\^ ^2 T^*_\C M$.  $\omega$ extended in this way will be the same $\omega$ as in Definition \ref{definition from N} above.  Then
$$e^{-B} \^ \kappa_J = e^{i\omega} \^ \kappa_I.$$
\end{proof}

\subsection{Example}

Using Definition \ref{definition from N}, we can answer our question in the affirmitive in some special cases.  The following was originally observed by Cavalcanti \cite{Cavalcanti}:

\begin{prop}
If the leafwise-closed symplectic form $\omega$ extends to a closed form $\tilde\omega$ on $M$, then $(P,I)$ comes from a generalized complex structure, integrable with curvature $H=0$.
\end{prop}
\begin{proof}
The generalized complex structure is defined by the canonical line bundle
$$e^{i\tilde\omega} \^ \kappa_I,$$
which admits closed local sections.
\end{proof}

\begin{cor}\label{complementary foliation}
If $S$ admits a complementary foliation $R$, for which $\omega$ is constant in the directions of $R$, then $(P,I)$ comes from a generalized complex structure.
\end{cor}
\begin{proof}
We embed the normal bundle as $N = TR \subset TM$, using the same symbol $N$.  As in Definition \ref{definition from N}, we extend $\omega$ to $M$ by specifying that $\ker(\omega) = N$.  Since $\omega$ is constant along the directions of $R$, and $\omega$ is closed on $S$, then on the total space $d\omega=0$.  (We can see this by expressing a neighbourhood as a product decomposition for $S$ and $R$.)
\end{proof}

\section{Integrability---the general case}\label{integrability conditions}

If $J$ is a generalized complex structure inducing $(P,I)$, then for \emph{any} choice of $N$ complementary to $S$, $J$ is equivalent via a $B$-transform to the almost generalized complex structure $J_N$ in definition \ref{definition from N}.  But then $J_N$ is integrable---if $J$ was $H'$--integrable and $B\cdot J = J_N$, then $J_N$ is $H$--integrable for $H=H'+dB$.  Contrapositively, if $J_N$ is not integrable, then there is no generalized complex structure $J$ inducing $(P,I)$.

Therefore, we answer the question of whether $(P,I)$ comes from a generalized complex structure by choosing any complementary $N$, and then testing to see if $J_N$ is integrable for some closed 3-form $H$.

Fix the choice $N \subset TM$ complementary to $S$.  Suppose that $H$ is a closed 3-form such that $J_N$ is $H$-integrable.  We will study two types of conditions on $H$---the $H$-integrability of $J_N$, and the closedness of $H$.

\subsection{Trigrading}\label{trigrading}
The decomposition $TM = N_{1,0} \dsum N_{0,1} \dsum S$ gives us a trigrading on forms. We write the components of $H$ as $H^{ij,k}$, where $i$ and $j$ are the degrees in the Dolbeault complex of $N$, and $k$ is the degree on $S$.

We decompose the exterior derivative according to the grading.  If the distributions $N_{1,0}$, $N_{0,1}$ and $S$ were each integrable, then $d$ would decompose as $\del + \bar\del + d_S$, where each term increases by one the respective degree in the trigrading.  The terms would anticommute, and we would have a triple complex.  However, since the distributions may not be integrable, there may be additional terms.

\begin{lem}\label{d decomposition}
$d$ decomposes as
\begin{equation}\label{d formula}
d = \underset{10,0}{\del}
+ \underset{01,0}{\bar\del}
+ \underset{20,-1}{\theta}
+ \underset{02,-1}{\bar\theta}
+ \underset{00,1}{d_S}
\end{equation}
Under each term we indicate the degree of the operator.  Furthermore, $\theta$ and $\bar\theta$  are tensors in \mbox{$\Gamma(\^ ^2 N^* \tens S)$,} acting on forms by contracting in $S$ and wedging in $N^*$.
\end{lem}

\begin{proof}
We don't give details, since results similar to this are found in the literature.  See, for example, a real counterpart in \cite[Proposition 10.1]{BGV1992}; $d$ decomposes as
$$d = d_N + \Theta + d_S,$$
where $d_N$ has degree $+1$ in $N^*$, and $\Theta$ acts as a tensor with degree $+2$ in $N^*$ and $-1$ in $S^*$.  Similarly, pulled back to the integrable distributions $N_{1,0} \dsum S$ and $N_{0,1} \dsum S$, $d$ becomes $\del + \theta + d_S$ and $\bar\del + \theta + d_S$ respectively, and then $d_N=\del+\bar\del$ and $\Theta = \theta + \bar\theta$.
\end{proof}

\begin{rem}
It is not the case that each term in the decomposition of $d$ squares to zero; but, of course, $d^2=0$, and by decomposing this equation according to degree we may find second-order relations between terms.
\end{rem}

\subsection{$H$-integrability}

\begin{prop}
Let $H$ be a real closed 3-form.  Then $J_N$, as in Definition \ref{definition from N}, is $H$-integrable if and only if the following equations hold:
\begin{eqnarray}
H^{00,3} &=& 0 \label{H1} \\
H^{01,2} &=& -i\bar\del\omega \label{H2} \\
H^{02,1} &=& -i\bar\theta\omega \label{H3} \\
H^{03,0} &=& 0 \label{H4}
\end{eqnarray}
\end{prop}

\begin{proof}
The $H$-integrability of $J_N$ says that, for a local closed generating section $\Omega \in \Gamma(\kappa_I)$, we should have
\begin{eqnarray*}
0 &=& d_H (\Omega \^ e^{i\omega}) \\
&=& \Omega \^ (d_H e^{i\omega}) \\
&=& \Omega \^ (d + H\^) e^{i\omega} \\
\Leftrightarrow 0 &=& \Omega \^ (d i\omega + H)
\end{eqnarray*}
That is, for all $j$ and $k$,
\begin{equation}\label{constraint on H}
-i(d \omega)^{0j,k} = H^{0j,k}.
\end{equation}

Equations \eqref{H2}, \eqref{H3} and \eqref{H4} are just selected degrees of this condition, according to the trigrading.  The $(00,3)$-degree component of this condition is $H^{00,3} = -i d_S \omega$; however, we supposed that $\omega$ was leafwise-closed, i.e., $d_S\omega=0$, and so we get Equation \eqref{H1}.

Since $H$ is real, we have that $H^{ij,k} = \Bar{H^{ji,k}}$, and so these equations also determine $H^{10,2}$, $H^{20,1}$ and $H^{30,0}$.
\end{proof}

The effect of these conditions is that certain degrees of $H$ are completely prescribed by $(P,I)$ and the choice of $N$.  These prescribed terms are not in themselves an obstruction to the existence of an appropriate $H$.  Given these conditions, there remain two free terms, $H^{11,1}$ and $H^{21,0}$.  ($H^{12,0}$ must be conjugate to $H^{21,0}$.)

\subsection{Closedness of $H$}

\begin{thm}\label{H iff alpha and beta}
Suppose that $P$ and $I$ are a regular Poisson structure and a transverse complex structure respectively.  Let $N$ be a choice of complementary distribution to the symplectic distribution $S$, let $\omega$ be the extended symplectic form (as in Definition \ref{definition from N}), and let $d = \del + \bar\del + \theta + \bar\theta + d_S$ be the decomposition as in Lemma \ref{d decomposition}.

Then $(P,I)$ comes from a generalized complex structure if and only if there exist forms
$$\alpha \in \Omega^{11,1} \;\textnormal{and}\; \beta \in \Omega^{21,0}$$
(where $\alpha$ is real) such that the following hold:
\begin{eqnarray}
0 &=& \del\beta + \theta\alpha   \label{dH1} \\
0 &=& \del\bar\beta + \bar\del\beta + 2i\bar\theta\theta\omega   \label{dH2} \\
0 &=& \del\alpha + 2i\bar\del\theta\omega + d_S\beta   \label{dH3} \\
0 &=& -2i\del\bar\del\omega + d_S\alpha   \label{dH4}
\end{eqnarray}
\end{thm}

\begin{proof}
We take the almost generalized complex structure $J_N$ as in Definition \ref{definition from N}, and try to find a closed 3-form $H$ integrating it.  This succeeds if and only if $(P,I)$ comes from a generalized complex structure.  Such an $H$, if it exists, is determined by Equations \eqref{H1} through \eqref{H4}, the free terms $\alpha = H^{11,1}$ and $\beta = H^{21,0}$, and the reality conidition on $H$.

We will decompose the condition $dH=0$ by degree in the trigrading.  We only look at terms which involve $\alpha$ or $\beta$, and we consider only one term from each conjugate pair.  Furthermore, we make the substitutions in equations \eqref{H1} through \eqref{H4} and their conjugates.  Then
\begin{eqnarray*}
(dH)^{31,0} &=& \del H^{21,0} + \bar\del H^{30,0} + \theta H^{11.1} \\
0 &=& \del\beta + \theta\alpha \\ \\
(dH)^{22,0} &=& \del H^{12,0} + \bar\del H^{21,0} + \theta H^{02,1} + \bar\theta H^{20,1} \\  
0 &=& \del\bar\beta + \bar\del\beta - i\theta\bar\theta\omega + i\bar\theta\theta\omega \\ \\
(dH)^{21,1} &=& \del H^{11,1} + \bar\del H^{20,1} + \theta H^{01,2} + d_S H^{21,0} \\  
0 &=& \del\alpha + i\bar\del\theta\omega - i \theta\bar\del\omega + d_S\beta \\ \\
(dH)^{11,2} &=& \del H^{01,2} + \bar\del H^{10,2} + d_S H^{11,1} \\
0 &=& -i\del\bar\del\omega + i\bar\del\del\omega + d_S\alpha
\end{eqnarray*}
These resemble Equations \eqref{dH1} through \eqref{dH4}, except that in the statement of the theorem we have used the following anticommutation relations:
$$\theta\bar\theta + \bar\theta\theta = 0,  \;\;\;  \bar\del\theta + \theta\bar\del = 0  \;\;\textnormal{and}\;\;
\del\bar\del + \bar\del\del = 0$$
(These relations follow from grouping terms in $d^2 = 0$ by degree.)

The other terms, which do not feature $\alpha$ or $\beta$, automatically vanish, and thus provide no constraint.  (If they didn't vanish, then there could be no possible solution for $H$, \emph{even locally}, and as we remarked earlier, this problem is always locally solvable.)
\end{proof}

\begin{rem}\label{type 1}
Note that if the type---that is, the complex dimension of $N_{1,0}$---is 1, then equations \eqref{dH1}, \eqref{dH2} and \eqref{dH3} are trivially satisfied.
\end{rem}

\section{Smooth symplectic families}\label{smooth symplectic families}

\begin{defn}\label{define smooth symplectic family}
A smooth symplectic family over a complex manifold $B$ is a fibre bundle $\pi:X \to B$ with a Poisson structure whose symplectic leaves coincide with the fibres.  By pullback from $B$, it inherits a complex structure transverse to the symplectic foliation.

We say that a smooth symplectic family over a complex manifold is \emph{generalized complex} if its Poisson structure and transverse complex structure are induced by a generalized complex structure.
\end{defn}

\begin{rem}
A smooth symplectic family need not be a \emph{symplectic fibre bundle}, since it may not have symplectic local trivializations.
\end{rem}

\subsection{Surface bundles}\label{surface bundles}
We first consider smooth symplectic families with compact, 2-dimensional fibres over a complex base.  We show that the symplectic volume of the fibres, as a function on the base, is pluriharmonic.  Furthermore, if the base is compact and connected, then the symplectic family is in fact a symplectic fibre bundle.

\begin{prop}\label{V pluriharmonic}
Let $\pi:X \to B$ be a smooth symplectic family with compact, 2-dimensional fibres over a complex manifold $B$, with Poisson structure $P$ and transverse complex structure $I$.

If $(P,I)$ comes from a generalized complex structure on $X$, then the function $V : B \to \R$ giving the symplectic volume of each fibre must be pluriharmonic, i.e., $\del\bar\del V = 0$.  In particular, if $B$ is compact then $V$ is locally constant.
\end{prop}

\begin{proof}
Suppose that $(P,I)$ comes from a generalized complex structure $J$.  Pick a point $p \in B$.  Over a neighbourhood $U \ni p$, we may trivialize $\pi^{-1}(U)$ as a surface bundle, and thus choose a horizontal distribution, $N\subset TX$, which is integrable.  Thus we get the decomposition of $d$ as in Lemma \ref{d decomposition}.

We then construct the almost generalized complex structure $J_N$ on $\pi^{-1}(U)$ as in Definition \ref{definition from N}.  Since $J$ is integrable, so must be $J_N$ for some curvature $H$. Then there must exist $\alpha = H^{11,1}$ and $\beta = H^{21,0}$ which satisfy equations \eqref{dH1} through \eqref{dH4}.  In particular, if $\omega$ is the leafwise symplectic form (extended to $X$ by the choice of $N$), then
$$d_S\alpha = 2i\del\bar\del\omega.$$

The Poisson structure determines an orientation on the fibres, and since the fibres are compact surfaces, we may integrate this equation over the fibres---we denote this integration by $\int_F$.  Since $d_S\alpha$ is $d_S$-exact and $F$ is compact, by Stokes' theorem $\int_F d_S = 0$.  Furthermore, $\del\bar\del$ commutes with the integral; so we have
\begin{eqnarray*}
\int_F d_S\alpha &=& \int_F 2i\del\bar\del\omega \\
0 &=& 2i \del\bar\del \int_F \omega \\ 
0 &=& 2i \del\bar\del V 
\end{eqnarray*}
thus proving the claim.

Finally, if $B$ is compact, then, by the maximum principle, the condition $\del\bar\del V = 0$ implies that $V$ must be locally constant.
\end{proof}

This provides a source of counterexamples:
\begin{cor}
Let $X \to B$ be a compact symplectic surface bundle over a compact complex manifold $B$, with Poisson structure $P$ coming from the symplectic fibration.  Let $I$ be the transverse complex structure pulled back from $B$.

If $V : B \to \R$ is a smooth, positive, not-locally-constant real function, then the product $\tilde{P}=VP$ is a Poisson structure on $X$, and the pair $\tilde{P}$ and $I$ does not come from a generalized complex structure on $X$.
\end{cor}

If the base is compact and connected, then we can strengthen Proposition \ref{V pluriharmonic} as follows.
\begin{prop}\label{symplectic bundle}
Let $\pi:X \to B$ be a smooth symplectic family, with compact, 2-dimensional fibres, over a \emph{compact, connected} complex manifold $B$, let $P$ be its Poisson structure and let $I$ be its transverse complex structure.

If $(P,I)$ comes from a generalized complex structure, then $\pi:X \to B$ is a \emph{symplectic} fibre bundle.
\end{prop}

\begin{proof}
We prove the existence of local \emph{symplectic} trivializations.  Our technique is to check that in this case Moser's trick (see, eg., \cite{CannasDaSilva}) works in a smooth fibrewise way.

Let $p$ be any point in $B$.  Let $S = \pi^{-1}(p)$ be the fibre over $p$.  Since $X$ is locally trivializable, there exists a neighbourhood $U$ of $p$ such that $\pi^{-1}(U) \iso U \times S$ as smooth fibre bundles.

Given this identification, we may consider two fibrewise symplectic forms: $\omega$, which comes from the Poisson structure $P$, and $\omega_0$, which is constructed by taking $\omega|_S$ and copying it to every fibre in $\pi^{-1}(U)$ according to the local trivialization.  (Then the local trivialization is symplectic for $\omega_0$.)

Let $\omega_t = (1-t)\omega_0 + t\omega$ be an interpolation between $\omega$ and $\omega_0$.  Then
$$\frac{d\omega_t}{dt} = \omega - \omega_0.$$

The conclusions of Proposition \ref{V pluriharmonic} hold.  In particular, since $B$ is compact and connected, $V$ is constant.  But for a given fibre $F$, the classes $[\omega|_F]$ and $[\omega_0|_F]$ in $H^2(F)$ are determined by $V(F)$ and $V(S)$ respectively.  Thus $[\omega|_F]=[\omega_0|_F]$, and there is some 1-form $\mu|_F$ on $F$ such that $\omega|_F - \omega_0|_F = d\mu|_F$.

Since the fibres are compact, and since $\omega$ and $\omega_0$ vary smoothly, we may choose fibrewise $\mu$ smoothly with respect to the fibres (for example, by picking the harmonic representative according to some metric).  Then, as fibrewise forms, we have
$$\frac{d\omega_t}{dt} = \omega - \omega_0 = d\mu$$

On the distinguished fibre $S$, $\omega_t|_S$ is a constant family of symplectic forms.  By possibly taking a smaller neighbourhood around $p$, we can ensure that $\omega_t$ is fibrewise-nondegenerate for all fibres and all $t \in [0,1]$.

Since $\omega_t$ is fibrewise-nondegenerate, we may solve
$$\iota_{v_t} \omega_t = \mu$$
for a smooth vertical vector field $v_t$.  Then
$$\Lie_{v_t} \omega_t = d\iota_{v_t} \omega_t = d\mu = \frac{d\omega_t}{dt}.$$
Since the fibres are compact, we may integrate $v_t$ to a fibre-preserving flow taking $\omega$ to $\omega_0$, which is the symplectic trivialization.
\end{proof}

\begin{example}
We would like to give examples of generalized complex structures on surface bundles for which the fibrewise volume function is non-constant.  In particular, these would not be symplectic fibre bundles.  According to Proposition \ref{V pluriharmonic}, such examples would have a non-compact base, and the fibrewise volume would be pluriharmonic.  We define them as follows:

Let $\tilde{X} \to B$ be a symplectic fibre bundle with flat connection over a noncompact Riemann surface, and suppose that $V : B \to \R$ is a nonconstant, positive pluriharmonic function (for example, a real linear function on the upper half plane $\sr H^+ \subset \C$).  Let $X$ be the same bundle as $\tilde{X}$ but with the symplectic form on each fibre scaled by $V$---that is, if $\omega$ is the symplectic form on the fibres of $\tilde{X}$, then $V\omega$ is the symplectic form on the fibres of $X$.

Then $V\omega$ determines a Poisson structure, $P$, on $X$ and, as usual, the complex structure on $B$ pulls back to a transverse complex structure $I$ on $X$.

\begin{prop}
In this case, $(P,I)$ comes from a generalized complex structure.
\end{prop}
\begin{proof}
The flat connection on $\tilde{X}$ determines an integrable distribution, $N$, complementary to the fibres, which determinines the decomposition of $d$, as in Lemma \ref{d decomposition}.  Since $N$ is integrable, we have a triple complex, and the tensor terms $\theta$ and $\bar\theta$ vanish.  The flat connection is given by $\del+\bar\del$, and since the connection is symplectic for $\omega$, we have that $(\del + \bar\del) \omega = 0$.  This splits according to degree as $\del\omega=0$ and $\bar\del\omega=0$.  Then
\begin{eqnarray*}
\del\bar\del (V\omega) &=& (\del\bar\del V) \omega + (\del V)(\bar\del\omega) - (\bar\del V)(\del\omega) + V(\del\bar\del\omega) \\
&=& 0 + 0 - 0 + 0
\end{eqnarray*}
We see that in this case equation \eqref{dH4} is satisfied:
$$d_S\alpha = \del\bar\del (V\omega) = 0.$$
As we remarked earlier, since the number of complex dimensions of $X$ is 1, the other closedness conditions are trivial, and the choices $\alpha=0$ and $\beta=0$ satisfy the conditions of Theorem \ref{H iff alpha and beta}.
\end{proof}
\end{example}

\subsection{Higher-dimensional smooth symplectic families}

We now consider smooth symplectic families whose fibres might have dimension greater than 2.  In this case, we cannot integrate equation \eqref{dH4} over the fibres, as we did in Proposition \ref{V pluriharmonic}, so the symplectic volume of the fibres has no obvious relation to the form $\alpha$.  The solution to this difficulty is to consider \eqref{dH4} as a condition on the fibrewise cohomology class of the symplectic form.  Corresponding to any fibre bundle, there is a canonical flat connection on the relative, i.e., fibrewise cohomology, called the Gauss-Manin connection.

We will show that, if a smooth symplectic family over a complex manifold comes from a generalized complex structure, then the cohomology of the symplectic form must be \emph{pluriharmonic} with respect to this connection.  If, furthermore, the family is compact and satisfies certain triviality conditions, then it is a symplectic fibre bundle.  The triviality conditions are strong, however, and we give an example of a compact smooth symplectic family over a complex manifold which comes from a generalized complex structure but whose symplectic form is not flat in cohomology.

We describe the Gauss-Manin connection first in the real case for clarity, and then note how it may be complexified.  For details, and a somewhat more general presentation, see \cite{KatzOda}.

\subsection*{Relative cohomology}
There is a natural filtration on the forms on $\pi:X \to B$,
$$\Omega^\bullet(X) = F_0^\bullet \supset F_1^\bullet \supset F_2^\bullet \supset \ldots,$$
where $F_n^m$ consists of the $m$-forms on $X$ generated (over $\Omega^\bullet(X)$) by pullbacks of $n$-forms on $B$.  Let $\Lambda^{n,k}(X) = F_n^{n+k}/F_{n+1}^{n+k}$ be the associated graded object.

Let $S \subset TX$ be the vertical distribution; by $\Omega^k(S)$ we mean the sections of $\^ ^k S^*$.  Then there is a canonical isomorphism,
$$\Lambda^{n,k}(X) = \pi^*\Omega^n(B) \tens \Omega^k(S).$$
Each $\Lambda^{n,\bullet}(X)$ is a differential complex, for the fibrewise differential $d_S = 1 \tens d$, and thus a cohomology,
$$H^{n,k}_{d_S}(X) := \frac{\{\sigma \in \Lambda^{n,k}(X) \;|\; d_S \sigma = 0\}}{d_S\Lambda^{n,k-1}(X)}.$$

\begin{prop}
$H^{n,k}_{d_S}(X)$ is naturally isomorphic to the sections of a finite-dimensional vector bundle over $B$, whose fibre over $x \in B$ is
$$\left(\^ ^n T_x^*B\right) \tens H^k_{dR}\left(\pi^{-1}(x)\right).$$
\end{prop}

We give the correspondence, with proof omitted: a section over $B$ of $(\^ ^n T^*B) \tens H^k_{dR}(S)$ has representatives in each fibre.  These may be chosen smoothly, giving an element of $\Lambda^{n,k}(X)$.  Conversely, given $[\rho]_{d_S} \in H^{n,k}_{d_S}(X)$, let $\rho \in \Lambda^{n,k}(X)$ be a representative, and produce a section of $(\^ ^n T^*B) \tens H^k_{dR}(S)$ by taking, in each fibre, the $d_S$-cohomology of $\rho$.

\begin{rem}
Since $\Lambda^{n+1,k-1}(X) = F^{n+k}_{n+1}/F^{n+k}_{n+2}$ and $\Lambda^{n,k}(X) = F^{n+k}_n/F^{n+k}_{n+1}$, we get a short exact sequence of complexes,
$$0 \to \Lambda^{n+1,k-1}(X) \to F^{n+k}_n/F^{n+k}_{n+2} \to \Lambda^{n,k}(X) \to 0,$$
giving rise to a long exact sequence in cohomology.
\end{rem}

\begin{defn}
The \emph{Gauss-Manin connection},
$$\nabla : H^{n,k}_{d_S}(X) \to H^{n+1,k}_{d_S}(X),$$
is the connecting homomorphism arising from the short exact sequence
$$0 \to \Lambda^{n+1,k-1}(X) \to F^{n+k}_n/F^{n+k}_{n+2} \to \Lambda^{n,k}(X) \to 0.$$
\end{defn}

\begin{prop}
The Gauss-Manin connection is flat.
\end{prop}

\subsection*{Computing the connection}

If we choose a distribution $N \subset TM$ complementary to $S$, then we have $\Lambda^{n,k}(X) \iso \Omega^{n,k}(X)$, where the bidegree $n$ is the degree on $N$ and $k$ the degree on $S$.  As in Lemma \ref{d decomposition}, we get a decomposition $d = d_S + \del + \bar\del + \theta + \bar\theta$.  This $d_S$ on $\Omega^{n,\bullet}$ agrees with the $d_S$ defined on $\Lambda^{n,\bullet}(X)$ under the isomorphism.  Under this isomorphism, then, the Gauss-Manin connection is just $\nabla = \del + \bar\del$, passing to $d_S$-cohomology.

We may complexify the above story, giving a decomposition
$$\C \tens H^{n,k}_{d_S}(X) = \bigoplus_{i+j=n} H^{ij,k}_{d_S}(X),$$
where $i$ and $j$ are the holomorphic and anti-holomorphic degrees on $B$.  With $\del$ and $\bar\del$ again passing to cohomology,
\begin{eqnarray*}
\del &:& H^{ij,k}_{d_S}(X) \to H^{(i+1)j,k}_{d_S}(X) \\
\textnormal{and}\quad \bar\del &:& H^{ij,k}_{d_S}(X) \to H^{i(j+1),k}_{d_S}(X)
\end{eqnarray*}
are the holomorphic and anti-holomorphic parts of $\nabla$.

\subsection*{Conditions on symplectic families}
We now use the Gauss-Manin connection to state and prove higher-dimensional analogues of the results in Section \ref{surface bundles}.

\begin{thm}\label{[omega] pluriharmonic}
Let $\pi : X \to B$ be a smooth symplectic family over a complex manifold with Poisson structure $P$ and transverse complex structure $I$.  Let $\omega$ be its leafwise symplectic structure.

If $(P,I)$ comes from a generalized complex structure on $X$, then
\begin{equation}\label{dH4 cohomology}
\del\bar\del[\omega]_{d_S} = 0.
\end{equation}

In fact, $\del\bar\del[\omega]_{d_S} = 0$ if and only if there exists $\alpha$ such that $d_S\alpha = 2i\del\bar\del\omega$.
\end{thm}

\begin{proof}
In interpreting \eqref{dH4 cohomology}, we consider
$$\omega \in \Lambda^{00,2}(X) = 1 \tens \Omega^2(S),$$
with notation as in the previous section.  Since $d_S\omega=0$, we may reinterpret equation \eqref{dH4},
$$d_S\alpha = 2i\del\bar\del\omega,$$
in cohomology:
\begin{eqnarray*}
[d_S\alpha]_{d_S} &=& [2i\del\bar\del\omega]_{d_S} \\
0 &=& 2i\del\bar\del[\omega]_{d_S}.
\end{eqnarray*}
\end{proof}

\begin{cor}
If $\pi:X \to B$ is a symplectic fibre bundle with compact fibres over a Riemann surface $B$, then the induced Poisson structure and transverse complex structure on $X$ come from a generalized complex structure.
\end{cor}
\begin{proof}
By the existence of local symplectic trivializations of $X$, we see that $\del\bar\del[\omega]_{d_S}=0$.  Then by the second claim of Theorem \ref{[omega] pluriharmonic}, there is some $\alpha$ satisfying equation \eqref{dH4}.  But since $B$ has complex dimension $1$, as per Remark \ref{type 1} equations \eqref{dH1} through \eqref{dH3} are trivially satisfied.
\end{proof}

\begin{rem}
We call the condition \eqref{dH4 cohomology} \emph{pluriharmonicity}, in analogy with the condition on the volume function in Proposition \ref{V pluriharmonic}.  Let $\pi^{-1}(U) \iso F \times U$ be a local trivialization of $\pi:X \to B$.  Let $\{\sigma_1,\ldots,\sigma_k\}$ be a basis for $H^2_{dR}(F)$, the second cohomology of the fibre.  If $\omega \in \Lambda^{00,2}(\pi^{-1}(U))$, then
$$[\omega|_{X_U}]_{d_S} = f_1\sigma_1 + \ldots + f_k\sigma_k$$
for some functions $f_1,\ldots,f_k$.

In this notation, the pluriharmonicity condition is just
$$\del\bar\del f_i=0,\;\forall i.$$
\end{rem}

\begin{prop}\label{omega flat on trivial bundle}
Let $\pi:X \to B$ be a smooth symplectic family over a \emph{compact} complex manifold, with Poisson structure $P$, complex structure $I$, and fibrewise symplectic structure $\omega$.  Furthermore, suppose $H^{00,2}_{d_S}(X)$ has a flat trivialization over $B$.

If $(P,I)$ comes from a generalized complex structure, then $[\omega]_{d_S}$ is flat.
\end{prop}
\begin{proof}
In the flat trivialization, $[\omega]_{d_S} = f_1\sigma_1 + \ldots + f_k\sigma_k$, for flat basis sections $\sigma_i$ and functions $f_i$.  By Theorem \ref{[omega] pluriharmonic} and the above remark, each $f_i$ is pluriharmonic.  Since $B$ is compact, by the maximum principle this means that each $f_i$ is locally constant.
\end{proof}

\begin{lem}\label{symplectic bundle 2}
Let $\pi:X \to B$ be a smooth symplectic family over a connected complex manifold, with fibrewise symplectic form $\omega$.

If $[\omega]_{d_S}$ is flat, then $\pi:X \to B$ is a \emph{symplectic} fibre bundle for the symplectic structure $\omega$. 
\end{lem}
\begin{proof}
The argument is almost the same as for Proposition \ref{symplectic bundle}; in that case, it was necessary that $\omega$ have the same cohomology class in nearby fibres in a local trivialization---a fact which in this case follows directly from the hypothesis that $[\omega]_{d_S}$ is flat.
\end{proof}

We summarize the situation for two particular cases where the hypotheses of Propositions \ref{omega flat on trivial bundle} and \ref{symplectic bundle 2} hold:
\begin{thm}\label{symplectic}
Let $\pi: X \to B$ be a smooth symplectic family over a compact complex manifold, which comes from a generalized complex structure.  If $B$ is simply connected, or if $\pi: X \to B$ is a trivial bundle, then in fact $X$ is a symplectic fibre bundle over $B$.
\end{thm}
\begin{proof}
If $B$ is simply connected, then the Gauss-Manin connection trivializes $H^{00,2}_{d_S}(X)$, or if $\pi:X \to B$ is trivial, this induces a trivialization of $H^{00,2}_{d_S}(X)$.  In either case, the hypotheses of Proposition \ref{omega flat on trivial bundle} are satsified, and thus the hypotheses of Lemma \ref{symplectic bundle 2}.
\end{proof}

\subsection{Generalized Calabi-Yau manifolds}\label{generalized Calabi-Yau manifolds}
A \emph{generalized Calabi-Yau} manifold (originally described in \cite{Hitchin2003}) is a generalized complex manifold whose canonical line bundle is generated by a global $d_H$-closed spinor.

Let $\pi:X \to B$ be a generalized complex smooth symplectic family over a complex manifold.  If $B$ is Calabi-Yau---that is, if its canonical bundle has a closed generating section $\rho_B$---then the spinor
$$\rho = e^{i\omega}\^\rho_B$$
on $X$ is $d_H$-closed for some closed 3-form $H$ and generates the canonical bundle $\kappa$ for the generalized complex structure.  Thus $X$ is generalized Calabi-Yau.  We remark that Example \ref{counterexample} below is generalized Calabi-Yau in this way.

\begin{example}\label{counterexample}
In the higher-dimensional case, in contrast with surface bundles, the fact that a generalized complex smooth symplectic family is compact and connected does not imply that it will be a symplectic fibre bundle.  We give as a counterexample a generalized complex structure on a $T^4$-bundle over $T^2$.

Consider the flat trivial bundle
$$X = T^4 \times \C \to \C.$$
Let $\theta_1,\theta_2,\theta_3,\theta_4$ be the standard basis of 1-forms for $T^4$, and let $x+iy$ be the complex coordinate on the base.  Let
$$\omega = \theta_1\^\theta_2 \,+\, \theta_3\^\theta_4 \,+\, x\,\theta_1\^\theta_3.$$
Let $N \subset TX$ be the horizontal distribution, giving a decomposition $d = d_S + \del + \bar\del$ and an extension of $\omega$ to $X$.  Then $d_S\omega = 0$ and $\del\bar\del\omega=0$, but $\nabla\omega \neq 0$---indeed, $\nabla[\omega]_{d_S} \neq 0$.

Let $\Lambda = \Z + i\Z \subset \C$ be the standard integral lattice.  We will define a monodromy homomorphism $\lambda : \Lambda \to \Aut(T^4)$ as follows: in the imaginary direction, $\lambda(i) = \Id$, and in the real direction, $\lambda(1)$ is the automorphism of $T^4$ which takes $\theta_2$ to $\theta_2 - \theta_3$ and leaves the others fixed.  Then
$$\lambda(1)^* : \omega \mapsto \theta_1\^\theta_2 + \theta_3\^\theta_4 + (x-1)\,\theta_1\^\theta_3.$$

Thus, at any $m+in \in \Lambda \subset \C$,
\begin{eqnarray*}
\lambda(m+in)^*(\omega(m+in)) &=& \theta_1\^\theta_2 + \theta_3\^\theta_4 + (m-m)\,\theta_1\^\theta_3 \\
&=& \omega(0)
\end{eqnarray*}
Then $\omega$ passes to $\tilde\omega$ on the flat bundle $\tilde{X} = X / \Lambda$.  It is still the case that $d_S\tilde\omega = 0$ and $\del\bar\del\tilde\omega=0$, so with the choices $\alpha=0$ and $\beta=0$ as in Theorem \ref{H iff alpha and beta}, we see that these data come from a generalized complex structure.  But $[\tilde\omega]_{d_S}$ is still not flat, so $(\tilde{X},\tilde\omega)$ is not a symplectic fibre bundle.
\end{example}

%\input{biblio}

%\end{document}
% ----------------------------------------------------------------

%\documentclass{amsart}
%\usepackage{graphicx,amsmath,amscd,amssymb}
% ----------------------------------------------------------------

%\input{declarations}

% ----------------------------------------------------------------
%\begin{document}

\chapter{Nondegenerate type change}

This chapter is a joint project with Marco Gualtieri.  Whereas in Chapter 3 we considered only \emph{regular} generalized complex structures, in this chapter we study generalized complex structures with type change.  Subject to a nondegeneracy condition, we show that a type change locus inherits the structure of a generalized Calabi-Yau manifold---in particular, given some compactness assumptions, it is a smooth symplectic family over an elliptic curve, of the sort discussed in the last chapter.  Finally, we prove a normal form theorem for a tubular neighbourhood of a nondegenerate type-change locus; we use a Moser-type argument, integrating a family of generalized vector fields to generalized diffeomorphisms (see Definition \ref{Courant flow}) which take the initial neighbourhood to the normal form.

\section{The setting}\label{the setting}

Let $M$ be a $(2n \geq 4)$-dimensional manifold and let $J$ be a generalized complex structure on $M$ with canonical line bundle $K\subset \^ ^\bullet T^*_\C M$ (Definition \ref{almost gc}).  $K$ projects to $\^ ^0 T^*_\C M = \C \times M$, giving a section $s \in K^*$.  Throughout this chapter, suppose that $s$ intersects the zero section transversally---we say that $J$ has \emph{nondegenerate type change}.

$J$ is symplectic type wherever $s \neq 0$.  Let $D\subset M$ be the type-change locus, which equals the vanishing locus of $s$.  Since $s$ is a nondegenerately vanishing section of a complex line bundle $s$, $D$ is a smooth $(2n-2)$-real-dimensional manifold.  We assume furthermore that $D$ is compact and connected.

\begin{rem}
Though many of our results hold for any $2n \geq 4$, a case of special interest is $2n=6$, since the 4-dimensional case has already been studied along similar lines (see \cite{CavalcantiGualtieri}), and since the 6-dimensional case is of particular interest in compactifications of string theory.
\end{rem}

\section{Generalized Calabi-Yau structure on the type change locus}

\begin{prop}\label{drho holomorphic}
Given the hypotheses of the previous section, $J$ is type 2 on $D$, i.e. it has 2 complex dimensions.  $N^*D$ is complex (and thus the other complex dimension lies in $D$).  Furthermore, if $\rho = \rho_0+\rho_2+\ldots$ is a local pure spinor in $\Gamma(K)$ near $x\in D$, then $d\rho_0|_x \in N^*_{1,0}D$.
\end{prop}
\begin{proof}
Like $s$, $\rho_0$ vanishes nondegenerately at $D$.  Thus $d\rho_0$ is nonvanishing at $D$ and conormal to $D$.

Since $J$ is integrable (Definition \ref{define integrability}), $\exists\; X+\xi \in \Gamma(T_\C M \dsum T_\C^* M)$ such that $d\rho_0 = \iota_X \rho_2 + \xi\^\rho_0$. On $D$, $d\rho_0|_D = \iota_X \rho_2|_D$.  Since $d\rho_0|_D$ is nonzero, so is $\rho_2|_D$, and thus $\rho_2$ defines a 2-dimensional transverse complex structure on $D$.

But then $d\rho_0|_x = \iota_X \rho_2|_x$ is a $(1,0)$ form, whose real and imaginary parts generate $ND$.  So $ND$ is complex and $d\rho_0|_x \in N^*_{1,0}D$
\end{proof}

The generalized complex structure on $M$ will reduce to a generalized complex structure on $D$:

\begin{prop}
There is a unique pure spinor $\sigma \in \Gamma(\^ ^\bullet T_\C^*D)$ defined globally on $D$ such that, for any local pure spinor, $\rho \in \Gamma(K)$, generating the generalized complex structure near $D$,
$$\rho|_D = d\rho_0 \^ \sigma.$$
\end{prop}

\begin{proof} Let $x \in D$ and let $\rho=\rho_0+\rho_2+\ldots$ be a nonvanishing section of $K$ near $x$.  Since $ND|_x$ is complex, by the decomposition of pure spinors (Proposition \ref{pointwise form of J}),
$$\rho|_x = \Omega \^ \tau$$
for some $\Omega \in N^*_{1,0}|_xD$ and $\tau \in \^ ^\bullet T^*_\C|_xM$.

By Proposition \ref{drho holomorphic}, $d\rho_0|_x \in N^*_{1,0}D|_x$ also, and since $N^*_{1,0}D|_x$ is a complex line and $d\rho_0 \neq 0$, $\Omega = c\,d\rho_0|_x$ for some $c \in \C$.  Let $\sigma_x := c\tau$; then
$$\rho|_x = d\rho_0|_x \^ \sigma_x.$$

If $\rho' = f\rho \in \Gamma(K)$ is another choice of generating spinor near $x$, for some nonvanishing smooth function $f$, then
$$d\rho'_0 = d(f\rho_0) = \rho_0 df + f\,d\rho_0$$
so on $D$, $d\rho'_0|_x = f\,d\rho_0|_x$.  As with $\rho$, we have 
\begin{eqnarray*}
\rho'|_x &=& d\rho'_0|_x \^ \sigma'_x \\ 
f\rho|_x &=& f\,d\rho_0|_x \^ \sigma'_x\\
d\rho_0|_x \^ \sigma_x &=& d\rho_0|_x\^ \sigma'_x
\end{eqnarray*} 
So $\sigma_x$ and $\sigma'_x$ differ by some multiple of $d\rho_0|_x$.  In particular, they agree when pulled back to $D$.  We take $\sigma$ to be the pullback of $\sigma_x$ to $D$ at each point.
\end{proof}

\begin{rem}
For any vector field $Y$ for which $\iota_Y d\rho_0$ does not vanish (such a $Y$ certainly exists locally), $\sigma$ is equal to the pullback of
$$\tilde\sigma = \frac{\iota_Y \rho}{\iota_Y d\rho_0}$$
to $D$.
\end{rem}

\begin{example}\label{local model}
Let $w$ and $z$ be coordinates for $\C^2$, let $\omega$ be a symplectic form on $\R^{2n}$, and let closed 2-form.  The generalized complex structure on $\C^2 \times \R^{2n}$ given by the pure spinor
$$\rho = (w + dw\^dz)\^e^{B+i\omega}$$
satisfies our hypotheses.  The type change locus $D$ corresponds to $w=0$, and
\begin{eqnarray*}
\sigma &=& \frac{\iota_{\del_w} \rho}{\iota_{\del_w} d\rho_0} \\
&=& dz\^e^{B+i\omega}
\end{eqnarray*}

As we remarked, this $\sigma$ is only uniquely defined when pulled back to $D$
\end{example}

In general, to characterize the local structure near a point in $D$, we use the following two theorems.

\begin{thm}[Abouzaid, Boyarchenko \cite{AbouzaidBoyarchenko}]\label{local splitting} 
If $M$ is a generalized complex manifold and $p\in M$, then there is a neighourhood of $p$ which is isomorphic (via diffeomorphism and $B$-transform) to a product of a symplectic manifold and a generalized complex manifold which is of complex type at the image of $p$.
\end{thm}

\begin{thm}[Cavalcanti, Gualtieri \cite{CavalcantiGualtieri}]\label{4d model}
If $M$ is a 4-dimensional generalized complex manifold whose type changes nondegenerately (i.e., $s \in \Gamma(K^*)$ is transverse to the zero section as above), then about any point in the type change locus there is a neighbourhood isomorphic (via diffeomorphism and $B$-transform) to a neighbourhood of the origin in $\C^2$, with generalized complex structure given by the pure spinor
$$\rho = w + dw\^dz,$$
where $w$ and $z$ are complex coordinates.
\end{thm}

Under our hypotheses, the type change locus $D$ is of type 2, so by Theorem \ref{local splitting}, about any point $x \in D$ there is a neighbourhood which is a product of a
symplectic manifold with a 4-real-dimensional manifold which is of complex type at $x$.  Invoking
Theorem \ref{4d model}, we conclude:
\begin{prop}\label{local model holds}
About any point $x \in D$ there is a neighbourhood which is isomorphic to a neighbourhood of the origin in Example \ref{local model}.
\end{prop}

From this local model, we conclude that the reduced pure spinor $\sigma$ on $D$ is generalized complex and closed; thus,
\begin{prop}
Under our hypotheses, $D$ is a generalized Calabi-Yau manifold (as in Chapter 3, Section \ref{generalized Calabi-Yau manifolds}).
\end{prop}
\begin{rem}
We do not address explicitly the possibilty of a twisting 3-form, $H$, as was discussed in the previous chapter.  If we suppose that $J$ is $H$-integrable, for real closed $H \neq 0$, then Proposition \ref{local model holds} still holds, but the allowed isomorphisms include non-closed $B$-transforms which shift $H$, and thus the spinor $\sigma$ is $d_H$-closed, and $D$ is an $H$-twisted generalized Calabi-Yau manifold.
\end{rem}

In particular, $\sigma$ determines a symplectic foliation on $D$ and a 1-dimensional transverse complex structure.  According to our assumptions, $D$ is compact and connected.  For $D$, but also for any suitable generalized Calabi-Yau manifold, we may say the following.

\begin{prop}\label{family over elliptic curve}
If $D$ is a compact, connected, generalized Calabi-Yau manifold of type 1 everywhere, and it has at least one compact symplectic leaf, then it is a smooth symplectic family over an elliptic curve (Definition \ref{define smooth symplectic family}).  Furthermore, if $\dim(D)=4$ then $D$ is in fact a symplectic fibre bundle.
\end{prop}
\begin{proof}
The degree-1 piece, $\sigma_1$, of $\sigma$ defines the complex structure transverse to the symplectic leaves.  Since $\sigma_1$ is closed and annihilates tangents to the leaves, it is basic.

Let $S \subset D$ be a compact symplectic leaf and let $U$ be a tubular neighbourhood of $S$.  A contraction of $U$ to $S$ induces an isomorphism of their de Rham cohomology, and since $\sigma_1$ annihilates $TS$, it is cohomologically trivial on $S$ and hence on $U$.  Thus on $U$ there is a holomorphic function $z$ ($\sigma_1 = dz$) which is constant on each leaf, and which identifies the leaf space of (a possibly smaller neighbourhood) $U$ with a disc $B\subset\C$.  Then $U$ is a fibre bundle over $B$.

In particular, $S$ has trivial holonomy.  Then by a result of Brunella \cite{Brunella} on transversely holomorphic foliations of complex codimension 1, since one leaf is compact, all the leaves are compact.

The leaf space, $E$, is locally modelled on complex discs, and the projection has local trivializations, thus $D$ is a fibre bundle over $E$.   $E$ is compact complex curve carrying a closed complex 1-form, $dz$; thus $E$ is an elliptic curve.

The final claim, which says that if $\dim(D)=4$ then the fibration has local \emph{symplectic} trivializations, follows from Proposition \ref{symplectic bundle}.
\end{proof}

In our setting, where $D$ is the type-change locus of $M$, $\dim(D)=4$ occurs when $2n=\dim(M)=6$, which is the case of special interest to us.

\section{Local structure near the type change locus}\label{local structure near D}

We continue to assume that $J$ is a generalized complex structure on $M$ with nondegenerate type change, as in Section \ref{the setting}, and that $D$ is its type change locus.  Using $J$, we will define a canonical generalized complex structure $J_N$ on the total space of the normal bundle $ND$, and give some information about it.  Then we will prove the existence of a Courant isomorphism $\psi_0$ between a neighbourhood of the zero section of $ND$ and a neighbourhood of $D$ and  such that $\psi_0^*J = J_N$.

\subsection{Linear structure on $ND$}

We define the generalized complex structure on $ND$ through a somewhat inelegant procedure of taking a limit under dilations.  A more direct definition is not clear to us: because of the type change, any definition of $J_N$ must implicitly take a derivative, so there is a limit happening somewhere; but the trick used in \cite{CavalcantiGualtieri} in dimension $2n=4$ does not work for $2n>4$.

\begin{defn}
Suppose we have chosen some embedding of a neighbourhood of $0_{ND} \subset ND$ into $M$ so that, abusing notation, $ND \subset M$.  This determines a family of dilations $\phi_t$ of $ND \subset M$ which scale each fibre by $t\in\R$.  The linearization of $J$ with respect to this choice is
$$J_N := \lim_{t\to0} \phi_t^* J = \lim_{t\to0} J_t,$$
where $J_t=\phi_t^* J$.  If $\rho \in \Gamma(K)$ is a local pure spinor generating $J$ near $x\in D$, the linearization of $\rho$ with respect to this choice is
$$\rho^N = \lim_{t\to0} t^{-1} \phi_t^* \rho = \lim_{t\to0} \rho^t.$$
\end{defn}

\begin{rem}
$\rho^t$ as defined above generates $J_t$.  We include the factor $t^{-1}$ in $\rho^t$, which doesn't change the generalized complex structure, in order to ensure the limit exists.  If $\rho^N$ exists and is generalized complex (which we will show), then it generates $J_N$.  In what follows, we are concerned primarily with $\rho_0$ and $\rho_2$, since for a generically symplectic structure these determine $\rho$.  Indeed, wherever $\rho_0 \neq 0$,
$$\rho = \rho_0 e^{\frac{\rho_2}{\rho_0}}.$$
\end{rem}

\begin{lem}\label{decompose rho_2 near D}
If $\rho=\rho_0+\rho_2+\ldots$ is a local pure spinor generating $J$ near $x\in D$, then there is a 1-form $\alpha$ and a 2-form $\theta$ such that
$$\rho_2 = d\rho_0\^\alpha + \rho_0\theta.$$
Furthermore, the pullbacks to $D$, $\iota_D^*\alpha$ and $\iota_D^*\theta$, do not depend on the choice of $\alpha$ and $\theta$.
\end{lem}

\begin{proof}  Since $J$ is integrable, the 2-form $\rho_2/\rho_0$ is closed where it is defined, thus
$$\rho_0 d\rho_2 = d\rho_0\, \rho_2.$$
By continuity, this will be true even where $\rho_0=0$.  Then $\rho_2$ is the sum of a piece vanishing when $\rho_0=0$ and a piece which annihilates $d\rho_0$.  Therefore we may write
\begin{equation}
\rho_2 = d\rho_0 \^ \alpha + \rho_0 \theta
\end{equation}

For uniqueness, we rely on the fact that $\iota_D^*d\rho_0 = 0$.  For different such choices $\alpha,\theta$ and $\alpha',\theta'$,
$$0 = \rho_2 - \rho_2 = d\rho_0 \^ (\alpha - \alpha') + \rho_0 (\theta - \theta').$$
On $D$, when $\rho_0=0$, we see that $\alpha - \alpha'$ is a multiple of $d\rho_0$, and thus $\iota^*(\alpha - \alpha')=0$.  Wedging the above equation with $d\rho_0$, we get
$$0 = 0 + \rho_0 d\rho_0 \^ (\theta - \theta').$$
Away from $D$, when $\rho_0 \neq 0$, we have $d\rho_0 \^ (\theta - \theta')=0$, and by continuity this must extend to $D$.  Again, $\iota^*(\theta-\theta')=0$s.
\end{proof}

\begin{prop}
The limit $\rho^N$ exists, and as a spinor on $ND$ is independent of the choice of embedding $ND\subset M$.
\end{prop}
\begin{proof}
First we look at $\rho_0$.  In this case we are just differentiating a function:
$$\rho_0^N = \lim_{t\to0} t^{-1}\phi_t^*\rho_0 = w,$$
where $w$ is just $d\rho_0$ interpreted as a linear function on $ND$.  This doesn't depend on the embedding of $ND$.

Now we look at $\rho_2$.  By Lemma \ref{decompose rho_2 near D}, we have
\begin{eqnarray*}
\rho_2^N &=& \lim_{t\to0} t^{-1}\phi_t^*\rho_2 \\
&=& \lim_{t\to0} d(t^{-1}\phi_t^*\rho_0) \^ \phi_t^*\alpha + t^{-1}(\phi_t^*\rho_0)\, \phi_t^*\theta \\
&=& dw \^ \iota_D^*\alpha + w\, \iota_D^*\theta
\end{eqnarray*}
This expression does not depend on the choice of embedding, nor---as per Lemma \ref{decompose rho_2 near D}---does it depend on the choices of $\alpha$ and $\beta$.

Thus without loss of generality we may compute $\rho_0$ and $\rho_2$ for a convenient choice of coordinates.  As per Proposition \ref{local model holds}, near any point in $D$ there are complex coordinates $w$ and $z$, and a collection of real coordinates which we denote $x$, for which $J$ is generated by
$$\rho = (w + dw\^dz)\^e^{B+i\omega}.$$
($\omega$ is a symplectic form in the $x$'s and $B$ is a closed 2-form.)  These coordinates give us a choice of embedding of $ND$ (fibres have constant $z$ and $x$), and thus a means of scaling (by $w \mapsto tw$).  In the limit, with coordinates $w$ and $z$ reinterpreted as coordinates on $ND$, we have
\begin{equation}\label{normal form on ND}
\rho^N = (w + dw\^dz)\^e^{\iota_D^*B + i\omega}.
\end{equation}

We conclude that, regardless of our choices, $\rho^N = \rho_0^N \exp(\rho_2^N / \rho_0^N)$ exists and is generalized complex; thus it generates $J_N$.
\end{proof}

\begin{cor}
$J_N$ exists and is generalized complex.
\end{cor}

\begin{rem}\label{structure on ND}
\eqref{normal form on ND} gives us a concrete local description of $J^N$ on $ND$.  It is given by a linear holomorphic Poisson structure $w\del_w\^\del_z$, with $w\in\Gamma(N^*D)$, times a local symplectic leaf in $D$.  As per Remark \ref{S and I well-defined} in Chapter 3, the $B$-field $\iota_D^*B$ is globally defined on $D$ only up to forms of complex bidegree $(2,0) + (1,1)$.  If $\iota_D^*B$ has a closed, real extension to $\Omega^2(D)$, then we may subtract it away by a $B$-transform, but this may not be possible in general.
\end{rem}

\subsection{Equivalence of local and linear structures near the type change locus}

\begin{rem}\label{smooth convergence}
We make the general observation---a consequence of Hadamard's lemma---that if a smooth tensor $\theta$ on the total space of $ND$ goes to zero under the dilation, i.e.,
$$\lim_{t\to0} \phi_t^*\theta = 0,$$
then the convergence is smooth in $t$.
\end{rem}

Recall that, as described in Chapter 2, Section \ref{Courant automorphisms}, a \emph{generalized vector field} integrates to a family of Courant automorphisms; and if we differentiate the pushforward (or pullback) of a generalized complex structure under these automorphisms, we have the \emph{infinitesimal action} of a generalized vector field on a generalized complex structure.
\begin{lem}\label{J(closed) acts trivially} 
If $\theta$ is a closed 1-form and $J$ a generalized complex structure, then the infinitesimal action of the generalized vector field $J(\theta)$ on $J$ is trivial.
\end{lem}

\begin{prop}\label{neighbourhood of D}
The generalized complex structures $J$ and $J_N$ are isomorphic in a neighbourhood of $D$. 
\end{prop}
\begin{proof}
Fix an embedding of $ND$ in $M$.  We will show that all of the structures in the family $J_t,\;t\in[0,1]$ are isomorphic in a neighbourhood of $D$, by defining a family $\psi_t$ of Courant automorphisms fixing $D$ so that $J_t = \psi_t^* J$.

$\psi_t$ will be defined by integrating a family, $v_t$, of generalized vector fields vanishing at $D$ (Definition \ref{Courant flow}).  If we can find $v_t$ such that $\Lie_{v_t} J_t = \dot{J_t},$ then the resulting $\psi_t$ (defined in a neighbourhood of $D$) satisfies our requirements.

For $t>0$, we already have $J_t = \phi_t^*J$.  The family $\phi_t$ is generated by the vector field $t^{-1}\chi = t^{-1}\chi$, where $\chi = \Re(w\del_w)$ is the Euler vector field on the fibres.  Note that $\phi_t^* \chi = \chi$ for any $t>0$, so $t^{-1}\chi$ blows up at $t=0$, and does not itself satisfy our requirements.

As in Example \ref{local model} and \eqref{normal form on ND}, let $dz$ be the degree-1 part of $\sigma$, pulled back to $ND$. For $t>0$, let
$$v_t = t^{-1} \phi_t^*(\chi - \Re(J(dz))).$$
Since $\chi$ and $dz$ are invariant under $\phi_t^*$, $v_t=t^{-1}(\chi - \Re(J_t(dz))).$  Since $dz$ is closed, by Lemma \ref{J(closed) acts trivially},
$$\Lie_{v_t} J_t = \Lie_{t^{-1}\chi} J_t = \dot{J}_t.$$

Inspecting the normal form on $ND$, equation \eqref{normal form on ND}, we see
$$\lim_{t\to0} J_t(dz) = J_N(dz) = w\del_w.$$
Then
$$\lim_{t\to0} t\,v_t = \lim_{t\to0} (\chi - \Re(J_t(dz))) = 0.$$
As per Remark \ref{smooth convergence}, $\displaystyle{\lim_{t\to0} v_t}$ exists.  Then $v_t$ integrates to $\phi_t$ near $D$, for $t\in[0,1]$, and
$$\phi_0^*J = \lim_{t\to0} \psi_t^*J = J_N.$$
\end{proof}

Combining Proposition \ref{neighbourhood of D} with Remark \ref{structure on ND},
\begin{prop}
If $D\subset M$ is the nondegenerate-type-change locus of a generalized complex manifold, then (as we have said) it inherits a generalized Calabi-Yau form, $\sigma$, of type 1, and it has a neighbourhood which is isomorphic to a neighbourhood of the zero section in the following construction:
\end{prop}

\begin{example}
Let $p:X \to D$ be a holomorphic line bundle over $D$, with canonical bundle $X^*_{1,0}$.  Then $X^*_{1,0} \^ p^*\sigma$ is a well-defined spinor bundle giving a regular generalized complex structure of type 2 on the total space of $X$.  If $\bar{L}$ is the $-i$-eigenbundle of this structure, then let $\beta \subset \Gamma(\^ ^2\bar{L})$ be a holomorphic Poisson structure on $X$ which is linear in the fibres.  Then $\beta$ deforms $X^*_{1,0} \^ p^*\sigma$ to give a generalized complex structure with nondegenerate type change along the zero section.
\end{example}

%\input{biblio}

%\end{document}

%\documentclass{amsart}
%\usepackage{graphicx}
%\usepackage{amsmath,amscd,amssymb,amsthm}%,accents}

%\input{declarations}

% ----------------------------------------------------------------
%\begin{document}

%\title{On generalized complex torus bundles of mixed type}
%\author{Michael Bailey}
%\maketitle

%\section{Introduction}

\chapter{Generalized complex flat principal bundles}

In this chapter, we study generalized complex manifolds of mixed type with group symmetry.  If we suppose that the group action (in some generalized sense) is free and proper, then we will have something like a principal bundle.

A generalized complex structure induces a Poisson structure, and if the symplectic foliation is complementary to the orbits, we will have something like a flat connection.  We study this case in Section \ref{free and proper group actions}.  In order to do so, in Section \ref{flat bundle section} we define a generalized version of equivariant, flat fibre bundles, and then we prove that free and proper generalized complex group actions (satisfying the complementarity condition) are examples of such bundles.

Classifications of bundles with flat connections are well-known, and we prove similar classification results in the generalized geometric context (see Proposition \ref{classification 1}).  As a consequence of the Courant automorphism group being larger than the diffeomorphisms, the classification includes data corresponding to ``non-geometric'' degrees of freedom.  We describe this concretely in the case of generalized complex principal torus bundles, in Section \ref{torus bundle section}.

\section{Definitions}

\subsection{Courant algebroids}\label{Courant algebroid section}

In this section we consider generalized complex geometry from the perspective of abstract Courant algebroids.  Our presentation of this formalism is close to that in \cite{BursztynCavalcantiGualtieri}.  For proofs, and details, see \cite{Gualtieri2011}.

\begin{defn}\label{Courant algebroid}
A Courant algebroid $E$ on a manifold $M$ is a vector bundle on $M$ with the additional data of a bracket, $[\cdot,\cdot]:\Gamma(E)\times\Gamma(E)\to\Gamma(E)$, a nondegenerate symmetric bilinear pairing, $\pair{\cdot,\cdot}:E^2 \to E$, and an \emph{anchor map}, $\pi:E\to TM$, such that, if $e_1,e_2,e_3\in\Gamma(E)$ and $f \in C^\infty(M)$, then
\begin{enumerate}
\item $[e_1,[e_2,e_3]] = [[e_1,e_2],e_3] + [e_2,[e_1,e_3]]$,
\item $[e_1,fe_2] = f[e_1,e_2] + \left(\pi(e_1)\cdot f\right) e_2$, \label{C2}
\item $\pi(e_1)\pair{e_2,e_3} = \pair{[e_1,e_2],e_3} + \pair{e_2,[e_1,e_3]}$, and
\item $2[e_1,e_1] = \pi^*(d\pair{e_1,e_1})$. \label{C4}
\end{enumerate}
We identify $E^*$ with $E$ by the nondegenerate pairing, explaining the meaning of \eqref{C4}.  $E$ is called \emph{exact} if the following sequence is exact:
$$0 \to T^*M \oto{\pi^*} E \oto{\pi} TM \to 0.$$
By convention, for an exact Courant algebroid we consider $T^*M \subset E$ to be an inclusion, and omit mention of $\pi^*$.
\end{defn}

\begin{defn}\label{Courant isomorphism}
Let $E_M$ and $E_N$ be Courant algebroids on $M$ and $N$, and let $\phi : M \to N$ be a diffeomorphism.  A Courant isomorphism covering $\phi$ is a vector bundle map $\Phi : E_M \to E_N$ covering $\phi$ which respects the structures $[\cdot,\cdot]$, $\pair{\cdot,\cdot}$ and $\pi$.  (It follows that $\pi\comp\Phi = \phi_*\comp\pi$.)
\end{defn}

There are nontrivial automorphisms of $E$ over the identity diffeomorphism, consisting precisely of the $B$-transforms:
\begin{defn}
If $B$ is a closed $2$-form on a manifold $M$, the $B$-transform of a Courant algebroid $E$ on $M$ is the map
$$e^B : e \mapsto e + \iota_{\pi(e)}B.$$
\end{defn}

Any exact Courant algebroid on $M$ is isomorphic (by a choice of isotropic splitting of $T^*M \to E \to TM$) to a \emph{standard Courant algebroid} $\Tc_H M = TM \dsum T^*M$; in this case $\pi$ projects to $TM$, $\pair{\cdot,\cdot}$ is $\frac{1}{2}$ the natural pairing between $TM$ and $T^*M$, and if $X,Y \in \Gamma(TM)$ and $\xi,\eta \in \Gamma(T^*M)$ then
$$[X+\xi,Y+\eta] = [X,Y] + \Lie_X \eta + \iota_Y d\xi + \iota_x\iota_y H$$
for some closed 3-form $H$.  If $H=0$ we will just write $\Tc M$ for $\Tc_H M$.

\subsection{Generalized complex structures on Courant algebroids}\label{generalized complex section}

\begin{defn}\label{abstract generalized complex}
A \emph{generalized complex structure} $J$ on a manifold equipped with an exact Courant algebroid $E$ is a fibrewise complex structure on $E$ ($J : E \to E$, $J^2=-1$) which is $\pair{\cdot,\cdot}$--orthogonal and whose $+i$-eigenbundle is $[\cdot,\cdot]$--involutive (in $\C \tens E$).  A generalized complex isomorphism is a Courant isomorphism which respects the generalized complex structures.
\end{defn}

\begin{rem}\label{P and I}
The map $\pi \comp J \comp \pi^* : T^*M \to TM$ induces a Poisson bivector on $M$.  The kernel of this map, i.e., $J(T^*M) \cap T^*M$, inherits a complex structure.  That is, $J$ induces a Poisson structure, and a complex structure transverse to its symplectic foliation.
\end{rem}

\begin{example}\label{symplectic}
If $\omega$ is a symplectic form on $M$, then up to isomorphism there is a unique generalized complex structure inducing $\omega^{-1}$ as its Poisson structure:
$$J_\omega =
\left[\begin{array}{cc}
0 & -\omega^{-1} \\
\omega & 0
\end{array}\right]
: \Tc M \to \Tc M$$
\end{example}

\begin{example}\label{twisted complex}
If $I : TM \to TM$ is a complex structure, then there is a family of generalized complex structures inducing $I$ (and trivial Poisson structure), called \emph{twisted} complex structures.  Up to isomorphism, they correspond to
$$J_{I,H} = 
\left[\begin{array}{cc}
-I & 0 \\
0 & I^*
\end{array}\right]
: \Tc_H M \to \Tc_H M.$$
In this case, the integrability condition entails that $H$ must be a real closed 3-form of bidegree $(1,2) + (2,1)$ in the Dolbeault complex.
\end{example}

A generalized complex structure is invariant under a $B$-transform precisely when $B$ is a real closed 2-form vanishing on the symplectic foliation and of bidegree $(1,1)$ in the transverse complex structure.  In particular, $J_\omega$ has no nontrivial $B$-transforms.

\begin{defn}\label{Courant product}
If $E_M$ and $E_N$ are Courant algebroids on manifolds $M$ and $N$ respectively, then we define the \emph{product algebroid} $E_M \times E_N$ on the manifold $M \times N$ as follows.

The total space of $E_M \times E_N$ is just the (non-fibred) Cartesian product, the anchor map is the product map, and the pairing is the product pairing with $E_M \perp E_N$.  Transversely flat sections of $E_M \times 0_N$, i.e., those which are flat along $N$, inherit a bracket from $E_M$, and similarly for transversely flat sections of $E_N$.  The bracket between transversely flat sections of $E_M$ and $E_N$ vanishes.  All other brackets are determined from these along with the Leibniz rule (Axiom (2) in Definition \ref{Courant algebroid}).

If $J_M$ and $J_N$ are generalized complex structures on $E_M$ and $E_N$, then their product, $J_M \times J_N$, is a generalized complex structure on $E_M \times E_N$
\end{defn}

\begin{defn}
A generalized complex structure is \emph{regular} at a point $x$ if its induced Poisson structure is regular at $x$.
\end{defn}

\begin{thm}[Gualtieri \cite{Gualtieri2011}]\label{normal form}
If a generalized complex structure is regular at $x$, then there is a neighbourhood of $x$ on which it is isomorphic to $J_\omega \times J_I$ on a subset of $\Tc \R^{2n} \times \Tc \C^k$, for the standard symplectic structure $\omega$ on $\R^{2n}$ and complex structure $I$ on $\C^k$.
\end{thm}

\begin{rem}
In the above theorem, to decompose a neighbourhood in $M$ as product $U \subset \R^{2n} \times \C^k$, one must specify a complementary foliation to the symplectic foliation, which will correspond to the copies of $\C^k$.  This choice is not unique.  In fact, we augment the result with the following trivial observation of the original proof (see \cite{Gualtieri2011}):
\end{rem}
\begin{prop}\label{augmented normal form}
In Theorem \ref{normal form}, \emph{any} foliation complementary to the symplectic foliation, with respect to which the leafwise symplectic form is invariant, may be chosen to correspond to the complex leaves in the normal form.
\end{prop}

\subsection{Equivariance}

\begin{defn}
In this chapter, we will consider only left actions.  If $G$ is a Lie group, then a \emph{generalized (left) $G$-action} on a manifold $M$ equipped with Courant algebroid $E$ consists of a (left) $G$-action on $E$ by Courant automorphisms.  In particular, this covers an action of $G$ on $M$ in the usual sense.

If $M$ has a generalized complex structure $J$ on $E$, then a \emph{generalized complex} $G$-action on $M$ acts on $E$ by generalized complex automorphisms. 
\end{defn}

\begin{rem}
This notion of action in generalized geometry is somewhat different from that in \cite{BursztynCavalcantiGualtieri}.  They require that the group action be generated by the adjoint action of some sections of the Courant algebroid. This difference comes to the fore later in the chapter; we consider families of generalized complex structures which have symmetries in our sense, but for many of which there would be no corresponding \emph{lifted action} in the sense of \cite{BursztynCavalcantiGualtieri}.
\end{rem}

\section{Flat fibre bundles in generalized geometry}\label{flat bundle section}

Most of the conceptual issues in extending the usual notion of flat fibre bundles to generalized complex manifolds are related to the introduction of the Courant algebroid structure; thus, for the sake of clarity we separate the definition of generalized flat fibre bundles (without generalized complex or equivariant structure); then it is straightforward to add these additional structures.

\begin{rem}
We have only defined Courant \emph{isomorphisms}.  The language of general Courant morphisms (see \cite{LiuWeinsteinXu}) is more than we need here, though in our case a definition in those terms would in the end be equivalent to what we give below.
\end{rem} 

\subsection{Generalized flat fibre bundles}

\begin{defn}\label{generalized flat bundle}
A generalized flat fibre bundle consists of the following data: a fibre bundle $\psi : X \to B$ of smooth manifolds, with fibre $F$; exact Courant algebroids $E_F$, $E_X$ and $E_B$ on $F$, $X$ and $B$ respectively; a \emph{generalized horizontal distribution} $S \subset E_X$ such that $E_X = S \dsum S^\perp$.  (We call $S^\perp$ the \emph{generalized vertical distribution.})

Furthermore, we have the data of a map $\Psi : S \to E_B$, covering the projection $\psi$.   Finally, these data should satisfy a condition of \emph{local triviality}:

About any point $b \in B$, there is a neighbourhood $U \subset B$ such that there is an isomorphism of Courant algebroids,
$$E_B|_U \times E_F \;\iso\; (S \dsum S^\perp)|_{\psi^{-1}(U)} \,=\, E_X|_{\psi^{-1}(U)},$$
where this identification matches the respective terms; that is, $S$ is identified with copies of $E_B$, while $S^\perp$ is identified with copies of $E_F$.  We require that, with respect to this identification, $\Psi$ is just the projection of $S$ onto $E_B|_U$.
\end{defn}

\begin{example}
Let $F=S^2$, let $X = \R \times S^2$, let $B = \R$, and let $\psi : X \to B$ be the projection of $X$ to the first factor.  The trivialization of $X \to B$ determines a flat connection.  First, we consider the trivial ``generalization'' of this flat fibre bundle:

Let $E_F = \Tc F$, let $E_X = \Tc X$ and let $E_B = \Tc B$.  The horiztonal and vertical distributions in $TX$ induce decompositions $TX = (TB \times F) \dsum (B \times TF)$ and $T^*X = (T^*B \times F) \dsum (B \times T^*F)$, and thus the decomposition
$$E_X = (\Tc B \times F) \dsum (B \times \Tc F).$$
Let the generalized horizontal distribution be the first summand, and the generalized vertical distribution be the second summand.  With the projection of the first summand to $\Tc B$, this determines a generalized flat fibre bundle, as in the definition above.

Now we consider a further construction: let $\Z$ act on $E_B=\Tc\R$ by pushforward of translation.  Let $\omega$ be a volume form on $F=S^2$, and let $\Z$ act on $E_X$ such that $1$ acts, first by $B$-transform of the vertical component $E_F$ by $\omega$, and then by pushforward of horizontal translation.  These two $\Z$-actions are Courant automorphisms and, since the $B$-transform does not affect the generalized horizontal distribution, they commute with the Courant projection map.

Thus, $E_X/\Z \to E_B/\Z$ is a generalized flat $S^2$ bundle over $S^1$.  Its monodromy is nontrivial in the $B$-transforms, though the underlying (non-generalized) flat fibre bundle is just $S^2 \times S^1$.
\end{example}

We give two kinds of additional structure on generalized flat fibre bundles:
\begin{defn}\label{equivariant generalized complex flat bundle}
A generalized complex flat fibre bundle consists of a generalized flat fibre bundle as defined above, with the additional data of generalized complex structures $J_F$, $J$ and $J_B$ on $E_F$, $E_X$ and $E_B$ respectively, and the additional condition that the product in the local trivializations be a product of generalized complex structures.

Let $G$ be a Lie group.  A \emph{(left)-equivariant generalized complex flat fibre bundle with group $G$} (or \emph{generalized complex flat (left)-$G$-bundle} for short) consists of a generalized complex flat fibre bundle, with the additional data of (left) $G$-actions by generalized complex automorphism on $E_X$ and $E_F$, and the additional condition that the product in the local trivializations respect the group action.
\end{defn}

\begin{rem}\label{inspect local product}
We draw some immediate conclusions from Definition \ref{generalized flat bundle} by inspecting the local trivializations:
\begin{itemize}
\item $\pi(S)$ is complementary to the fibres, while $\pi(S^\perp)$ is tangent to the fibres.
\item $T^*X \cap S$ is the conormal bundle to the fibres.
\item Both $\pi(S)$ and $\pi(S^\perp) \subset TX$ are integrable distributions.
\item $\pi(S)$ is the horizontal distribution for a flat connection for the (non-generalized) fibre bundle $\psi : X \to B$.
\end{itemize}
Our tactic in understanding an isomorphism between generalized flat fibre bundles will be to see how it acts in the local trivializations.  The local trivializations are almost unique, in the following sense:
\end{rem}

\begin{prop}\label{local trivializations unique}
If
$$\tau_1 \;\textnormal{and}\; \tau_2 : E_X|_{\psi^{-1}(U)} \isoto E_B|_U \times E_F$$
are two local trivializations over the same connected neighbourhood $U \subset B$, and if $\tau_1 = \tau_2$ over a single point $x \in U$, then $\tau_1 = \tau_2$.
\end{prop}
\begin{proof}
In the non-generalized geometric context, this is a well-known fact.  Thus, $\tau_1$ and $\tau_2$ cover the same diffeomorphism.  As maps between Courant algebroids, they may differ by a $B$-transform, i.e., there is a closed 2-form $B$ such that
$$\tau_2^{-1}\comp\tau_1 = e^B.$$
Since $\tau_1$ and $\tau_2$ respect the decomposition of $E_X$ into $S \dsum S^\perp$, we may write $B = B_S + B_{S^\perp}$, where the terms vanish on $\pi(S^\perp)$ and $\pi(S)$ respectively.  But $B_S=0$, since $\tau_1|_S$ and $\tau_2|_S$ are determined by compatibility with the projection $\Psi : S \to E_B$.

$B_{S^\perp}$ is closed and vanishes on $\pi(S)$, therefore it is basic with respect to the horizontal foliation, i.e., it is horizontally flat.  But by hypothesis $B_{S^\perp}$ vanishes on some fibre, so it vanishes on all fibres.
\end{proof}

\begin{defn}
We define the \emph{deep bundle},
$$\psi \comp p : S^\perp \to B,$$
with fibre equal to the total space of $E_F$.  The deep bundle has a canonical flat connection coming from the local trivializations.
\end{defn}
Parallel transport by this flat connection respects the Courant algebroid structure, the local trivializations, and the generalized complex and equivariant structures when applicable.

\subsection{Classifications}\label{classifications}

\begin{defn}\label{monodromy}
Let $E_X \to E_B$ be a generalized flat fibre bundle with fibre $F$.  Fix a basepoint $b \in B$, and an identification $\psi^{-1}(b) \iso F$.  Each loop based at $b$ (up to homotopy) induces a Courant automorphism of $E_F$ by parallel transport via the flat connection on the deep bundle.  Then the connection induces a \emph{monodromy map}
$$\lambda : \pi_1(B,b) \to \Aut(E_F),$$
from the fundamental group of $B$ to the automorphism group of $E_F$.  If the bundle is also generalized complex, then $\lambda$ maps into $\Aut(E_F,J_F)$, the \emph{generalized complex} automorphisms of $E_F$.  And if furthermore the bundle is equivariant, then $\lambda$ maps into $\Aut^G(E_F,J_F)$, the \emph{equivariant} generalized complex automorphisms of $E_F$.
\end{defn}

\begin{defn}\label{generalized bundle isomorphism}
Let $\Psi_1 : E_{X_1} \to E_B$ and $\Psi_2 : E_{X_2} \to E_B$ be generalized flat fibre bundles over a common base $B$, with generalized horizontal distributions $S_1$ and $S_2$ respectively.  Then an \emph{isomorphism} between them is an isomorphism of Courant algebroids, $\Phi : E_{X_1} \to E_{X_2}$, such that $\Phi(S_1) = S_2$ and $\Phi\comp\Psi_1 = \Psi_2\comp\Phi$.

As we add more structure, we require isomorphisms to satisfy more conditions.  For example, fix a distinguished common fibre $F$ over $b \in B$, and fix the identification $E_F = S_1^\perp|_F = S_2^\perp|_F$; then with these extra data, an isomorphism $\Phi$ is as above, with the additional condition that $\Phi|_{E_F}$ is the identity.

If $\Psi_i : E_{X_i} \to E_B$ are \emph{generalized complex} flat fibre bundles (as in \ref{equivariant generalized complex flat bundle}), then an isomorphism $\Phi$ should furthermore respect the generalized complex structure; if they are also equivariant, then $\Phi$ should be an equivariant isomorphism.
\end{defn}

As is well-known, (non-generalized) flat bundles with a distinguished fibre are classified by their monodromy (see, eg., \cite{Milnor}).  Applying this standard result to the deep bundle, we will prove the following generalized version:
\begin{prop}\label{classification 1}
Suppose that $B$ is a connected manifold with exact Courant algebroid $E_B$, and let $b \in B$.  Let $F$ be a manifold with exact Courant algebroid $E_F$.

Then, \emph{up to isomorphism}, generalized flat fibre bundles over base $B$, with fibre $F$ over basepoint $b$ (and the corresponding Courant algebroids $E_B$ and $E_F$), are in one-to-one correspondence with the set of their monodromies,
$$\Hom(\pi_1(B,b),\Aut(E_F)).$$

Suppose furthermore that $E_B$ and $E_F$ have generalized complex structures $J_B$ and $J_F$.  Then, up to isomorphism, generalized complex flat fibre bundles over $B$, with fibre $F$ over $b$ (and the corresponding generalized complex structures), are in one-to-one correspondence with
$$\Hom(\pi_1(B,b),\Aut(E_F,J_F)).$$

Finally, suppose furthermore that $E_F$ has a generalized complex $G$-action.  Then, up to isomorphism, generalized complex flat $G$-bundles over $B$, with fibre $F$ over $b$ (and the corresponding generalized complex and equivariant structures), are in one-to-one correspondence with
$$\Hom(\pi_1(B,b),\Aut^G(E_F,J_F)).$$
\end{prop}

\begin{proof}[Proof; uniqueness, given the monodromy.]
Suppose that $X_1 \to B$ and $X_2 \to B$ are two generalized flat fibre bundles over base $B$, with fibre $F$ over basepoint $b$, and suppose furthermore that they have the same monodromy, $\lambda = \lambda_1 = \lambda_2$.  We will show that they are isomorphic.

Since $\lambda$ is just the monodromy for their deep bundles $S_1^\perp \to B$ and $S_2^\perp \to B$, according to the usual classification of flat bundles, there is a flat bundle isomorphism respecting the base and the distinguished fibre, 
$$\Phi|_{S_1^\perp} : S_1^\perp \isoto S_2^\perp.$$

$\Phi|_{S_1^\perp}$ covers a bundle map $\phi : X_1 \to X_2$.  We will specify how $\Phi|_{S_1^\perp}$ extends to a generalized isomorphism
$$\Phi : E_{X_1} \isoto E_{X_2}$$
covering $\phi$.

Let $x \in X_1$, and let $v \in E_{X_1}|_x$.  We see from the local trivializations that the generalized bundle projection map, $\Psi_2$, is a bijection from $S_2|_{\phi(x)}$ to $E_B|_{\psi_2(x)}$.  Thus, the map $\Psi_2^{-1} \comp \Psi_1$ takes the fibre $S_1|_x$ to $S_2|_{\phi(x)}$.

In this way we define $\Phi|_{S_1} : S_1 \isoto S_2$, and thus the whole map
$$\Phi = \Phi|_{S_1} + \Phi|_{S_1^\perp} : E_{X_1} \to E_{X_2}.$$
$\Phi$ respects local trivializations; from this we deduce that it is a Courant isomorphism, and that it is a flat generalized bundle map.

In the latter two cases of the proposition, it follows from the compatibility of $\Phi$ with local trivializations that it will respect the generalized complex and equivariant structures, respectively.
\end{proof}

\begin{proof}[Proof; existence.]
For any given monodromy $\lambda : \pi_1(B,b) \to \Aut(E_F)$, it is easy to exhibit an example bundle.  Let $\tilde{B}$ be the universal cover of $B$, and let $\tilde{X} = \tilde{B} \times F$ with Courant algebroid $E_{\tilde{X}} = E_{\tilde{B}} \times E_F$.  We choose the horizontal bundle $\tilde{S} = E_{\tilde{B}}$.

The monodromy $\lambda$, together with the action of $\pi_1(B,b)$ on $\tilde{X} = \tilde{B} \times F$, induces an action of the discrete group $\pi_1(B,b)$ on $E_{\tilde{X}} = E_{\tilde{B}} \times E_F$, by first applying a deck transformation to $E_{\tilde{B}}$ and then applying the corresponding monodromy to $E_F$.  Then $X := \tilde{X} / \pi(B,b)$ with Courant algebroid $E_X : = E_{\tilde{X}} / \pi(B,b)$ is a generalized flat fibre bundle over $B = \tilde{B} / \pi(B,b)$, with the desired monodromy.

In the latter two cases of the proposition we see that, since the monodromy respects the generalized complex and equivariant structures respectively, so will the deck transformations, and thus these respective structures will pass to the quotient.
\end{proof}

\section{Flat principal bundles from free and proper group actions}\label{free and proper group actions}

To justify the definitions given in the previous section, we should show that the setting mentioned in the introduction is actually an instance of them.  

\begin{defn}
Let $G$ be a Lie group.  We will say that a manifold $X$ with generalized complex structure $J$ on an exact Courant algebroid $E$ has a \emph{transversely symplectic} generalized complex $G$-action if $G$ acts by generalized complex automorphism on $E$, and the symplectic foliation of $X$ is everywhere complementary to the orbits of the induced action on $G$.

We will say this generalized complex action is free and proper if it is free and proper on $E$. 
\end{defn}

\begin{rem}\label{first conclusions}
As is well-known, for (non-generalized) manifolds, if $G$ acts freely and properly on $X$ then the quotient map $\psi : X \to X/G$ is a principal $G$-bundle.  And if a symplectic foliation is complementary to the orbits of $G$, then it induces a flat connection for $\psi : X \to X/G$ in the usual sense.  We note that since the orbits, and hence the complementary symplectic foliation, are of constant rank, the generalized complex structure is regular.
\end{rem}

We wish to define a generalized complex flat $G$-bundle from these data.  Recall that if $\sr{K}$ and $\sr{F}$ are complementary foliations of a manifold $X$, then there is a natural decomposition
$$T^*X = T^*\sr{K} \dsum T^*\sr{F}.$$

\begin{defn}
Suppose $J$ is a generalized complex structure on Courant algebroid $E\to X$, with symplectic foliation $\sr{K}$, and suppose that $G$ acts by transversely symplectic generalized complex action, with orbit foliation $\sr{F}$ on $X$.  Then the \emph{lifted symplectic distribution} is
\begin{equation}\label{define S}
S = J(T^*\sr{K}) \dsum T^*\sr{K} \subset E
\end{equation}
and the \emph{lifted orbit distribution} is $S^\perp$.

If this action is free and proper, then the \emph{generalized complex quotient} is a manifold $B=X/G$ with Courant algebroid $E_B=S/G$ and generalized complex structure $J_B=J|_S/G$.
\end{defn}

\begin{rem}
By definition, $S$ and $S^\perp$ are $J$-invariant, and since $\sr{K}$ is symplectic, we have that  $S\iso T\sr{K}\dsum T^*\sr{K}$; so $S$ is nondegenerate under the pairing and $E=S\dsum S^\perp$.  Then since the $G$-action is generalized complex, the Courant algebroid structure of $E$ passes to $E_B$, and $J_B$ is a generalized complex structure on $E_B$.
\end{rem}

\begin{defn}
If, as above, $G$ acts by a transversely symplectic generalized complex action on $X$, with lifted orbit distribution $S^\perp \subset E$, and if $F\subset X$ is a closed orbit, then the \emph{orbit Courant algebroid} of $F$ is $E_F = S^\perp|_F$, with pairing $\pair{\cdot,\cdot}$ and anchor map $\pi:E_F \to TF$ inherited from $E$.  The fact that $E_F$ also inherits a unique bracket, and that the \emph{orbit generalized complex structure} $J_F = J|_{E_F}$ is generalized complex, is a consequence of the following result. 
\end{defn}

\begin{prop}\label{is a bundle}
Let $X$ be a connected, finite-dimensional manifold with generalized complex structure $J$ on exact Courant algebroid $E$, and let $G$ act freely and properly by transversely symplectic generalized complex action.  Then $X$ is a generalized complex flat $G$-bundle, with generalized horizontal distribution $S$ equal to the lifted symplectic distribution and vertical distribution $S^\perp$, with the generalized complex quotient as base and, for any orbit $F\in X$, the orbit generalized complex structure as the fibre.

Furthermore, the orbit generalized complex structure is of complex type and the quotient generalized complex structure is symplectic (as in Examples \ref{twisted complex} and \ref{symplectic}).
\end{prop}
There is an analytic claim hiding in this result.  In particular, if the generalized complex structure were not integrable, then we should not expect the remaining hypotheses even to give a generalized flat fibre bundle.  However, we accomplish the necessary analysis by resorting to the local normal form theorem, Proposition \ref{normal form}.

\begin{proof}
As noted, $X\to B=X/G$ is a flat principal $G$-bundle in the non-generalized sense.  We construct local trivializations for the Courant algebroids $E\to E_B$, first locally in $X$, then locally near a fibre.

Let $F\subset X$ be a fibre (i.e., an orbit) and let $x \in F$.  (Since $X$ is connected, any choice of fibre will do.)  Since $J$ is regular, by Proposition \ref{normal form} there is a neighbourhood $U \subset X$ of $x$ on which $J$ is isomorphic to the standard structure on a rectangle in $\C^k \times \R^{2n}$.  Since the orbit foliation is complementary to the symplectic leaves and, by hypothesis, preserves the symplectic structure, by the augmented form of the theorem (Proposition \ref{augmented normal form}), we may choose this product structure so that the orbits map into copies of $\C^k$.

Thus we have a generalized complex trivialization,
$$\tilde\tau:E|_U\to\Tc\R^{2n} \times \Tc\C^k,$$
for a smaller bundle on $U$.  If we repeat the construction of the lifted symplectic distribution $S$, but now on $\Tc\R^{2n} \times \Tc\C^k$, we find that $S = \Tc\R^{2n} \times (\C^k,0) $ and $S^\perp = (\R^{2n},0) \times \Tc\C^k$.  Pulling this decomposition back to $E|_U$, we endow the orbit Courant algebroid $E_F|_U$ with the Courant algebroid structure coming from $\Tc\C^k$.  For this Courant algebroid structure, the orbit structure $J_F$ is generalized complex.

By the identification of $E_F|_U$ with a subset of $\Tc\C^k$ and of $E_B|_U$ with a subset of $\Tc\R^{2n}$, we have the trivialization
$$\tau:E|_U \to E_B|_{U/G} \times E_F|_{U\cap F}.$$

Now we consider the same construction for a neighbourhood $W$ of some other $y \in F$.  Let
$$\tau':E|_W \to E_B|_{W/G} \times E_F|_{W \cap F}$$
be the generalized complex trivialization on $W$.  Note that $\tau$ and $\tau'$ will agree on $E_U|_{U \cap W \cap F}$; thus, if $U \cap W$ is connected and contains points in $F$, then by Proposition \ref{local trivializations unique}, $\tau$ and $\tau'$ agree on $E_U|_{U \cap W}$.  Thus, by taking local trivializations near every point in $F$, we may construct a generalized complex local trivialization on a neighbourhood of $F$, for which $S$ is the generalized horizontal distribution and $S^\perp$ is the generalized vertical distribution.

Since the local trivializations were all Courant isomorphic, the algebroid structure thus defined on $E_F$ is unique.  It is clear from the local trivalizations that $J_F$ is complex and $J_B$ is symplectic.

Finally, it is straightforward that the $G$-action respects the local trivializations: the local trivializations were uniquely determined by the generalized complex geometry and the orbit foliation, which are invariant under the action.
\end{proof}

\begin{rem}
It is tempting to try to apply the classification result (Proposition \ref{classification 1}) to classify free and proper, transversely symplectic generalized complex $G$-manifolds---given data about the orbits, the quotient and the monodromy.  However, the setup in Proposition \ref{classification 1} fixes a base and a distinguished fibre \emph{a priori}.  As is well-known, there is some redundancy in this data, from the point of view of classifying $G$-manifolds.  There will be a nontrivial equivalence relation on the classifying data for bundles, which will tell us whether they are isomorphic as $G$-manifolds.  When the quotient is a surface, this is not hard to compute: see, for example, the classification of symplectic torus bundles in \cite{Pelayo} for the general idea.

We do not pursue this in detail.  Regardless, the classifying data (see Section \ref{classifications}) for generalized complex flat principal bundles, with symplectic base and complex fibre, do exhaust all possible free and proper, transversely symplectic generalized complex $G$-manifolds.
\end{rem}

\section{Principal torus bundles}\label{torus bundle section}

We would like to further probe the case of Section \ref{free and proper group actions} when the group $G$ is a torus.

\subsection{Twisted complex tori}

To begin, we enumerate invariant twisted complex structures on tori.  (For more general work along these lines, see \cite{AlekseevskyDavid}.)  A twisted complex structure induces a usual complex structure, and any invariant complex structure on a torus, $G$, may be determined by identifying the torus with $\R^{2n} / L \iso \C^n / L$, where $\Z^{2n} \iso L \subset \R^{2n}$ is some lattice, and then taking the standard complex structure, $I$, on $\R^{2n} \iso \C^{n}$.  Let $\gcg = \R^{2n} \iso \C^{n}$ be its Lie algebra.

\subsubsection*{Splitting and curvature}\label{splitting and curvature}
Suppose that $J_G$ is a generalized complex structure on a Courant algebroid $E_G$ on $G$ such that $J_G$ induces the complex structure $I$, and suppose that the self-action of $G$ lifts to a generalized complex action $\rho : G \to \Aut^G(E_G,J_G)$.

$J_G$ restricts to a linear generalized complex structure on the fibre $E_G|_e$ at the identity.  We choose a splitting, $E_G|_e \iso \Tc_e G = \gcg \dsum \gcg^*$, for which $J_G$ is the standard structure,
$$\left[\begin{array}{cc}
-I & 0 \\
0 & I^*
\end{array}\right].
$$
Then, by the correspondence between $\gcg$ and tangent spaces of $G$, we get a splitting, $E_G \iso TG \dsum T^*G$, for the whole algebroid.  Thus, without loss of generality, we suppose that $J_G$ is the standard generalized complex structure induced on $\Tc_H G$ by $I$.

This splitting may have a nontrivial curvature 3-form, $H$.  Since $H$ will be $G$-invariant, by the same correspondence as above, $H$ corresponds to a $3$-covector $\gch \in \^ ^3 \gcg^*$.  In fact, any $\gch \in \^ ^3 \gcg^*$ determines a closed, invariant 3-form.  (Similarly as above, if $G$ were non-Abelian, there would be a closedness condition on $\gch$.)  Recall that if $J_G$ is integrable then $\gch$ is real of bidegree $(1,2)+(2,1)$. Thus,

\begin{prop}\label{count tori}
An equivariant generalized complex torus of complex type and real dimension $2n$ is determined by a choice of
\begin{itemize}
\item lattice $L \subset \R^{2n},\;$ ($L \iso \Z^{2n}$), and
\item real curvature $\gch \in \Re\left(\^ ^{1,2} \gcg^* \dsum \^ ^{2,1} \gcg^*\right),\;$ ($\gcg = \R^{2n} \iso \C^n$).
\end{itemize}
\end{prop}

\begin{rem}
The real dimension of $\Re\left(\^ ^{1,2} \gcg^* \dsum \^ ^{2,1} \gcg^*\right)$ is $n^2(n-1)$, since $$\dim_\R(\^ ^{1,2} \gcg^*) = 2 {{n}\choose{1}} {{n}\choose{2}}$$ and reality fixes a choice in $\^ ^{2,1} \gcg^*$.
\end{rem}

\subsection{Automorphisms of the torus}

\begin{prop}\label{count automorphisms}
If $G$ is a torus with invariant generalized complex structure $J_G$ of complex type on an equivariant Courant algebroid $E_G$, then
$$\Aut^G(E_G,J_G) \iso G \times \Re(\^ ^{1,1} \gcg),$$
(where $\Re$ indicates the real subspace).
\end{prop}

\begin{proof}
As in subsection \ref{splitting and curvature}, without loss of generality, we suppose that $J_G$ is the standard generalized complex structure determined by a complex structure $I$ on $\Tc_H G$.  A Courant automorphism of $\Tc_H G$ may be decomposed into a $B$-transform and a pushforward by diffeomorphism.  If the diffeomorphism is equivariant, it is just the action of some $g \in G$, and is automatically generalized complex.  The $B$-transform corresponds to an invariant $(1,1)$-form.  (Since $B$ is invariant, it is automatically closed, and its action commutes with that of $G$.)
\end{proof}

\begin{rem}
If $\gcg = \R^{2n}$, then the real dimension of $\Re(\^ ^{1,1} \gcg)$ is $n^2$.
\end{rem}

\subsection{Summary: classifying data for bundles}

\begin{rem}
In the case of torus bundles, the monodromy maps, $\lambda : \pi_1(B,b) \to \Aut^G(E_G,J_G)$, are mapping into an Abelian group, as per Proposition \ref{count automorphisms}.  Thus, they are determined by their Abelianization.  That is, we may equivalently consider maps
$$\lambda : \pi_1^{\textrm{ab}}(B,b) \iso H_1(B) \to \Aut^G(E_G,J_G),$$
where $H_1(B)$ is the first homology of the base.  (See \cite{Pelayo} for more details.)
\end{rem}

As per Propositions \ref{classification 1}, \ref{is a bundle}, \ref{count tori} and \ref{count automorphisms}, and the above remark, generalized complex principal torus bundles whose symplectic leaves are complementary to the fibres are classified as follows:
\begin{enumerate}
\item Choose a symplectic base manifold, $B$, with basepoint $b$.
\item Choose a lattice $L \subset \R^{2n},\;$ ($L \iso \Z^{2n}$).  This determines a torus $G = \R^{2n}/L$ with invariant complex structure.
\item Choose a real curvature $\gch \in \Re\left(\^ ^{2,1} \gcg^* \dsum \^ ^{1,2} \gcg^*\right),\;$ ($\gcg = \R^{2n} \iso \C^n$).  This determines an invariant generalized complex structure $J_G$ on $G$.
\item Then generalized complex principal $G$-bundles over base $B$, with distinguished generalized complex fibre $G$ over basepoint $b$, are (up to isomorphism) in one-to-one correspondence with the set of monodromies,
\begin{eqnarray*}
&& \Hom(H_1(B),\Aut^G(E_G,J_G)) \\
&\iso& \Hom(H_1(B),G \times \Re(\^ ^{1,1} \gcg)) \\
&\iso& \Hom(H_1(B),G \times \R^{n^2})
\end{eqnarray*}
\end{enumerate}

For example, if $B$ is a compact surface of genus $k$, then $H_1(B) \iso \Z^{2k}$, and
\begin{eqnarray*}
&& \Hom(H_1(B),\Aut^G(E_G,J_G)) \\
&\iso& \Hom(\Z^{2k},G \times \R^{n^2}) \\
&\iso& G^{2k} \times \left(\R^{n^2}\right)^{2k}
\end{eqnarray*}
The second term represents degrees of freedom which do not occur in the case of non-generalized bundles, neither in the case of ``lifted actions'' in the sense of \cite{BursztynCavalcantiGualtieri}.

%\input{biblio}

%\end{document}
% ----------------------------------------------------------------

%% This adds a line for the Bibliography in the Table of Contents.
\addcontentsline{toc}{chapter}{Bibliography}
%% ***   Set the bibliography style.   ***
%% (change according to your preference)
\bibliographystyle{plain}
%% ***   Set the bibliography file.   ***
%% ("thesis.bib" by default; change if needed)
%%\bibliography{thesis}

%% ***   NOTE   ***
%% If you don't use bibliography files, comment out the previous line
%% and use \begin{thebibliography}...\end{thebibliography}.  (In that
%% case, you should probably put the bibliography in a separate file
%% and `\include' or `\input' it here).

\end{document}